\documentclass[10pt,a4paper]{article}

\usepackage[a4paper, left=31mm, right=31mm]{geometry}
\usepackage[T1]{fontenc}							
\usepackage[utf8]{inputenc}							
\usepackage{amsmath, amsfonts, amssymb, amsthm, empheq}		
\usepackage{graphicx}								
\usepackage{microtype}								
\usepackage[hidelinks]{hyperref}								
\usepackage{cleveref}								
\usepackage{booktabs}								
\usepackage{todonotes}
\usepackage{float}
\usepackage{caption}
\usepackage{subcaption}						
\usepackage{siunitx}
\usepackage{array} 

\newcommand*\samethanks[1][\value{footnote}]{\footnotemark[#1]} 
\newcommand{\IR}{\mathbb{R}}  									
\newcommand{\IO}{\Omega}  										

\definecolor{darkblue}{rgb}{0,0,.45}
\definecolor{darkblue2}{rgb}{0,0,.65}
\definecolor{darkred}{rgb}{0.6,0,0}
\definecolor{darkred2}{rgb}{0.85,0,0}
\definecolor{darkgreen}{rgb}{0,0.5,0}
\definecolor{darkgreen2}{rgb}{0,0.75,0}
\definecolor{darkgray}{rgb}{0.35,0.35,0.35}

\definecolor{mygray}{rgb}{.8, .8, .8}

\def\N{\mathbb{N}}

\def\R{\mathbb{R}}

\def\T{\mathcal{T}}
\def\E{\mathcal{E}}

\newcounter{hypocounter}
\renewcommand\thehypocounter{(H\arabic{hypocounter})} 
\newcommand{\overbar}[1]{\mkern 1.5mu\overline{\mkern-1.5mu#1\mkern-1.5mu}\mkern 1.5mu}

\theoremstyle{definition}								
\newtheorem{theorem}{Theorem}
\newtheorem{lemma}[theorem]{Lemma}					

\newtheorem{corollary}[theorem]{Corollary}
\newtheorem{remark}[theorem]{Remark}

\numberwithin{equation}{section}
\numberwithin{figure}{section}
\numberwithin{theorem}{section}
\title{Numerical analysis of a finite volume scheme for charge transport in perovskite solar cells}

\author{Dilara Abdel\thanks{Weierstrass Institute (WIAS), Mohrenstr. 39, 10117 Berlin, Germany}
    \and Claire Chainais-Hillairet\thanks{Univ. Lille, CNRS, Inria, UMR 8524 - Laboratoire Paul Painlevé, F-59000 Lille, France}
    \and Patricio Farrell\samethanks[1]
    \and Maxime Herda\thanks{Inria, Univ. Lille, CNRS, UMR 8524 - Laboratoire Paul Painlevé, F-59000 Lille, France}
}

\date{\today}

\begin{document}

    \maketitle
    \begin{abstract}
        In this paper, we consider a drift-diffusion charge transport model for perovskite solar cells, where electrons and holes may diffuse linearly (Boltzmann approximation) or nonlinearly (e.g.\ due to Fermi-Dirac statistics). To incorporate volume exclusion effects, we rely on the Fermi-Dirac integral of order $-1$ when modeling moving anionic vacancies within the perovskite layer which is sandwiched between electron and hole transport layers. After non-dimensionalization, we first prove a continuous entropy-dissipation inequality for the model. Then, we formulate a corresponding two-point flux finite volume scheme on Voronoi meshes and show an analogous discrete entropy-dissipation inequality. This inequality helps us to show the existence of a discrete solution of the nonlinear discrete system with the help of a corollary of Brouwer's fixed point theorem and the minimization of a convex functional.
        Finally, we verify our theoretically proven properties numerically, simulate a realistic device setup and show exponential decay in time with respect to the $L^2$ error as well as a physically and analytically meaningful relative entropy.
    \end{abstract}

    \tableofcontents
    \section{Introduction}

    In recent years, Perovskite Solar Cells (PSCs) have become one of the fastest growing photovoltaic technologies \cite{kim2020high, Tessler2020}. Two advantages of PSCs stand out: On the one hand, certain architectures have significantly lower production costs than conventional solar cells. On the other hand, silicon-perovskite tandem cells have become more efficient than classical single-junction silicon solar cells. However, the commercialization of PSCs is still in its early stages and several challenges need to be overcome, the most prominent being the relatively fast degradation of these devices. Apart from solar cells, perovskite materials also show promise for use in LEDs, photodetectors and memristors.

    There is well-established mathematical literature concerning drift-diffusion mathematical models to describe charge transport in classical or organic semiconductors and similar physical systems (see for instance and  non-exhaustively \cite{gajewski85, gajewski86, markowich1985stationary, markowich2012semiconductor, mock1983analysis, van1950theory}). In the perovskite material, a major difference is that apart from electrons and holes, ion migration plays a fundamental role. Therefore, to correctly reflect the physical behavior, perovskite models must contain additional freely moving ion species. In PSCs, the various charge carrier species live in different parts of the domain, evolve on different time scales and obey different diffusion laws. Moreover, one must take into account  the photogeneration effect. Because of the differences and new features compared to classical semiconductors, there is a need for new mathematical and numerical  modeling and analysis for these devices. Several models have been proposed recently in the literature -- with the exception of \cite{Abdel2021}, nearly all of them are in one dimension \cite{Calado2016,Courtier2018}.

    In this paper, we consider a three-layer drift-diffusion perovskite charge transport model for three different charge carriers. Within the entire device electrons and holes may diffuse linearly (Boltzmann approximation) or nonlinearly. The nonlinear diffusion may be governed, for example, by Fermi-Dirac statistics \cite{gajewski1989semiconductor} but our model allows even more general statistical relationships between densities and potentials. In the middle perovskite region, we include anion vacancies. It has been demonstrated that the ion migration despite being considerably slow cannot be neglected \cite{Calado2016}. To correctly incorporate volume exclusion effects, we rely on Fermi-Dirac statistics of order $-1$. All three drift-diffusion equations are coupled self-consistently to a nonlinear Poisson equation.

    Given that perovskite models are relatively new, the arising Partial Differential Equations (PDE) model has not been  studied yet mathematically. In comparison, classical drift-diffusion models have been studied in detail \cite{gajewski85, gajewski86, markowich1985stationary, markowich2012semiconductor, mock1983analysis}. For drift-diffusion models, an essential a priori estimate is based on the evolution of the physical \textit{free energy} or a related functional. It allows to study the well-posedness of the equations as well as the asymptotic behavior of its solution \cite{biler2000long, gajewski85, mock1983analysis}. The techniques relying on a well-chosen physically relevant Lyapunov functional have been used for many systems of dissipative PDEs and are usually referred to as entropy methods (see \cite{jungel2016entropy} and references therein), which is why in the following we will use the term \emph{entropy} instead of \emph{free energy}.

    The design of numerical schemes for drift-diffusion models is also an mature but still very active field of research (see for instance \cite{brezzi2005discretization, Cances2020, jungel1995numerical, Kantner2020, Kantner2020b, moatti2022structure,mock1983analysis, SG69, Yu1988}). In order to ensure the quality of the numerical simulation and the stability of the numerical method, efforts have been made towards the design of structure preserving schemes \cite{BessemoulinChatard2017,ChainaisHillairet2019, glitzky2011uniform, liu2021positivity, moatti2022structure}. Their aim is to preserve physical features of the original model such as the decay of free-energy or non-negativity of solutions. Because of the stiffness in drift-diffusion models arising from small parameters such as the Debye length, fully implicit in time numerical methods are usually preferred. They yield robustness of the scheme as well as asymptotic preserving properties \cite{bessemoulin14,ChainaisHillairet2019}. Finally, the treatment of nonlinear diffusion to handle general statistics function has also been investigated \cite{bessemoulin2012finite, farrell2018comparison, jungel1995numerical}.

    The main goal of this paper is the analysis of an implicit in time two-point flux approximation (TPFA) finite volume scheme for the perovskite model. The model and the scheme originate from \cite{Abdel2021}. The scheme relies on the the excess chemical potential flux scheme which appears to be used for the first time in \cite{Yu1988} and was later numerically analyzed in \cite{Cances2020, Gaudeul2021} and compared in \cite{Abdel2021b, Kantner2020}.

    The main tool for our analysis is an entropy-dissipation inequality for the perovskite model. After non-dimensionalizing the perovskite model, we establish such an inequality in the continuous setting in Theorem~\ref{thm:cont-E-D}. Then, we adapt the arguments to show in Theorem~\ref{thm:discrete-E-D} that the discrete counterpart of this inequality also holds for a solution of the implicit finite volume scheme. This \emph{a priori} estimate on the scheme allows us to prove the existence of a discrete solution at each time step in Theorem~\ref{thm.exresult}. The proof relies on a corollary of Brouwer's fixed point theorem for the quasi Fermi potentials, coupled with the minimization of a convex functional for the electric potential.

    We illustrate and complement our theoretical results with numerical experiments. We investigate the convergence in space and witness second order accuracy, as expected. By introducing a relative free energy with respect to the steady solution, we illustrate the long time behavior of transient solutions and their convergence towards steady state solutions exponentially fast in time. Finally, we investigate the large time behavior of the perovskite model at a constant applied voltage which physically corresponds to investigating the influence of preconditioning a PSC before current-voltage measurements.

    The remainder of the paper is organized as follows: In Section \ref{sec:model}, we introduce the original and the non-dimensionalized perovskite model for general statistics functions and the underlying free energy. In Section \ref{sec:cont_entropy_dissipation}, we then prove a continuous entropy-dissipation inequality for the perovskite model. In Section \ref{sec:FVM}, we present the corresponding finite volume scheme, for which we prove a discrete entropy-dissipation inequality. This inequality will allow us to deduce the existence of a discrete solution in Section \ref{sec:existence}. Finally, in Section \ref{sec:numerics}, we corroborate and complement our theoretical observations numerically before we conclude and discuss future research in Section \ref{sec:conclusion}.

    \section{Charge transport model for perovskite solar cells}
    \label{sec:model}
    Let $\mathbf{\IO} \subseteq \IR^d, d \leq 3$, be an open, connected and bounded spatial domain, which is partitioned into three pairwise disjoint, open subdomains
    $
    \mathbf{\IO} = \left( \cup_{k } \overline{\mathbf{\IO}}_{k}\right)^{\circ},
    $
    $k \in \{ \text{intr}, \text{HTL}, \text{ETL}\}$.
    Here, $\mathbf{\IO}_{\text{intr}}$ refers to the intrinsic perovskite region and $\mathbf{\IO}_{\text{HTL}}, \mathbf{\IO}
    _{\text{ETL}}$ to the doped electron and hole transport layers, respectively. We denote the interface between the transport layers and the perovskite layer by $\mathbf{\Sigma}_{\text{HTL}} = \partial \mathbf{\Omega}_{\text{HTL}} \cap \partial \mathbf{\Omega}_{\text{intr}}$ and $\mathbf{\Sigma}_{\text{ETL}} = \partial \mathbf{\Omega}_{\text{ETL}} \cap \partial \mathbf{\Omega}_{\text{intr}}$. The boundaries of the hole and electron transport layer do not intersect,  $\partial \mathbf{\Omega}_{\text{HTL}} \cap \partial \mathbf{\Omega}_{\text{ETL}} = \emptyset$, see \Cref{fig:2D-architecture} for a potential device geometry. The unknowns of the charge transport model are given by the electric potential $\psi(\mathbf{x}, t)$ and the quasi Fermi potentials (frequently called electrochemical potentials) of moving charge carriers, denoted by $\varphi_\alpha(\mathbf{x}, t)$, $\alpha \in \{ n, p, a\}$.
    Here, the quantities $n$, $p$ and $a$ refer to the electrons, holes and anion vacancies.
    Unlike quasi Fermi potentials for electrons and holes $\varphi_n, \varphi_p$, which are defined for any $\mathbf{x}\in\mathbf{\IO}$, the quasi Fermi potential of anion vacancies $\varphi_a$ is defined only in the intrinsic domain $\mathbf{\IO}_{\text{intr}}$. The density of a charge carrier is denoted by $n_\alpha$, $\alpha \in \{ n, p, a\}$. We examine the model for the charge transport in PSCs formulated in \cite{Abdel2021}, where for $t\geq 0$ the mass balances are given by
    \begin{subequations} \label{eq:model}
        \begin{align}
            z_n q \partial_t n_n + \nabla\cdot \mathbf{j}_n &=   z_n q \Bigl(G(\mathbf{x})-R(n_n,n_p)\Bigr),&&{\bf x}\in\mathbf{\Omega},\ t\geq 0, \label{eq:model-cont-eq-n}\\
            z_p q \partial_t n_p + \nabla\cdot \mathbf{j}_p &=  z_p q \Bigl(G(\mathbf{x})-R(n_n,n_p) \Bigr),&&{\bf x}\in\mathbf{\Omega},\ t\geq 0, \label{eq:model-cont-eq-p}\\
            z_a q \partial_t n_a + \nabla\cdot \mathbf{j}_a &= 0,&&{\bf x}\in\mathbf{\Omega}_\text{intr},\ t\geq 0, \label{eq:model-cont-eq-a}
        \end{align}
    \end{subequations}
    self-consistently coupled with the nonlinear and region-wise defined Poisson equation
    \begin{equation}\label{eq:model-poisson}
        - \nabla \cdot (\varepsilon_s \nabla \psi ) =\left\{{\begin{aligned}
                & q \Bigl( z_nn_n + z_pn_p + C({\bf x})\Bigr),   &&{\bf x}\in{\mathbf{\Omega}}_{\text{HTL}}\cup{\mathbf{\Omega}}_{\text{ETL}},\ t\geq 0,\\
                &q \Bigl( z_nn_n + z_pn_p + z_an_a + C({\bf x})\Bigr),  &&{\bf x}\in{\mathbf{\Omega}}_{\text{intr}},\ t\geq 0.
            \end{aligned}}\right.
    \end{equation}
    The charge numbers of the three moving charge carriers are given by $z_n$, $z_p$ and $z_a$, the elementary charge is denoted by $q$ and $\varepsilon_s$ refers to the region-dependent dielectric permittivity.  Throughout this paper, we assume the standard charge numbers $z_n=-1$ and $z_p=1$ for electrons and holes. For the anionic vacancies, we will simply assume $z_a>0$ and later for the numerical experiments we set $z_a=1$. Note that all of the following computations can be extended to include movement of \textit{cation} vacancies, i.e.\ movement of negative ion vacancy charge numbers. However, this case seems to be of less physical interest for PSCs. The charge carrier densities are linked to the set of unknowns $(\psi, \varphi_n, \varphi_p, \varphi_a)$ by the state equations \cite{Farrell2017}
    \begin{equation} \label{eq:state-eq}
        n_\alpha = N_\alpha \mathcal{F}_\alpha \Bigl(\eta_\alpha(\varphi_\alpha, \psi) \Bigr), \quad \eta_\alpha = z_\alpha \frac{q(\varphi_\alpha - \psi) + E_\alpha}{k_B T}, \quad \alpha \in \{ n, p, a\},
    \end{equation}
    where $\mathcal{F}_\alpha$ is called \textit{statistics function} which will be discussed in \Cref{sec:statistics}. The quantities $N_n, N_p$ denote the effective conduction and valence density of states, whereas $N_a$ is given by the maximal vacancy concentration. The argument  $\eta_\alpha$ of the statistics function is called chemical potential. It depends on the band-edge energies $E_\alpha$, a constant temperature $T$ and the Boltzmann constant $k_B$. Moreover, the electric currents $\mathbf{j}_\alpha$ for each species are given by
    \begin{equation}
        \mathbf{j}_\alpha =
        - z_\alpha^2 \mu_\alpha
        n_\alpha
        \nabla\varphi_\alpha, \quad \alpha \in \{ n, p, a\}, \label{eq:cont-flux}
    \end{equation}
    with the carrier mobility $\mu_\alpha$. Concerning the right-hand side of the continuity equations \eqref{eq:model-cont-eq-n}, \eqref{eq:model-cont-eq-p} and Poisson equation \eqref{eq:model-poisson} we assume that the doping profile $C$ is bounded, i.e.\ $C \in L^{\infty}(\mathbf{\Omega})$  and that the photogeneration rate satisfies $0 \leq G \in L^{\infty}(\mathbf{\Omega})$. In other words the carrier dependent doping profile and the photogeneration rate are constant in time. It is common to assume a Beer-Lambert generation profile, describing an exponential decay in the $z$ direction, see \Cref{fig:2D-architecture},
    \begin{equation*} \label{eq:generation}
	G(\mathbf{x}) = F_{\text{ph}} \alpha_g \exp(-\alpha_g z ), \quad \mathbf{x}=(x,y,z)^T,
	\end{equation*}
    where $F_{\text{ph}}$ denotes the incident photon flux and  $\alpha_g$ the material absorption coefficient.
    Lastly, the recombination rate $R$ is of the form 
    \cite{Farrell2017}
    \begin{align*}
        R(n_n, n_p) = r(n_n, n_p) n_nn_p \Bigl(1- \exp \left( \varphi_n - \varphi_p \right)  \Bigr), \quad r(n_n, n_p) = \sum_r r_r(n_n, n_p),
    \end{align*}
    where $r_r(n_n, n_p)\geq 0$ is given by the sum of all present recombination processes, which, for PSCs, are radiative and trap-assisted Shockley-Read-Hall recombination
    \begin{align*}
        r_{\text{rad}} = r_{0} \quad  \text{and} \quad r_{\text{SRH}} = \frac{1}{\tau_p (n_n + n_{n, \tau}) + \tau_p (n_p + n_{p, \tau} ) },
    \end{align*}
    where $r_{0}$ is a constant rate coefficient, $\tau_n, \tau_p$ are the carrier life times and $n_{n, \tau}, n_{p, \tau}$ the reference carrier densities.
    Furthermore, we supply the system \eqref{eq:model} with initial conditions for $t = 0$
    \begin{subequations}  \label{eq:initial-cond}
        \begin{align}
            \varphi_n(\mathbf{x}, 0) = \varphi_n^0(\mathbf{x}), \quad &\quad\varphi_p(\mathbf{x}, 0) = \varphi_p^0(\mathbf{x}) \quad &\text{for\;} \mathbf{x} \in \boldsymbol{\Omega},\hphantom{xx}\\
            \varphi_a(\mathbf{x}, 0) &= \varphi_a^0(\mathbf{x}), \quad &\text{for\;} \mathbf{x} \in \boldsymbol{\Omega}_{\text{intr}},
        \end{align}
    \end{subequations}
    where we assume $\varphi_n^0, \varphi_p^0 \in L^{\infty}(\mathbf{\Omega})$ and  $\varphi_a^0 \in L^{\infty}(\mathbf{\Omega}_{\text{intr}})$.
    Correspondingly, we define initial densities $n_\alpha^0(\mathbf{x}) = N_\alpha\mathcal{F}_\alpha(\eta_\alpha(\varphi_\alpha^0,\psi(\mathbf{x},0)))$.
    \subsection{Boundary conditions} \label{sec:BC}
    The outer boundary of $\mathbf{\Omega}$ is decomposed into two ohmic contacts modeled by Dirichlet conditions $\mathbf{\Gamma}^D$ and an isolated interface $\mathbf{\Gamma}^N$, where we impose no flux Neumann boundary conditions. We assume $\mathbf{\Gamma}^D \cap \overline{\mathbf{\Omega}}_{\text{intr}} = \emptyset $, i.e.\ the ohmic contacts are solely located at the outer boundary of transport layers. More precisely, let the Dirichlet values $\psi^D, \varphi^D \in W^{1, \infty}(\mathbf{\Omega})$ be given. Then, the outer boundary conditions are modeled via
    \begin{subequations}\label{eq:outer-bc}
        \begin{alignat}{2}
            \psi(\mathbf{x}, t) = \psi^D(\mathbf{x}), \;\; \varphi_n(\mathbf{x}, t)  = \varphi_p(\mathbf{x}, t) &= \varphi^D(\mathbf{x}), \quad
            ~
            &&\mathbf{x} \in \mathbf{\Gamma}^D,  \ t \geq 0,  \label{eq:Dirichlet-cond} \\[0.5ex]
            ~
            ~
            \nabla \psi(\mathbf{x}, t) \cdot \boldsymbol{\nu}(\mathbf{x}) = \mathbf{j}_n(\mathbf{x}, t) \cdot \boldsymbol{\nu}(\mathbf{x}) = \mathbf{j}_p(\mathbf{x}, t) \cdot \boldsymbol{\nu}(\mathbf{x}) &= 0, \quad
            ~
            &&\mathbf{x} \in \mathbf{\Gamma}^N,  \ t \geq 0,\label{eq:Neumann-cond}
        \end{alignat}
    \end{subequations}
    where $\boldsymbol{\nu}$ is the outward pointing unit normal to $\mathbf{\Gamma}^N$. Note that the same Dirichlet value $\varphi^D$ is imposed on both quasi Fermi potentials.
    Concerning the anion vacancies, we impose no flux Neumann boundary conditions on the whole intrinsic boundary, namely
    \begin{equation}\label{eq:outer-bc_anion}
        \mathbf{j}_{a}(\mathbf{x}, t) \cdot \boldsymbol{\nu}_\text{intr}(\mathbf{x}) = 0, \quad
        \mathbf{x} \in \partial\mathbf{\Omega}_{\text{intr}},  \ t \geq 0,
    \end{equation}
    where $\boldsymbol{\nu}_\text{intr}$ is the outward pointing unit normal to $ \partial\mathbf{\Omega}_{\text{intr}}$.
    \begin{figure}[H]
        \centering
        \includegraphics[scale = 0.65]{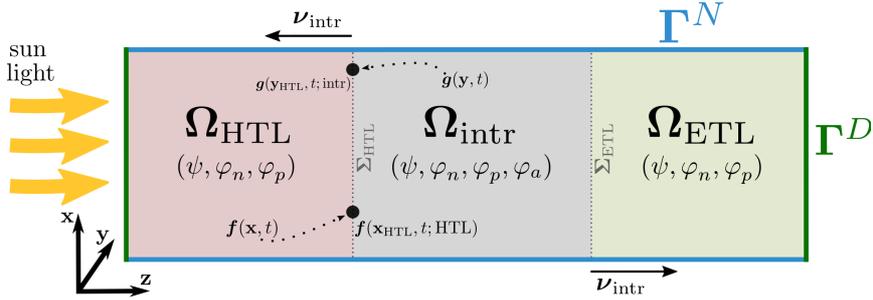}
        \caption{A two-dimensional three-layer device with the relevant potentials stated per subdomain.}
        \label{fig:2D-architecture}
    \end{figure}
    The traces of the potentials $\psi$, $\varphi_n$ and $\varphi_p$ coincide on both sides of the internal boundaries $\mathbf{\Sigma}_{\text{HTL}}$ and  $\mathbf{\Sigma}_{\text{ETL}}$. Moreover, the corresponding fluxes are also continuous across internal boundaries. More precisely, for $t\geq 0$
    \begin{subequations} \label{eq:inner-bc}
        \begin{align}
            \Bigl( \varepsilon_s \nabla \psi(\mathbf{x},t; k )  - \varepsilon_s \nabla \psi(\mathbf{x},t; \text{intr} ) \Bigr)  \cdot \boldsymbol{\nu}_\text{intr}(\mathbf{x})&= 0, \quad \mathbf{x}\in\mathbf{\Sigma}_k,\ k \in \{ \text{HTL}, \text{ETL}\}, \label{eq:inner-bc-psi} \\
            ~
            \Bigl(  \mathbf{j}_n(\mathbf{x},t; k )  - \mathbf{j}_n(\mathbf{x},t; \text{intr} ) \Bigr)  \cdot \boldsymbol{\nu}_\text{intr}(\mathbf{x})&= 0, \quad \mathbf{x}\in\mathbf{\Sigma}_k,\ k \in \{ \text{HTL}, \text{ETL}\}, \label{eq:inner-bc-n}\\
            ~
            \Bigl(  \mathbf{j}_p(\mathbf{x},t; k )  - \mathbf{j}_p(\mathbf{x},t; \text{intr} ) \Bigr)  \cdot \boldsymbol{\nu}_\text{intr}(\mathbf{x})&= 0, \quad\mathbf{x}\in\mathbf{\Sigma}_k,\ k \in \{ \text{HTL}, \text{ETL}\}. \label{eq:inner-bc-p}
        \end{align}
    \end{subequations}
    Here, we use the notation that for any function $f$ the expression $f(\mathbf{x},t; k)$ denotes the trace of $f$, restricted onto $\mathbf{\Omega}_k$, $k \in \{ \text{HTL},\text{ETL}, \text{intr}\}$, evaluated at the respective interface between transport and perovskite layer.

    \begin{remark}[Conservation of mass for anion vacancies]
        Observe that by integrating \eqref{eq:model-cont-eq-a} over $\mathbf{\Omega}_\text{intr}$ and using the Neumann boundary conditions \eqref{eq:outer-bc_anion} the total mass of anionic vacancies is conserved, namely
        \[
        \int_{\mathbf{\Omega}_\text{intr}}n_a(\mathbf{x}, t)\, d\mathbf{x} = \int_{\mathbf{\Omega}_\text{intr}}n_a^0(\mathbf{x})\, d\mathbf{x},\quad \text{for all}\; t\geq0.
        \]
        An equivalent condition does not hold for electron and hole densities due to the boundary conditions and the recombination/generation terms.
    \end{remark}

    \subsection{Statistics functions} \label{sec:statistics}
    Lastly, we need to discuss the choice of statistics functions $\mathcal{F}_\alpha$ in \eqref{eq:state-eq} depending on the charge carrier species $\alpha\in\{n,p,a\}$. Our results will hold under general abstract assumptions on these functions. However, we also provide specific examples below.
    \paragraph{Electric Charge Carriers $\boldsymbol{\alpha \in\{n,p\}}$.}
    For electrons and holes $\alpha \in\{n,p\}$ we assume
    \begin{equation}\tag{H1} \label{hyp:statistics-n-p}
        \left\{
        \begin{aligned}
            &\mathcal{F}_n, \mathcal{F}_p: \mathbb{R} \rightarrow (0,\infty)\text{ are }  C^1\text{- diffeomorphisms};\\[.5em]
            &0 < \mathcal{F}_\alpha'(\eta) \leq \mathcal{F}_\alpha(\eta) \leq \exp(\eta), \quad\eta \in \mathbb{R},\ \alpha \in\{n,p\} .
        \end{aligned}
        \right.
    \end{equation}
    There are two important statistics functions, satisfying assumption~\eqref{hyp:statistics-n-p}, which are commonly used for modeling electric charge transport in PSCs. The first one is given by the \textit{Fermi-Dirac integral of order $\mathit{1/2}$} defined as
    \begin{equation}
        F_{1/2}(\eta) =  \frac{2}{ \sqrt{\pi} } \int_{0}^{\infty} \frac{\xi^{1/2}}{\exp(\xi - \eta) + 1}\: \mathrm{d}\xi, \label{eq:FermiDirac}
    \end{equation}
    which is fundamental in the simulation of inorganic three-dimensional semiconductors \cite{Farrell2017, Sze2006}. The function $F_{1/2}$ behaves like $\eta^{3/2}$ when the chemical potential $\eta$ tends to $+\infty$, namely in the large density limit. In the low density limit, when the chemical potential $\eta$ tends to $-\infty$, it behaves like the Boltzmann statistics function
    \begin{equation}
        F_{B}(\eta) = \exp(\eta), \label{eq:Boltzmann}
    \end{equation}
    which is another important statistics functions for electrons and holes. Observe that the choice  $\mathcal{F}_n = \mathcal{F}_p = F_{1/2}$ leads to nonlinear diffusion in the electric currents \eqref{eq:cont-flux}, whereas $\mathcal{F}_n = \mathcal{F}_p = F_{B}$  yields linear diffusion.
    \paragraph{Ionic Charge Carriers $\boldsymbol{\alpha = a}$.}
   We assume that the statistics function for ionic charge carriers satisfies the following assumption
    \begin{equation}\tag{H2} \label{hyp:statistics-a}
        \left\{
        \begin{aligned}
            &\mathcal{F}_a: \mathbb{R} \rightarrow (0,1)\text{ is a }  C^1\text{- diffeomorphism};\\[.5em]
            &0 < \mathcal{F}_a'(\eta) \leq \mathcal{F}_a(\eta) \leq \exp(\eta), \quad\eta \in \mathbb{R}.
        \end{aligned}
        \right.
    \end{equation}
    Observe that the boundedness of the image of $\mathcal{F}_a$ reflects the boundedness of the anion vacancy density. Such a choice necessarily leads to nonlinear diffusion for the densities in \eqref{eq:model-cont-eq-n}. For anion vacancies in PSCs the \textit{Fermi-Dirac integral of order $\mathit{-1}$} is chosen to reflect the limitation of ion concentration by the lattice sites available in the crystal \cite{Abdel2021}. This particular statistics function reads
    \begin{align}
        F_{-1}(\eta) = \frac{1}{ \exp(-\eta) + 1}. \label{eq:Fermi--1}
    \end{align}
    Note that for both assumptions \eqref{hyp:statistics-n-p} and \eqref{hyp:statistics-a}, the positivity of the statistics functions reflects the positivity of the number densities of charge carriers.
    \subsection{Thermodynamic free energy}
    \label{ssec:free_energy}
    The thermodynamic free energy for the discussed model is given by the sum of different energy contributions. Following \cite{Albinus2002, Kantner2020b}, on the one hand, the contribution of electrons and holes can be derived from a quasi-free Fermi gas. On the other hand, the electric contribution to the total energy is given by the electrostatic field energy. Lastly, assuming an ideal lattice gas \cite{Abdel2021} we can derive a consistent energy contribution of anion vacancies which extends the electric free energy formulation in \cite{Kantner2020b}. Hence, in total the free energy functional for the PSC model reads
    \begin{align*}
        \mathbb{E}_{f}(t) = &\frac12 \int_{\mathbf{\Omega}} \varepsilon_s|\nabla \psi|^2\,d\mathbf{x}
        ~ + \sum_{\alpha \in \{n, p  \} } \int_{\mathbf{\Omega}} \left[ k_B T N_\alpha\Phi_\alpha\left(\frac{n_\alpha}{N_\alpha}\right) - z_\alpha E_\alpha n_\alpha \right] \,d\mathbf{x}
        \\
        ~
        ~& +  \int_{\mathbf{\Omega}_{\text{intr}}} \left[ k_B TN_a\Phi_a\left(\frac{n_a}{N_a}\right) - z_aE_a n_a \right] \,d\mathbf{x},
    \end{align*}
    where $\Phi_\alpha$ is an antiderivative of $\mathcal{F}_\alpha^{-1}$ and where for the sake of simplicity we neglected external interaction effects of the electric potential. Thus, for non-degenerate semiconductors, i.e.\ for $\mathcal{F}_n = \mathcal{F}_p = \exp$ and $\mathcal{F}_a$ chosen as Fermi-Dirac integral of order $-1$ the free energy simplifies to
    \begin{equation} \label{eq:physical-free-energy}
        \begin{split}
            \mathbb{E}_f(t) = &\frac12 \int_{\mathbf{\Omega}} \varepsilon_s|\nabla \psi|^2\,d\mathbf{x}
        ~
        + \sum_{\alpha \in \{ n, p \}} \int_{\mathbf{\Omega}} \left[  k_B T n_\alpha \left(  \log\left(\frac{n_\alpha}{N_\alpha} \right)  - 1 \right) - z_\alpha E_\alpha n_\alpha \right] \,d\mathbf{x}\\
        ~
        &+ \int_{\mathbf{\Omega}_{\text{intr}}}  \left[ k_B T  \left(  n_\alpha \log \left( \frac{n_\alpha}{N_\alpha}\right) + \left( N_\alpha - n_\alpha \right) \log \left( 1- \frac{n_\alpha}{N_\alpha} \right) \right)- z_aE_a n_a \right] \,d\mathbf{x}.
        ~
        \end{split}
    \end{equation}
    \subsection{Non-dimensionalization of the model} \label{sec:nondimens-model}
   In this subsection, we derive the relevant non-dimensional parameters of the model, following \cite{Courtier2018} and \cite[Section 2.4]{markowich1985stationary}. Starting from the charge transport model in \eqref{eq:model}-\eqref{eq:cont-flux}, we rewrite the equations in terms of the scaled variables given as the ratio of the unscaled physical quantity and the scaling factors defined in \Cref{tab:scalingfactors} (where $U_T = k_B T / q$ shall denote the thermal voltage).

   In the following we make several simplifications in order to simplify the presentation and the forthcoming computations.
   More precisely, we assume from now on and until the end of Section~\ref{sec:existence} that the mobilities $\mu_n$ and $\mu_p$, the dielelectric permittivity $\varepsilon_s$, and the effective conduction and valence density of states $N_n$ and $N_p$ are constant in the domain $\mathbf{\Omega}$. Moreover, we assume that $\mu_n = \mu_p$ and $N_n=N_p$. In practice, the previous quantities vary in each subdomain. Finally, the band-edge energy $E_\alpha$ is assumed to be null for all moving charge carriers $\alpha$. In Section~\ref{sec:pscsetup}, we perform numerical simulations with heterogeneous parameters and non-zero band-edge energies. All the analysis of Section~\ref{sec:cont_entropy_dissipation} and Section~\ref{sec:FVM} can be adapted without the previous simplifications. However, apart from creating notational overhead, the key ideas remain the same.
   \begin{table}[h!]
   \centering
   \begin{tabular}{|c|l|c|l|}
        \hline Symbol & Meaning & Scaling factor & Order of magnitude \\\hline&&&\\
        $\:\:\mathbf{x}$ & space variable & $l$ & $10^{-5}\,$cm\\[.5em]
        ~
        $\psi$, $\psi^D$, $\varphi_\alpha$, $\varphi^D$& electric and quasi Fermi potentials & $U_T$ & $10^{-2}\,\text{V}$\\[.75em]
        ~
        $n_n$,  $n_p$ & densities of electrons and holes & $\tilde{N}$ & $10^{18}\,$cm$^{-3}$\\[.5em]
        ~
        $n_a$ & density of anion vacancies & $\tilde{N}_a$ & $10^{21}\,$cm$^{-3}$\\[.5em]
        ~
        $C$ & doping profile & $\tilde{N}$ & $10^{18}\,$cm$^{-3}$\\[.5em]
        ~
        $\mu_n,\mu_p$ & electron and hole mobility & $\tilde{\mu}$ & $10^0\,$cm$^2$V$^{-1}$s$^{-1}$\\[.5em]
        ~
        $\mu_a$ & anion vacancy mobility & $\tilde{\mu}_a$&$10^{-12}\,$cm$^2$V$^{-1}$s$^{-1}$\\[.5em]
        ~
        $t$ & time variable & $\displaystyle\frac{l^2}{\tilde{\mu}_aU_T}$& $10^{4}\,$s\\[.75em]
        ~
        $\mathbf{j}_n, \mathbf{j}_p$ & current density for electrons and holes & $\displaystyle\frac{q U_T \tilde{N}\tilde{\mu}}{l}$&  $10^2\,$Acm$^{-2}$\\[.75em]
        ~
        $\mathbf{j}_a$ & current density for anion vacancies & $\displaystyle\frac{q U_T \tilde{N}_a\tilde{\mu}_a}{l}$&  $10^{-7}\,$Acm$^{-2}$\\[.75em]
        ~
        $R$ & recombination rate & $\displaystyle\frac{\tilde{\mu}U_T\tilde{N}}{l^2}$&$10^{26}\,$cm$^{-3}$s$^{-1}$\\[.75em]
        ~
        $G$ & photogeneration rate & $F_\text{ph}\alpha_g$ &  $10^{22}\,$cm$^{-3}$s$^{-1}$\\[.25em]
        \hline
   \end{tabular}
   \caption{Scaling factors for a PSC device based on the default parameters of \cite{Courtier2019} at $T=298K$.}
   \label{tab:scalingfactors}
    \end{table}

    In Table~\ref{tab:scalingfactors}, the time scale is chosen to be that of the anion vacancies. By replacing $\tilde{\mu}_a$ with $\tilde{\mu}$ in the scaling factor of the time variable, one could write the dimensionless version adapted to the electrons and holes time scale. We assume that the scaling factor $\tilde{N}$ is exactly equal to $N_n=N_p$ and that $\tilde{N}_a = N_a$. Similarly, since we assumed, for simplicity, that the mobilities are constant in the domain, we take $\tilde{\mu}_a = \mu_a$ and $\tilde{\mu} = \mu_n = \mu_p$.
    By denoting the scaled quantities with the same symbol as the corresponding unscaled quantities the dimensionless version of the model  \eqref{eq:model}-\eqref{eq:cont-flux} reads
    \begin{subequations} \label{eq:model-dimless}
        \begin{align}
            \nu\,z_n \partial_t n_n + \nabla\cdot \mathbf{j}_n &=   z_n \Bigl(\gamma\,G(\mathbf{x})-R(n_n,n_p)\Bigr),&&{\bf x}\in\mathbf{\Omega},\ t\geq 0, \label{eq:model-cont-eq-n-dimless}\\
            \nu\,z_p \partial_t n_p + \nabla\cdot \mathbf{j}_p &=  z_p \Bigl(\gamma\,G(\mathbf{x})-R(n_n,n_p) \Bigr),&&{\bf x}\in\mathbf{\Omega},\ t\geq 0, \label{eq:model-cont-eq-p-dimless}\\
            z_a \partial_t n_a + \nabla\cdot \mathbf{j}_a &= 0,&&{\bf x}\in\mathbf{\Omega}_\text{intr},\ t\geq 0, \label{eq:model-cont-eq-a-dimless}
        \end{align}
    \end{subequations}
    coupled to the Poisson equation
    \begin{equation}\label{eq:model-poisson-dimless}
         - \lambda^2\Delta\psi =\left\{{\begin{aligned}
            & \delta( z_nn_n + z_pn_p + C({\bf x})),   &&{\bf x}\in{\mathbf{\Omega}}_{\text{HTL}}\cup{\mathbf{\Omega}}_{\text{ETL}},\ t\geq 0,\\
            &z_an_a + \delta(z_nn_n + z_pn_p  + C({\bf x})),  &&{\bf x}\in{\mathbf{\Omega}}_{\text{intr}},\ t\geq 0.
             \end{aligned}}\right.
    \end{equation}
    The state equation can be rewritten as
    \begin{equation} \label{eq:state-eq-dimless}
         n_\alpha = \mathcal{F}_\alpha \Bigl(z_\alpha(\varphi_\alpha - \psi) \Bigr), \quad \alpha \in \{ n, p, a\},
    \end{equation}
    and we have the following expressions for the charge carrier currents
    \begin{align} \label{eq:cont-flux-dimless}
     \mathbf{j}_\alpha = - z_\alpha^2  n_\alpha \nabla \varphi_\alpha, \quad \alpha \in \{ n, p, a\}.
    \end{align}
    There are four dimensionless parameters, the rescaled Debye length, which is taken with respect to the anion vacancies
    \begin{equation}\label{eq:debye}
    \lambda = \sqrt{\frac{\varepsilon_sU_T}{l^2q\tilde{N}_a}},
    \end{equation}
    the relative mobility of anion vacancies with respect to the mobility of electrons and holes
    \begin{equation}\label{eq:ratiomobilities}
    \nu = \frac{\tilde{\mu}_a}{\tilde{\mu}},
    \end{equation}
    the relative concentration of electric carriers with respect to the anion vacancy concentration
    \begin{equation}\label{eq:ratiodensities}
    \delta = \frac{\tilde{N}}{\tilde{N}_a},
    \end{equation}
    and the rescaled photogeneration rate
    \begin{equation}\label{eq:dimlessphotogen}
    \gamma = \frac{F_\text{ph}\alpha_gl^2}{\tilde{\mu}U_T\tilde{N}}.
    \end{equation}
    The parameter $\nu$ can also be interpreted as the ratio between the electric and ionic carrier time scale. In a typical device all of these parameters are small. More precisely, $\nu\approx 10^{-12}$, $\lambda\approx 10^{-2}$, $\delta\approx10^{-3}$ and $\gamma\approx 10^{-5}$. In particular, the parameters $\nu$ and $\lambda^2$ generate important stiffness in the model, which motivates the use of a robust implicit-in-time numerical scheme (see Section~\ref{sec:FVM}).

    \section{Continuous entropy-dissipation inequality}
    \label{sec:cont_entropy_dissipation}
    For drift-diffusion systems in semiconductor modeling, the natural \emph{a priori} estimate \cite{gajewski85, gajewski86} is based on the evolution of a global quantity which has the physical meaning of a free energy. In the following, we call this quantity \emph{(relative) entropy},  in the sense of \emph{entropy method} for PDEs rather than in the physical sense.

    \subsection{Entropy functions}\label{sec:entropy}
     For $\alpha\in\{n,p,a\}$, we define the relative entropy function $\Phi_\alpha$, associated with the statistics $\mathcal{F}_\alpha$, to be an anti-derivative of the inverse statistics function namely
    \begin{equation}\label{eq:entropyfunc}
        \Phi_\alpha'(x) = \mathcal{F}_\alpha^{-1}(x), \quad x \geq 0.
    \end{equation}
    Observe that \eqref{hyp:statistics-n-p} and \eqref{hyp:statistics-a} imply that the statistics function is strictly increasing and therefore $\Phi_\alpha$ is strictly convex. Of course equation \eqref{eq:entropyfunc} does not define $\Phi_\alpha$ uniquely, but the value of the constant is not crucial for $\alpha=n,p$ in what follows because we will introduce relative entropies. The constant may be taken in general to ensure that $\Phi_\alpha$ is non-negative and vanishes at only one point, which is indeed necessary for $\alpha=a$.

    We also define the relative entropy $H_\alpha$ by
    \begin{align} \label{eq:H-function}
        H_\alpha(x,y) = \Phi_\alpha(x) - \Phi_\alpha(y) - \Phi_\alpha'(y)(x-y), \quad x \geq 0.
    \end{align}
    Observe that $H_\alpha$ is non-negative due to the convexity of $\Phi_\alpha$.

    \paragraph{Examples.} Let us give two examples for the typical statistics functions of the electric and ionic charge carriers. In the case of the Boltzmann statistics, one has
    \[
    \mathcal{F}_n(\eta) = \mathcal{F}_p(\eta) = e^\eta,\qquad \Phi_n(x) = \Phi_p(x) = x\log(x) -x +1.
    \]
    In the case of the Fermi-Dirac integral of order $-1$ for $\mathcal{F}_a$, one has
    \[
    \mathcal{F}_a(\eta) = \frac{1}{ \exp(-\eta) + 1},\qquad \Phi_a(x) = x\log(x) +(1-x)\log(1-x)+\log(2).
    \]
    Note that both examples for the mathematical entropy functions coincide with the respective physical free energy contributions in \eqref{eq:physical-free-energy}.
  
    \paragraph{Properties of the entropy and relative entropy function.}
    Let us state some useful results for the entropy functions. The proofs may be found in \Cref{app:estimates-statistics}.
    \begin{lemma}\label{lemma:bound-primitive-argument}
        One has the following bounds on the entropy functions \eqref{eq:entropyfunc} and \eqref{eq:H-function}.
        \begin{itemize}
            \item[(i)] Let $\mathcal{F}_\alpha$ be a statistics function satisfying \eqref{hyp:statistics-n-p} and $H_\alpha$ be the associated relative entropy function. Then, for any  $\varepsilon > 0$ and $y_0\geq 0$, there exists a constant $c_{y_0,\varepsilon}>0$ such that
            \begin{align*}
                x \leq c_{y_0, \varepsilon} + \varepsilon H_\alpha(x,y), \quad  \text{for all}\; x\geq0,\  y\in[0,y_0].
            \end{align*}
            \item[(ii)] Let $\mathcal{F}_a$ be a statistics function satisfying \eqref{hyp:statistics-a} and $\Phi_a$ be the associated entropy function. Then, for any  $\varepsilon > 0$, there exists a constant $c_{\varepsilon}>0$ such that
            \begin{align} \label{eq:fM1-legendre-bound}
                x \leq c_\varepsilon + \varepsilon \Phi_a(x), \quad \text{for all}\; x\geq 0.
            \end{align}
        \end{itemize}
    \end{lemma}
    Under a last assumption on the statistics functions for electrons and holes
    \begin{equation}\tag{H3} \label{hyp:limit-H-finverse}
        \lim_{x \rightarrow +\infty}\frac{ H_\alpha(x, y_0)}{\mathcal{F}_\alpha^{-1} (x)} = +\infty, \quad \text{for} \; y_0\geq 0 \quad \text{and} \quad \alpha = n,p,
    \end{equation}
    we have the following result.
    \begin{lemma} \label{lemma:relation-inverse-primitive}
        Let $\mathcal{F}_\alpha$ with $\alpha = n, p$ be a statistics function satisfying \eqref{hyp:statistics-n-p} and \eqref{hyp:limit-H-finverse}. Then, for any  $\varepsilon > 0$ and $y_0\geq 0$, there exists a constant $c_{y_0,\varepsilon}>0$ such that
        \begin{align*}
            \max(\mathcal{F}_\alpha^{-1} (x),0) \leq c_{y_0, \varepsilon} + \varepsilon H_\alpha(x, y) ,\quad  \text{for all}\; x\geq0,\  y\in[0,y_0].
        \end{align*}
    \end{lemma}

    We will also show in \Cref{app:estimates-statistics} that the Boltzmann statistics and the Fermi-Dirac statistics of order $1/2$ both satisfy \eqref{hyp:statistics-n-p}, \eqref{hyp:limit-H-finverse}, while the Fermi-Dirac statistics of order $-1$ satisfies \eqref{hyp:statistics-a}.
    \subsection{Proof of the entropy-dissipation inequality}
    The thermodynamic free energy introduced in Subsection \ref{ssec:free_energy} is of physical interest. Now, however, we would like to prove an  entropy-dissipation inequality whose discrete counterpart will allow us to prove the existence of a discrete solution and the stability of the scheme. For this reason, we introduce a variation of this functional which from now we will refer to as \textit{total relative entropy} in agreement with the mathematical literature.
    Adapting the functional of \cite{bessemoulin14, jungel01} to our system, the total relative entropy with respect to the Dirichlet boundary values $\psi^D, \varphi^D$ is given by
    \begin{equation}\label{eq:continuous-entropy}
        \mathbb{E}(t) = \frac{\lambda^2}{2}\int_{\mathbf{\Omega}} |\nabla (\psi - \psi^D)|^2\,d\mathbf{x}
        + \int_{\mathbf{\Omega}_{\text{intr}}} \Phi_a(n_a)\,d\mathbf{x}
        + \delta\sum_{\alpha \in \{n, p\}}
        \int_{\mathbf{\Omega}} H_\alpha(n_\alpha, n_\alpha^D)  \,d\mathbf{x},
    \end{equation}
    where the entropy functions $\Phi_a$, $H_n$ and $H_p$ are given by \eqref{eq:entropyfunc}, \eqref{eq:H-function} and $n_\alpha^D$ can be calculated by inserting $\varphi^D, \psi^D$ into the state equation  \eqref{eq:state-eq}.
    Note that due to our specific choice for $\Phi_a$ the middle term is
    non-negative as well which implies
    that the entropy is non-negative.
    Taking into account the fact that $z_\alpha^2=1$ for $\alpha\in\{n,p\}$, the associated non-negative dissipation $\mathbb{D}$ is defined as
        \begin{align} \label{eq:continuous-dissipation}
            \mathbb{D}(t) = \frac{\delta}{\nu}\int_\mathbf{\Omega} R(n_n, n_p) \left( \varphi_p - \varphi_n \right)\,d\mathbf{x} + \frac{z_a^2}{2}\int_\mathbf{\Omega_\text{intr}} n_a \vert \nabla \varphi_a\vert^2  \,d\mathbf{x} + \frac{\delta}{2\nu}\sum_{\alpha \in \{n, p\}}  \int_\mathbf{\Omega} n_\alpha \vert \nabla \varphi_\alpha\vert^2  \,d\mathbf{x}.
        \end{align}
    \begin{theorem}(\textbf{Continuous entropy-dissipation inequality}) \label{thm:cont-E-D}
        Consider a smooth solution to the model \eqref{eq:model-dimless}--\eqref{eq:cont-flux-dimless}, with initial conditions \eqref{eq:initial-cond} and boundary conditions \eqref{eq:outer-bc}, \eqref{eq:outer-bc_anion}, \eqref{eq:inner-bc}. Then, for any $\varepsilon>0$, there is a constant $c_{\varepsilon, \mathbf{\Omega}}>0$
        \begin{align} \label{eq:continuous-E-D-inequality}
            \frac{\text{d}}{\text{d}t}\mathbb{E}(t)  +  \mathbb{D}(t) \leq c_{\varepsilon, \mathbf{\Omega}} + \varepsilon \mathbb{E}(t),
        \end{align}
        where the entropy is defined in \eqref{eq:continuous-entropy} and the dissipation of entropy in  \eqref{eq:continuous-dissipation}. The constant $c_{\varepsilon, \mathbf{\Omega}}$ depends only on $\varepsilon$, $\mathbf{\Omega}$, on the boundary data and the photogeneration term via the norms $\|G\|_{L^\infty}, \|\varphi^D\|_{W^{1,\infty}}$ and $\|\psi^D\|_{W^{1,\infty}}$, on $z_a$ and on the dimensionless parameters $\delta$, $\gamma$ and $\nu$.
    \end{theorem}
    \begin{remark}[Thermodynamic equilibrium] \label{rem:thermodynamic-eq}
        If the boundary data is at thermodynamic equilibrium, \emph{i.e.} $\nabla\varphi^D = \nabla\psi^D = \mathbf{0} $ and without external generation of electric carriers, \emph{i.e.} $G=0$, then the entropy-dissipation inequality simplifies to
        \[
        \frac{\text{d}}{\text{d}t}\mathbb{E}(t)  +  \mathbb{D}(t)\leq 0.
        \]
        Indeed, while we do not specify precisely the dependencies of the constant on the data it is clear from the proof of Theorem~\ref{thm:cont-E-D} that the right hand-side of \eqref{eq:continuous-E-D-inequality} vanishes in this setting. In this case the entropy decays in time and the solution is expected to converge exponentially fast towards the thermodynamic equilibrium $(\varphi_n^\text{eq},\varphi_p^\text{eq}, \varphi_a^\text{eq}, \psi^\text{eq})$. This thermodynamic equilibrium is such that the quasi Fermi potentials for electrons and holes are constant on $\mathbf{\Omega}$
        \[
        \varphi_p^\text{eq} = \varphi_n^\text{eq} = \varphi^D
        \]
        and $\varphi_a^\text{eq}$ is constant on $\mathbf{\Omega}_{\textrm{intr}}$, determined by the conservation of mass for anion vacancies
        \[
         \int_{\mathbf{\Omega}_{\text{intr}}} n_a(\varphi_a^\text{eq}, \psi^\text{eq} ) \,d\mathbf{x} = \int_{\mathbf{\Omega}_{\text{intr}}}n_a(\mathbf{x}, 0)\,d\mathbf{x},
        \]
        where the electric potential $\psi^\text{eq}$ satisfies the following nonlinear Poisson equation
        \[
        - \lambda^2\Delta \psi^\text{eq} =\left\{{\begin{aligned}
                &\delta\left(n_p(\varphi_p^\text{eq}, \psi^\text{eq} ) -n_n(\varphi_n^\text{eq}, \psi^\text{eq} )  + C({\bf x})\right),&&{\bf x}\in{\mathbf{\Omega}}_{\text{HTL}}\cup{\mathbf{\Omega}}_{\text{ETL}},\\
                ~
                &  z_a n_a(\varphi_a^\text{eq}, \psi^\text{eq} ) + \delta\left(n_p(\varphi_p^\text{eq}, \psi^\text{eq} ) - n_n(\varphi_n^\text{eq}, \psi^\text{eq} )  + C({\bf x})\right),&&{\bf x}\in{\mathbf{\Omega}}_{\text{intr}}.\
            \end{aligned}}\right.
        \]
    The system is supplemented with the Dirichlet and Neumann boundary conditions  \eqref{eq:Dirichlet-cond} and \eqref{eq:Neumann-cond} for the electric potential. The proof of this asymptotic behavior is beyond the scope of the present paper but could be investigated following the lines of the seminal work of Gajewski \cite{gajewski85}.
    \end{remark}
    \begin{proof}[Proof of Theorem~\ref{thm:cont-E-D}]
        First, let us take the derivative of \eqref{eq:continuous-entropy} with respect to time
        \begin{equation} \label{eq:Dirichlet-dt_E}
            \begin{aligned}
                \frac{\text{d}}{\text{d}t}\mathbb{E}(t) &= \lambda^2\int_{\mathbf{\Omega}}  (\partial_t \nabla \psi) \cdot \nabla (\psi - \psi^D) \,d\mathbf{x}
                ~
                +  \int_{\mathbf{\Omega}_{\text{intr}}} \mathcal{F}_a^{-1}(n_a)\partial_t n_a\,d\mathbf{x}\\
                ~
                &+\delta\sum_{\alpha \in \{n, p\}} \int_\mathbf{\Omega} \left( \mathcal{F}_\alpha^{-1}(n_\alpha ) -\mathcal{F}_\alpha^{-1}(n_\alpha^D )\right)\partial_t n_\alpha \,d\mathbf{x}.
            \end{aligned}
        \end{equation}
        By integrating the first term by parts and using the Poisson equation \eqref{eq:model-poisson} one obtains
        \[
        \lambda^2\int_{\mathbf{\Omega}}  (\partial_t \nabla \psi) \cdot \nabla (\psi - \psi^D) \,d\mathbf{x}\ =\ \delta\int_{\mathbf{\Omega}}  (z_n\partial_t n_n + z_p\partial_t n_p) (\psi - \psi^D) \,d\mathbf{x}
        +\int_{\mathbf{\Omega}_\text{intr}}  z_a\partial_tn_a (\psi - \psi^D) \,d\mathbf{x},
        \]
        where all the boundary terms cancel thanks to the boundary conditions \eqref{eq:outer-bc} and \eqref{eq:inner-bc-psi}. Plugging this back into \eqref{eq:Dirichlet-dt_E} and using the state equation \eqref{eq:state-eq-dimless}, we have
        \[
        \begin{aligned}
            \frac{\text{d}}{\text{d}t}\mathbb{E}(t) =   \int_{\mathbf{\Omega}_{\text{intr}}} z_a \left( \varphi_a - \psi^D \right)\partial_t n_a\,d\mathbf{x}\!\!
            ~
            +\delta\!\!\sum_{\alpha \in \{n, p\}} \int_\mathbf{\Omega} z_\alpha \left( \varphi_\alpha - \varphi^D \right)\partial_t n_\alpha \,d\mathbf{x}.
        \end{aligned}
        \]
        Next, we insert the balance equations \eqref{eq:model-dimless} and the definition of the current densities \eqref{eq:cont-flux-dimless}
        \begin{align*}
            \frac{\text{d}}{\text{d}t}\mathbb{E}(t) =&
            -  \int_{\mathbf{\Omega}_{\text{intr}}} \nabla \cdot \mathbf{j}_{a}  \left( \varphi_a - \psi^D \right)\,d\mathbf{x}
            -\frac{\delta}{\nu}\sum_{\alpha \in \{n, p\}} \int_\mathbf{\Omega} \nabla \cdot \mathbf{j}_\alpha \left( \varphi_\alpha - \varphi^D \right) \,d\mathbf{x}\\&
            +\frac{\delta}{\nu}\sum_{\alpha \in \{n, p\}} \int_\mathbf{\Omega} z_\alpha \left(\gamma G - R \right) \left( \varphi_\alpha - \varphi^D \right) \,d\mathbf{x}\\
            ~
            =& -  \int_{\mathbf{\Omega}_{\text{intr}}} z_a^2n_a\nabla\varphi_a \cdot \nabla \left( \varphi_a - \psi^D \right)\,d\mathbf{x}
            -\frac{\delta}{\nu}\sum_{\alpha \in \{n, p\}} \int_\mathbf{\Omega} n_\alpha\nabla\varphi_\alpha \cdot \nabla\left( \varphi_\alpha - \varphi^D \right) \,d\mathbf{x}\\&
            +\frac{\delta}{\nu}\int_\mathbf{\Omega}  \left(\gamma G - R \right) (\varphi_p-\varphi_n) \,d\mathbf{x},
        \end{align*}
        where we used $z_n = -1=- z_p$ and  integrated by parts with boundary terms vanishing again thanks to \eqref{eq:Neumann-cond}, \eqref{eq:outer-bc_anion}, \eqref{eq:inner-bc-n} and \eqref{eq:inner-bc-p}. By expanding the first terms and using Young's inequality we get
        \begin{equation}\label{eq:entropy-dissip-with-remainders}
            \frac{\text{d}}{\text{d}t}\mathbb{E}(t) + \mathbb{D}(t) \leq \frac{\delta\gamma}{\nu} \int_\mathbf{\Omega}  G   \left( \varphi_p - \varphi_n \right) \,d\mathbf{x}
            + \frac{z_a^2}{2 } \int_{\mathbf{\Omega}_\text{intr}} n_a  | \nabla \psi^D|^2  \,d\mathbf{x}
            + \frac{\delta}{2\nu} \sum_{\alpha \in \{n, p\}}  \int_\mathbf{\Omega} n_\alpha  | \nabla \varphi^D|^2  \,d\mathbf{x}.
        \end{equation}
        It remains to bound the terms of the right-hand side. For the first term on the right hand side of \eqref{eq:entropy-dissip-with-remainders} we use the state equation \eqref{eq:state-eq-dimless} and \Cref{lemma:relation-inverse-primitive} to find for some $\varepsilon >0$
        \begin{align*}
            \frac{\delta\gamma}{\nu}   \int_\mathbf{\Omega}  G   \left( \varphi_p - \varphi_n \right) \,d\mathbf{x} &=  \frac{\delta\gamma}{\nu}    \sum_{\alpha \in \{ n, p\}} \int_\mathbf{\Omega} G \mathcal{F}^{-1}_\alpha(n_\alpha)\,d\mathbf{x}
            ~
            \leq  \frac{\delta\gamma}{\nu}    || G ||_{\infty } \!\!\! \sum_{\alpha \in \{ n, p\}} \int_\mathbf{\Omega}  \max(\mathcal{F}^{-1}_\alpha(n_\alpha),0) \,d\mathbf{x} \\
            ~
            ~
            &\leq  \frac{\delta\gamma}{\nu}     || G ||_{\infty } \sum_{\alpha \in \{ n, p\}}  \left(c_{y_{\alpha}^D, \varepsilon} | \mathbf{\Omega}| + \varepsilon\int_\mathbf{\Omega} H_\alpha \left(n_\alpha, n_\alpha^D \right) \,d\mathbf{x} \right),
        \end{align*}
        where  $H_\alpha$ is defined in \eqref{eq:H-function} and  $c_{y_\alpha^D, \varepsilon}$ is the corresponding constant introduced in \Cref{lemma:relation-inverse-primitive}, where for any species $\alpha\in\{n,p\}$ we introduce
        \[
         y_{\alpha}^D = {\mathcal F}_{\alpha}(\Vert \varphi^D\Vert_{\infty}+ \Vert \psi^D\Vert_{\infty}).
        \]
        With help of \Cref{lemma:bound-primitive-argument} the second remainder term of \eqref{eq:entropy-dissip-with-remainders} is estimated by
        \begin{align*}
            &\frac{z_a^2}{2}\int_{\mathbf{\Omega}_\text{intr}} n_a  | \nabla \psi^D|^2  \,d\mathbf{x}
            + \frac{\delta}{2 \nu}\sum_{\alpha \in \{n, p\}}  \int_\mathbf{\Omega} n_\alpha  | \nabla \varphi^D|^2  \,d\mathbf{x}\\
            ~
            \leq&
            ~
            \frac{z_a^2}{2} || \nabla \psi^D ||_{\infty}^2\int_{\mathbf{\Omega}_\text{intr}} n_a \,d\mathbf{x} + \frac{\delta}{2 \nu}|| \nabla \varphi^D ||_{\infty}^2\sum_{\alpha \in \{n, p\}}  \int_{\mathbf{\Omega}} n_\alpha  \,d\mathbf{x}\\
            ~
            \leq&
            ~
            \max \left\{\frac{z_a^2}{2}  || \nabla \psi^D ||_{\infty}^2, \frac{\delta}{2 \nu}|| \nabla \varphi^D ||_{\infty}^2\right\} \left((c_{y_{n}^D,\varepsilon} + c_{y_{p}^D,\varepsilon} +c_{\varepsilon})|\mathbf{\Omega}| + 3\varepsilon \mathbb{E}   \right)
        \end{align*}
        since the first term in \eqref{eq:continuous-entropy} is non-negative.
        Plugging these estimates back into \eqref{eq:entropy-dissip-with-remainders} proves the entropy-dissipation estimate (up to a redefinition of $\varepsilon$).
    \end{proof}
    Using Grönwall's lemma, an immediate consequence of Theorem~\ref{thm:cont-E-D} is that, as functions of time, the entropy $t\mapsto\mathbb{E}(t)$ and the dissipation $t\mapsto\mathbb{D}(t)$ are respectively locally bounded and locally integrable. More precisely, one has the following result.
    \begin{corollary}\label{cor:bounds-E-D}
        For any $\varepsilon>0$, one has
        \[
        \mathbb{E}(t) + \int_0^t\mathbb{D}(s)\mathrm{d}s\leq e^{\varepsilon t}\mathbb{E}(0) + \frac{c_{\varepsilon, \mathbf{\Omega}}}{\varepsilon}(e^{\varepsilon t}-1)\,,\quad t\geq0.
        \]
    \end{corollary}
   
    \section{Discrete version of charge transport model}
    \label{sec:FVM}
    In this section, we introduce our numerical scheme for \eqref{eq:model-dimless}-\eqref{eq:cont-flux-dimless}. It is a finite volume scheme with a two-point flux approximation of the fluxes and a backward Euler scheme in time. As in the continuous setting, we will show that an entropy-dissipation relation also holds at the discrete level, ensuring stability and preservation of the physical structure of the model.

    \subsection{Definition of discretization mesh}
    First, we introduce the time discretization and the spatial mesh of the domain $\mathbf{\Omega}$. The mesh, given by the triplet $\left( \mathcal{T}, \mathcal{E}, \lbrace\mathbf{x}_K\rbrace_{K \in \mathcal{T}}\right)$, will be assumed to be admissible in the sense of \cite{Eymard2000}.
    Let $\mathcal{T}$ denote a family of non-empty, convex, open and polygonal \textit{control volumes} $K \in \mathcal{T}$, whose Lebesgue measure is denoted by $m_K$. For $K, L \in \mathcal{T}$ with $K \neq L$ we assume that the intersection is empty. Also, we infer that the union of the closure of all control volumes is equal to the closure of the domain, i.e.\
    \[
    \mathbf{\overbar{\Omega}} = \bigcup_{K \in \mathcal{T}} \overbar{K}.
    \]
    The subset of cells contained in the intrisic domain is denoted by $\mathcal{T}_\text{intr}\subset\mathcal{T}$. It is assumed that the closure of the control volumes in the intrinsic domain form a partition of $\mathbf{\Omega}_\text{intr}$, namely
    \[
    \mathbf{\overbar{\Omega}_\text{intr}} = \bigcup_{K \in \mathcal{T}_\text{intr}} \overbar{K}.
    \]
    Further, we call $\mathcal{E}$ a \textit{family of faces}, where $\sigma \in \mathcal{E}$ is a closed subset of $\overbar{\mathbf{\Omega}}$ contained in a hyperplane of $\mathbb{R}^d$. Each $\sigma$ has a strictly positive $(d-1)$-dimensional measure, denoted by $m_\sigma $.
    We use the abbreviation $\sigma = K | L = \partial K \cap \partial L$ for the intersection between two distinct control volumes which is either empty or reduces to a face contained in $\mathcal{E}$.
    Also, for any $K \in \mathcal{T}$ we assume that there exists a subset $ \mathcal{E}_K$ of $\mathcal{E}$ such that the boundary of a control volume can be described by $\partial K = \bigcup_{\sigma \in \mathcal{E}_K} \sigma$ and, consequently, it follows that
    $\mathcal{E} = \bigcup_{K \in \mathcal{T}} \mathcal{E}_K$.
    The set of faces contained in the intrinsic domain are denoted by
    \[
    \mathcal{E}_\text{intr} = \{\sigma\in\mathcal{E}\ \text{s.t.}\ \sigma\subset\mathbf{\overline{\Omega}}_\text{intr}\}.
    \]
   Now, we distinguish the faces that are on the boundary of $\mathbf{\Omega}$ 
   by the notations
    \[
    \mathcal{E}^D = \{\sigma\in\mathcal{E}\ \text{s.t.}\ \sigma\subset \mathbf{\Gamma}^D\}, \quad \mathcal{E}^N = \{\sigma\in\mathcal{E}\ \text{s.t.}\ \sigma\subset \mathbf{\Gamma}^N\}.
    \]
    These sets form partitions of $\mathbf{\Gamma}^D$ and $\mathbf{\Gamma}^N$, respectively. We also introduce the set of interior faces in the whole and the intrinsic domain, respectively
    \[
    \mathcal{E}^\text{int} = \{\sigma\in\mathcal{E}\ \text{s.t.}\ \sigma\not\subset \partial\mathbf{\Omega}\}, \quad \mathcal{E}_{\text{intr}}^\text{int} = \{\sigma\in\mathcal{E}_\text{intr}\ \text{s.t.}\ \sigma\not\subset \partial\mathbf{\Omega}_\text{intr}\}.
    \]
    To each control volume $K \in \mathcal{T}$ we assign a \textit{cell center} $\mathbf{x}_K \in K$ and we assume that the family of cell centers $(x_K)_{K\in{\mathcal T}}$ satisfies the \textit{orthogonality condition}: If $K$ and  $L$ share a face $\sigma =K | L$, then the vector \[\overline{\mathbf{x}_K \mathbf{x}_L}\text{ is orthogonal to }\sigma = K | L.\]
    For each edge $\sigma\in{\mathcal E}$, we define $d_\sigma$ as the Euclidean distance between $\mathbf{x}_K$ and $\mathbf{x}_L$, if $\sigma= K|L$ or between $\mathbf{x}_K$ and the affine hyperplane spanned by $\sigma$, if $\sigma\subset\partial\mathbf{\Omega}$. Lastly, we introduce the transmissibility of the edge $\sigma$:
    \[
    \tau_\sigma = \frac{m_\sigma}{d_\sigma}.
    \]
    The notations are illustrated in Figure~\ref{fig:control-vol}.
    \begin{figure}[H]
        \centering
        \begin{subfigure}[b]{0.43\textwidth}
            \includegraphics[width=\textwidth]{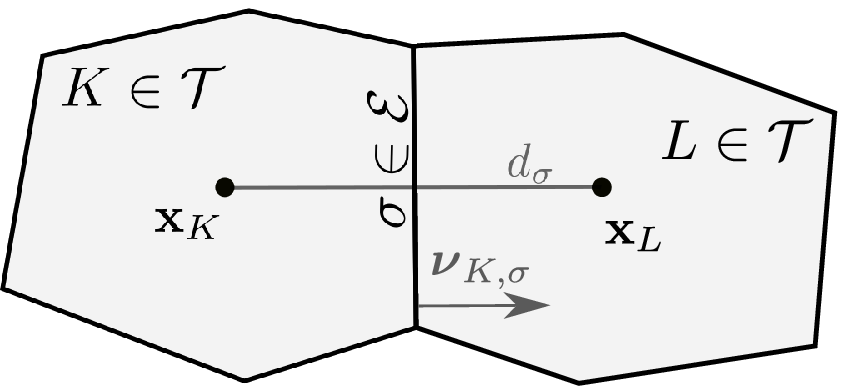}
            \caption{} 	\label{fig:inner-control}
        \end{subfigure}%
        \hspace{0.8cm}
        \begin{subfigure}[b]{0.30\textwidth}
            \includegraphics[width=\textwidth]{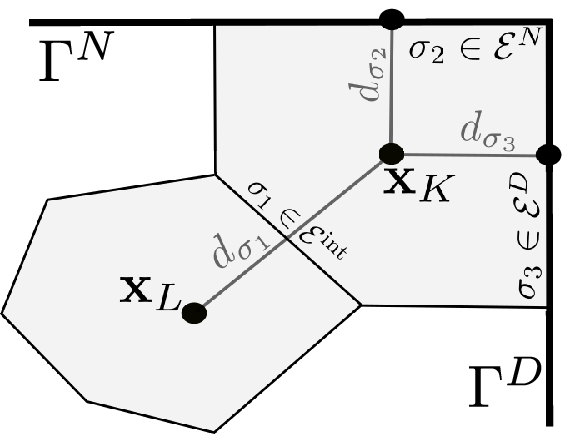}
            \caption{} \label{fig:outer-control}
        \end{subfigure}%
        \caption{Neighboring control volumes (a) in the interior of the domain and (b) near outer boundaries $\mathbf{\Gamma}^D$ and $\mathbf{\Gamma}^N$.}
        \label{fig:control-vol}
    \end{figure}
    We assume that the mesh is regular in the following sense. There is a constant $\xi >0$, which does not depend on the size of the mesh $h_\mathcal{T} = \max_K(\mathrm{diam}(K))$ such that
    \[
    \forall K\in\mathcal{T},\ \forall \sigma\in\mathcal{E}_K,\ \left\{
    \begin{array}{l}
        d_{\sigma} \geq \xi\, \mathrm{diam}(K)\,,\\
        m_K\geq \xi\,\sum_{\sigma\in\mathcal{E}_K}m_\sigma d_\sigma\,.
    \end{array}
    \right.
    \]
    The regularity assumptions have to be understood as an asymptotic property as $h_\mathcal{T}\to 0$ which are always satisifed on a given mesh due to the finite number of cells. We remark that Voronoi meshes satisfy all the assumptions stated in this section.

    For the time discretization we decompose the interval $[0, t_F]$, for a given end time $t_F>0$ into a finite and increasing number of time steps
    $
    0 = t^1 < \ldots < t^M = t_F
    $
    with a step-size $\tau^m = t^{m} - t^{m-1}$ at time step $m = 2,\dots, M$. We finally introduce $\Delta t = \max_{m = 2,\dots, M}\tau^m$.
    \subsection{Finite volume discretization}
    Now, we introduce the finite volume discretization for \eqref{eq:model-dimless}--\eqref{eq:cont-flux-dimless}. In what follows, the quantity $u_K^{m}$ represents an approximation of the mean value of $u(\mathbf{x}, t)$ on the cell $K$ at time $t^{m}$, where $u$ is one of the potentials $\varphi_n, \varphi_p$, $\psi$. In this case, we define ${\mathbf u}^m= (u_K^{m})_{K\in{\mathcal T}}$. For $\varphi_a$ the approximation is only given for $K\in\mathcal{T}_\text{intr}$, so that we define ${{\boldsymbol \varphi}}_a^m= (\varphi_{a,K}^{m})_{K\in{\mathcal T}_\textrm{intr}}$.
    The discretizations of the doping profile $C$, the photogeneration rate $G$ and the boundary data $\varphi^D$, $\psi^D$ are given by
    \[
    \chi_K = \frac{1}{m_K}\int_K \chi(\mathbf{x}) d \mathbf{x}, \quad K\in\mathcal{T},\ \chi = C,\, G,\,\psi^D\text{ or }\varphi^D,
    \]
    and
    \[
    \chi_\sigma = \frac{1}{m_\sigma}\int_\sigma \chi(\gamma) d \gamma, \quad \sigma\in\mathcal{E}^D,\  \chi = \psi^D,\, \varphi^D.
    \]
    We discretize in the same way the initial conditions $\varphi_n^0$, $\varphi_p^0$,  $\varphi_a^0$, which lead to the corresponding vectors ${{\boldsymbol \varphi}}_n^0$, ${{\boldsymbol \varphi}}_p^0$ and ${{\boldsymbol \varphi}}_a^0$.
    The finite volume scheme is formulated as follows. First, the discrete mass balance equations for the three charge carriers are given by
    \begin{subequations}\label{eq:discrete-mass-balance}
        \begin{align}
            \nu z_n m_K \frac{n_{n,K}^{m} - n_{n,K}^{m-1} }{\tau^m} + \sum_{\sigma \in \mathcal{E}_K} J_{n,K, \sigma}^m &=  z_n m_K \Bigl(\gamma G_K-R(n_{n,K}^m, n_{p,K}^m) \Bigr),&& K\in\mathcal{T},\ m\in\mathbb{N}, \label{eq:discrete-mass-balance-elec}\\
            \nu z_p m_K \frac{n_{p,K}^{m} - n_{p,K}^{m-1} }{\tau^m} + \sum_{\sigma \in \mathcal{E}_K} J_{p,K, \sigma}^m &=  z_p m_K \Bigl(\gamma G_K-R(n_{n,K}^m, n_{p,K}^m)\Bigr),&& K\in\mathcal{T},\ m\in\mathbb{N}, \label{eq:discrete-mass-balance-holes}\\
            z_a m_K \frac{n_{a,K}^{m} - n_{a,K}^{m-1} }{\tau^m} + \sum_{\sigma \in \mathcal{E}_K} J_{a,K, \sigma}^m &=  0,&& K\in\mathcal{T}_\text{intr},\ m\in\mathbb{N}. \label{eq:discrete-mass-balance-anions}
        \end{align}
    \end{subequations}
    They are coupled via the discrete Poisson equation
    \begin{equation}\label{eq:model-poisson-discrete}
        - \lambda^2\sum_{\sigma \in \mathcal{E}_K}\tau_\sigma D_{K,\sigma}\boldsymbol{\psi}^m =\left\{{\begin{aligned}
                &\delta m_K(z_nn_{n,K}^{m} + z_pn_{p,K}^{m} + C_K),&&K\in\mathcal{T}\setminus\mathcal{T}_\text{intr},\ m\in\mathbb{N},\\
                &m_Kz_an_{a,K}^{m} + \delta m_K (z_nn_{n,K}^{m} + z_pn_{p,K}^{m}  + C_K ),&&K\in\mathcal{T}_\text{intr},\ m\in\mathbb{N}.
            \end{aligned}}\right.
    \end{equation}
    In the previous equation, the notation $D_{K,\sigma}$ denotes the finite difference operator acting on vectors of unknowns ${\boldsymbol u} = (u_K)_K$ and is given by
    \begin{equation}\label{eq:difference-operator}
        D_{K,\sigma}{\boldsymbol u} = \left\{
        \begin{array}{ll}
            u_L - u_K, &\text{if }\sigma = K|L,\\
            u_\sigma^D - u_K, &\text{if }\sigma\in\mathcal{E}^D,\\
            0, &\text{otherwise.}
        \end{array}
        \right.
    \end{equation}
    The discrete densities are given by the state equation \eqref{eq:state-eq-dimless} inside the domain and at the Dirichlet boundary, namely
    \begin{subequations}\label{eq:discrete-state-eq}
        \begin{align}
            {\boldsymbol n}_{\alpha}^m &= \mathcal{F}_\alpha \left(z_\alpha ({\boldsymbol \varphi}_{\alpha}^m - {\boldsymbol \psi}^m)\right), && \alpha \in \{ n, p, a\},\ m\in\mathbb{N},\label{eq:discrete-state-eq-m}\\
            {\boldsymbol n}_{\alpha}^D &= \mathcal{F}_\alpha \left(z_\alpha ({\boldsymbol \varphi}_{\alpha}^D - {\boldsymbol \psi}^D)\right), && \alpha \in \{ n, p\},\label{eq:discrete-state-eq-Dirichlet-inside}\\
            n_{\alpha, \sigma}^D &= \mathcal{F}_\alpha \left(z_\alpha (\varphi_{\sigma}^D - \psi^D_\sigma)\right), && \alpha \in \{ n, p\},\ \sigma \in \mathcal{E}^D,\label{eq:discrete-state-eq-Dirichlet}
        \end{align}
    \end{subequations}
where the statistics function is applied pointwise to the input vector. Let us remark that the discrete values of the boundary densities, defined by \eqref{eq:discrete-state-eq-Dirichlet-inside} and \eqref{eq:discrete-state-eq-Dirichlet}, are bounded but that the upper bound may differ from $\Vert n_\alpha^D\Vert_{\infty}$. Indeed, for $\alpha\in \{n,p\}$, we have
\begin{equation}\label{def:yalphaD}
\max\Bigl(\max_{K\in{\mathcal T}} n_{\alpha,K}^D, \max_{\sigma\in{\mathcal E}^D} n_{\alpha,\sigma}^D
\Bigr)\leq {\mathcal F}_{\alpha}(\Vert \varphi^D\Vert_{\infty}+ \Vert \psi^D\Vert_{\infty}) = y_{\alpha}^D.
\end{equation}

We use the excess chemical potential scheme (frequently called Sedan scheme) as TPFA scheme for $J_{\alpha,K, \sigma}^m$. The earliest reference, we could find for this thermodynamically consistent flux discretization scheme is \cite{Yu1988}. The definition of the numerical flux is based on the following reformulation of the currents:
$$
{\mathbf j}_\alpha=-z_{\alpha}\left(\nabla n_{\alpha}+n_{\alpha}\nabla(z_\alpha\varphi_\alpha-\log n_\alpha)\right)
$$
and on the approximation of convection-diffusion fluxes by Scharfetter-Gummel numerical fluxes, see \cite{Cances2020, SG69}. For the electrons and holes, it reads
    \begin{equation}\label{eq:discrete-flux-elec-holes}
        \forall\alpha\in\{n,p\},\quad J_{\alpha,K, \sigma}^m = \left\{
        \begin{array}{ll}
            - z_\alpha \tau_\sigma \Bigl( B\left(- Q_{\alpha, K, \sigma}^m \right) n_{\alpha, L}^m - B\left( Q_{\alpha, K, \sigma}^m \right) n_{\alpha, K}^m\Bigr), &\text{if }\sigma = K|L,\\
            - z_\alpha \tau_\sigma \Bigl( B\left(- Q_{\alpha, K, \sigma}^m \right) n_{\alpha, \sigma}^D - B\left( Q_{\alpha, K, \sigma}^m \right) n_{\alpha, K}^m\Bigr), &\text{if }\sigma\in\mathcal{E}^D,\\
            0, &\text{otherwise}.
        \end{array}
        \right.
    \end{equation}
    For the anion vacancies it is given by
    \begin{equation}\label{eq:discrete-flux-anions}
        J_{a,K, \sigma}^m = \left\{
        \begin{array}{ll}
            - z_a \tau_\sigma \Bigl( B\left(- Q_{a, K, \sigma}^m \right) n_{a, L}^m - B\left( Q_{a, K, \sigma}^m \right) n_{a, K}^m\Bigr), &\text{if }\sigma= K|L\in\mathcal{E}^\text{int}_\text{intr},\\
            0, &\text{otherwise}.
        \end{array}
        \right.
    \end{equation}
    The quantity $Q_{\alpha, K, \sigma}^m$ is defined as
    \begin{equation}\label{eq:inside-Bernoulli}
        Q_{\alpha, K, \sigma}^m = D_{K,\sigma} \left( z_\alpha{\boldsymbol \varphi}_{\alpha}^m - \log {\boldsymbol n}_{\alpha}^m \right)
    \end{equation}
    with $K\in\mathcal{T}$, $\sigma\in\mathcal{E}_K\cap(\mathcal{E}^\text{int}\cup\mathcal{E}^D)$ for electrons and holes ($\alpha\in\{n,p\}$) and $K\in\mathcal{T}_\text{intr}$, $\sigma\in\mathcal{E}_K\cap\mathcal{E}^\text{int}_\text{intr}$ in the case of anion vacancies ($\alpha = a$). In the previous formula, the logarithm is applied componentwise.
    Lastly, the function $B$ denotes the Bernoulli function
    \begin{equation} \label{eq:bernoulli}
        B(x) = \frac{x}{\exp(x) -1},\; \text{for}\; x \neq 0 \quad \text{and} \quad B(0)=1.
    \end{equation}
     Note that the fluxes are locally conservative in the sense that for $\sigma =K|L$
    \begin{equation}\label{eq:conservativity}
        0=J_{\alpha,K, \sigma}^m + J_{\alpha,L, \sigma}^m = Q_{\alpha, K, \sigma}^m + Q_{\alpha, L, \sigma}^m = D_{K,\sigma}{\boldsymbol \psi}^m + D_{L,\sigma}{\boldsymbol \psi}^m = D_{K,\sigma}{\boldsymbol \varphi}_{\alpha}^m + D_{L,\sigma}{\boldsymbol \varphi}_{\alpha}^m.
    \end{equation}
    Since the fluxes $J_{\alpha,K, \sigma}^m$ and $J_{\alpha,L, \sigma}^m $  agree up to sign for any interior edge, we introduce the notation
    $$
    D_{\sigma}{\boldsymbol u} = \vert  D_{K,\sigma}{\boldsymbol u}\vert  \textrm { for } \ \sigma \in\E_K, \textrm{ for } {\boldsymbol u}={\boldsymbol\psi}^m,{\boldsymbol\varphi}_\alpha^m, \ \alpha\in\{n,p,a\}.
    $$
    \begin{remark}[Boundary conditions] Observe that all the boundary conditions have been considered in the definition of the scheme. The external boundary conditions \eqref{eq:outer-bc} for the electric potential are handled in the definition of \eqref{eq:difference-operator}. For the quasi Fermi potentials of electrons and holes external boundary conditions are included in the definition of \eqref{eq:discrete-state-eq-Dirichlet} and \eqref{eq:discrete-flux-elec-holes} as well as \eqref{eq:difference-operator} and \eqref{eq:inside-Bernoulli}. The Neumann boundary conditions for anion vacancies \eqref{eq:outer-bc_anion} are included in the definition of \eqref{eq:discrete-flux-anions}. Finally, observe that the continuity of fluxes of electrons, holes and electric potential through the interfaces $\mathbf{\Sigma}_\text{ETL}$ and $\mathbf{\Sigma}_\text{HTL}$ is automatically ensured thanks to \eqref{eq:conservativity}.
    \end{remark}
    \begin{remark}[Mobilities, permittivity and band-edge energies] As explained at the beginning of Section~\ref{sec:nondimens-model}, we made several simplifications concerning the mobilities, permittivity and band-edge energy in order to lighten the presentation. The generalization of the present scheme to take into account non-constant mobilities $\mu_n, \mu_p$ and permittivity $\varepsilon_s$ amounts to introducing a consistent prefactor depending on the edge $\sigma$ in formula \eqref{eq:discrete-flux-elec-holes} and in the sum of the left hand side of \eqref{eq:model-poisson-discrete} respectively. To take into account non-zero band-edge energies $E_\alpha$ one needs to add the corresponding term to the quasi Fermi potential $\boldsymbol{\varphi}_\alpha$ in \eqref{eq:inside-Bernoulli}.
    \end{remark}
    At first glance, it might not be obvious why the fluxes \eqref{eq:discrete-flux-elec-holes} and \eqref{eq:discrete-flux-anions} are discrete versions of \eqref{eq:cont-flux-dimless}. It turns out that one can define
    \begin{equation}\label{eq:interface-densities_np}
        \overline{n}^{m}_{\alpha,\sigma} :=\left\{\begin{array}{ll} \displaystyle \frac{ B(- Q_{\alpha, K, \sigma}^m ) n_{\alpha,L}^m - B( Q_{\alpha, K, \sigma}^m ) n_{\alpha, K}^m  }{z_\alpha (\varphi_{\alpha,L}^{m}-\varphi_{\alpha,K}^{m})}, & \sigma = K|L,\  \alpha\in\{n,p\},   \\[1em]
            \displaystyle \frac{ B(- Q_{\alpha, K, \sigma}^m ) n_{\alpha, \sigma}^D - B( Q_{\alpha, K, \sigma}^m ) n_{\alpha, K}^m  }{z_\alpha (\varphi_{\sigma}^{D}-\varphi_{\alpha,K}^{m})}, & \sigma \in\mathcal{E}^D\cap\mathcal{E}_K, \,\  \alpha\in\{n,p\},
        \end{array}
        \right.
    \end{equation}
    and
      \begin{equation}\label{eq:interface-densities_a}
        \overline{n}^{m}_{a,\sigma} :=\displaystyle \frac{ B(- Q_{a, K, \sigma}^m ) n_{a,L}^m - B( Q_{a, K, \sigma}^m ) n_{a, K}^m  }{z_\alpha (\varphi_{a,L}^{m}-\varphi_{a,K}^{m})},\ \ \sigma  = K|L \in {\mathcal E}_{\textrm{intr}}^{\textrm{int}},
    \end{equation}
    so that the fluxes can be rewritten to
    \begin{equation}\label{eq:reformulated-fluxes}
        {J}_{\alpha, K, \sigma}^m = -\tau_\sigma z_\alpha^2\overline{n}^{m}_{\alpha,\sigma}\,D_{K, \sigma} \boldsymbol{\varphi}_\alpha^m, \, \quad \text{for all} \;  \alpha\in\{n,p,a\}.
    \end{equation}
    Observe that $\overline{n}^{m}_{\alpha,\sigma}$ is well-defined in the sense that thanks to \eqref{eq:conservativity} it depends only on the edge (and not nodal values) as well as the fact that a boundary edge has only one associated control volume. The reformulation of the fluxes \eqref{eq:reformulated-fluxes} now is closer to \eqref{eq:cont-flux-dimless} but the analogy would not be complete, if $\overline{n}^{m}_{\alpha,\sigma}$ is not consistent with the density at the interface $\sigma$.
    This is actually the case as the following lemma shows. It is adapted from \cite[Lemma 3.1]{Cances2020}.
    \begin{lemma} \label{lemma:bound-mean-value}
        The interface value $\overline{n}^{m}_{\alpha,\sigma}$ defined by \eqref{eq:interface-densities_np} is a convex combination of $n^{m}_{\alpha,K}$ and $n^{m}_{\alpha,L}$ (\emph{resp.} $n^{D}_{\alpha,\sigma}$), if $\sigma = K|L$ (\emph{resp.} $\sigma \in\mathcal{E}^D$). In particular it is framed between the minimum and maximum of the two values. The same result holds for $\overline{n}^{m}_{a,\sigma}$ defined by \eqref{eq:interface-densities_a} for $\sigma =K|L\in {\mathcal E}_{\textrm{intr}}^{\textrm{int}} $.
    \end{lemma}
    \begin{proof}
        It suffices to observe that for $\sigma = K|L$ (the boundary case can be readily adapted),
        \[
        \overline{n}^{m}_{\alpha,\sigma}
        = \frac{B(y) - B(x)}{x-y} n_{\alpha,L}^{m} + \frac{B(-x) - B(-y)}{x-y} n_{\alpha,K}^m,
        \]
        with $x = D_{K,\sigma} \log {\boldsymbol n}_{\alpha}^{m}$ and $y = - Q_{\alpha, K,\sigma}^m$. To see this, use the expression of the Bernoulli function $B$ to get that the coefficients are non-negative and sum to $1$. We refer to \cite{Cances2020} for additional details concerning this computation.
    \end{proof}

    \subsection{Discrete entropy-dissipation inequality}
    In the following, we derive a discrete counterpart of \eqref{eq:continuous-E-D-inequality} for the discrete relative entropy ($m \in \mathbb{N}$)
    \begin{multline}\label{eq:discrete-entropy}
        \mathbb{E}_\mathcal{T}^m =\,\frac{\lambda^2}{2} \sum_{\sigma \in \mathcal{E}} \tau_\sigma \left(D_{\sigma} ( \boldsymbol{\psi}^m - \boldsymbol{\psi}^D)\right)^2 + \sum_{K \in \mathcal{T}_{\text{intr}}} m_K \Phi_a(n_{a, K}^m) \\+  \delta\sum_{\alpha \in \{ n,p \} }  \sum_{K \in \mathcal{T}} m_K H_\alpha(n_{\alpha, K}^m, n_{\alpha, K}^D).
    \end{multline}
    We recall that the entropy functions $\Phi_a$ and $H_{\alpha}$ are defined in \eqref{eq:entropyfunc} and \eqref{eq:H-function}. The corresponding discrete non-negative dissipation $\mathbb{D}_{\mathcal{T}}^m$ for $m \in \mathbb{N}$ is given by
    \begin{multline} \label{eq:discrete-dissip}
        \mathbb{D}_{\mathcal{T}}^m = \frac{z_a^2}{2}\sum_{\sigma \in \mathcal{E}_\text{intr}^\text{int}} \tau_\sigma \overline{n}^{m}_{a,\sigma}  (D_{ \sigma} \boldsymbol{\varphi}_{a}^m)^2 + \frac{\delta}{2\nu}\sum_{\alpha \in \{ n,p \} }  \sum_{\sigma \in \mathcal{E}} \tau_\sigma  \overline{n}^{m}_{\alpha,\sigma}  (D_{\sigma} \boldsymbol{\varphi}_{\alpha}^m)^2\\ +  \frac{\delta}{\nu}\sum_{K \in \mathcal{T}}  m_K R(n^m_{n,K},n^m_{p,K})  \left( \varphi_{p,K}^m - \varphi_{n,K}^m \right).
    \end{multline}
    \begin{theorem}(\textbf{Discrete entropy-dissipation inequality}) \label{thm:discrete-E-D}
        For any solution to the finite volume scheme \eqref{eq:discrete-mass-balance}--\eqref{eq:inside-Bernoulli} one has the following entropy-dissipation inequality: For any $\varepsilon>0$, there is a constant $c_{\varepsilon, \mathbf{\Omega},\xi}>0$ such that for any $m\in\mathbb{N}$, one has
        \begin{equation} \label{eq:discrete-E-D-inequality}
            \frac{\mathbb{E}_{\mathcal{T}}^m - \mathbb{E}_{\mathcal{T}}^{m-1}}{\tau^m} +  \mathbb{D}_{\mathcal{T}}^m
            \leq c_{\varepsilon, \mathbf{\Omega},\xi} + \varepsilon \mathbb{E}_\mathcal{T}^m.
        \end{equation}
        The constant $c_{\varepsilon, \mathbf{\Omega},\xi}$ depends solely on $\varepsilon$, the measure of $\mathbf{\Omega}$, the mesh regularity $\xi>0$, the boundary data and the photogeneration term via the norms $\|G\|_{L^\infty}, \|\varphi^D\|_{W^{1,\infty}}$ and $\|\psi^D\|_{W^{1,\infty}}$, as well as on $z_a$ and the dimensionless parameters $\delta$, $\gamma$ and $\nu$. If $G = 0$ and $\nabla\varphi^D = \nabla\psi^D = \mathbf{0}$, then the right hand-side of \eqref{eq:discrete-E-D-inequality} vanishes.
    \end{theorem}
    \begin{proof}
        Let us start by considering the difference of the entropies at time $t^m$ and $t^{m-1}$, that is
    \begin{align*}
            \mathbb{E}_\mathcal{T}^m - \mathbb{E}_\mathcal{T}^{m-1} =& \,
            \frac{\lambda^2}{2} \sum_{\sigma \in \mathcal{E}} \tau_\sigma
            \left(\left(D_{\sigma} ( \boldsymbol{\psi}^m - \boldsymbol{\psi}^D)\right)^2 - \left(D_{ \sigma} ( \boldsymbol{\psi}^{m-1} - \boldsymbol{\psi}^D)\right)^2 \right)\\
            ~
            &+ \sum_{K \in \mathcal{T}_{\text{intr}}} m_K \left( \Phi_a(n_{a, K}^m) - \Phi_a(n_{a, K}^{m-1}) \right)\\
            ~
            &+ \delta  \sum_{\alpha \in \{ n,p \} }  \sum_{K \in \mathcal{T}} m_K \left( \Phi_\alpha(n_{\alpha, K}^m) - \Phi_\alpha(n_{\alpha, K}^{m-1})  - \Phi_\alpha'\left(n_{\alpha, K}^D\right)(n_{\alpha, K}^m - n_{\alpha, K}^{m-1}) \right).
    \end{align*}
    Using a convexity inequality in every sum one finds
    \begin{align*}
        \mathbb{E}_\mathcal{T}^m - \mathbb{E}_\mathcal{T}^{m-1} \leq& \,
        \lambda^2\sum_{\sigma \in \mathcal{E}} \tau_\sigma D_{K, \sigma} \left( \boldsymbol{\psi}^m - \boldsymbol{\psi}^D \right) D_{K, \sigma} \left(\mathbf{\psi}^m - \mathbf{\psi}^{m-1} \right)\\
        ~
        &+ \sum_{K \in \mathcal{T}_{\text{intr}}} m_K \mathcal{F}^{-1}_a(n_{a, K}^m) \Bigl( n_{a, K}^{m} - n_{a, K}^{m-1}\Bigr)\\
        ~
        &+ \delta \sum_{\alpha \in \{ n,p \} }  \sum_{K \in \mathcal{T}} m_K \left( \mathcal{F}^{-1}_\alpha(n_{\alpha, K}^m)-  \mathcal{F}^{-1}_\alpha\left(n_{\alpha, K}^D\right) \right) \left(n_{\alpha, K}^m - n_{\alpha, K}^{m-1}\right).
    \end{align*}
    In order to compute the first sum, we use a discrete integration by parts (consisting in reordering sums by using the conservativity relations on fluxes) and the discrete Poisson equation \eqref{eq:model-poisson-discrete} to get
    \[
    \begin{aligned}
     &\lambda^2 \sum_{\sigma \in \mathcal{E}} \tau_\sigma D_{K, \sigma} \left( \boldsymbol{\psi}^m - \boldsymbol{\psi}^D \right) D_{K, \sigma} \left(\boldsymbol{\psi}^m - \boldsymbol{\psi}^{m-1} \right)\\
     =& -\lambda^2 \sum_{K\in \mathcal{T}}\sum_{\sigma \in \mathcal{E}_K} \tau_\sigma D_{K, \sigma} \left(\boldsymbol{\psi}^m - \boldsymbol{\psi}^{m-1} \right) \left( \psi_K^m - \psi_K^D \right)\\
      =&\ \delta\!\!\sum_{\alpha \in \{ n, p \}} \sum_{K\in \mathcal{T}} m_K     z_\alpha \left(n_{\alpha,K}^m - n_{\alpha,K}^{m-1} \right) \left( \psi_K^m - \psi_K^D \right)
        ~
        + \sum_{K\in \mathcal{T}_{\text{intr}}} m_K z_a \left(n_{a,K}^m - n_{a,K}^{m-1} \right)  \left( \psi_K^m - \psi_K^D \right).\\
    \end{aligned}
    \]
    Plugging this relation back into the initial estimate and using the relation in \eqref{eq:discrete-state-eq}, we obtain
\begin{align*}
        \mathbb{E}_\mathcal{T}^m - \mathbb{E}_\mathcal{T}^{m-1} \leq&
        \ \delta\!\!\sum_{\alpha \in \{ n,p \} }  \sum_{K \in \mathcal{T}} m_K z_\alpha \left( \varphi_{\alpha,K}^m - \varphi_{K}^D \right) \left(n_{\alpha, K}^m - n_{\alpha, K}^{m-1} \right)
        \\
        &+ \sum_{K \in \mathcal{T}_{\text{intr}}} m_K z_a \left(\varphi_{a,K}^m - \psi_K^D \right) \left(n_{a, K}^m - n_{a, K}^{m-1}\right).
\end{align*}
Now, divide by the time step size $\tau^m$ and insert the mass balances in \eqref{eq:discrete-mass-balance}
\begin{multline*}
       \frac{ \mathbb{E}_\mathcal{T}^m - \mathbb{E}_\mathcal{T}^{m-1}}{\tau^m} \leq
       -\frac{\delta}{\nu} \sum_{\alpha \in \{ n,p \} } \sum_{K \in \mathcal{T}}  \sum_{\sigma \in \mathcal{E}_K} J_{\alpha,K, \sigma}^m  \left( \varphi_{\alpha,K}^m - \varphi_{K}^D \right)- \sum_{K \in \mathcal{T}_{\text{intr}}} \sum_{\sigma \in \mathcal{E}_K} J_{a,K, \sigma}^m   \left(\varphi_{a,K}^m - \psi_K^D \right) \\+ \frac{\delta}{\nu}\sum_{\alpha \in \{ n,p \} } \sum_{K \in \mathcal{T}} z_\alpha m_K \Bigl(\gamma G_K-R(n_{n,K}^m,n_{p,K}^m) \Bigr)  \left( \varphi_{\alpha,K}^m - \varphi_{K}^D \right).
\end{multline*}
Next, we insert the formulas for the fluxes \eqref{eq:reformulated-fluxes} with $z_\alpha^2 =1$, $\alpha = n, p$,  and perform a discrete integration by parts to deduce
\begin{multline*}
\frac{ \mathbb{E}_\mathcal{T}^m - \mathbb{E}_\mathcal{T}^{m-1}}{\tau^m} \leq
- \frac{\delta}{\nu}\sum_{\alpha \in \{ n,p \} }   \sum_{\sigma \in \mathcal{E}^\text{int}\cup\mathcal{E}^D} \tau_\sigma \overline{n}^{m}_{\alpha,\sigma}\,D_{K, \sigma} \boldsymbol{\varphi}_\alpha^m  D_{K, \sigma}\left( \boldsymbol{\varphi}_{\alpha}^m - \boldsymbol{\varphi}^D \right)\\
~
-  z_a^2\sum_{\sigma \in \mathcal{E}^\text{int}_\text{intr}}  \tau_\sigma \overline{n}^{m}_{a,\sigma}\,D_{K, \sigma} \boldsymbol{\varphi}_a^m    D_{K, \sigma}\left(\boldsymbol{\varphi}_{a}^m - \boldsymbol{\psi}^D \right)\\
~
- \frac{\delta}{\nu}\sum_{K \in \mathcal{T}} m_K R(n_{n,K}^m,n_{p,K}^m)  \left( \varphi_{p,K}^m - \varphi_{n,K}^m\right) + \frac{\delta\gamma}{\nu}\sum_{K \in \mathcal{T}}  m_K G_{K}   \left( \varphi_{p,K}^m - \varphi_{n,K}^m\right).
   \end{multline*}
After using the inequality $-a(a-b)\leq -(a^2-b^2)/2$ in the first two sums, we obtain

\begin{equation}\label{eq:remainders_to_estimate}
    \begin{split}
    \frac{ \mathbb{E}_\mathcal{T}^m - \mathbb{E}_\mathcal{T}^{m-1}}{\tau^m} +\mathbb{D}_\mathcal{T}^m &\leq
    ~
    \frac{\delta}{2\nu}\sum_{\alpha \in \{ n,p \} }   \sum_{\sigma \in \mathcal{E}^\text{int}\cup\mathcal{E}^D} \tau_\sigma \overline{n}^{m}_{\alpha,\sigma}\,(D_{\sigma} \boldsymbol{\varphi}^D)^2+\frac{z_a^2}{2} \sum_{\sigma \in \mathcal{E}^\text{int}_\text{intr}}  \tau_\sigma \overline{n}^{m}_{a,\sigma}\,(D_{\sigma} \boldsymbol{\psi}^D)^2  \\
    ~
    &\quad\; +\frac{\delta\gamma}{\nu} \sum_{K \in \mathcal{T}}  m_K G_{K}   \left( \varphi_{p,K}^m - \varphi_{n,K}^m\right).
    \end{split}
\end{equation}

   Observe that at this stage it is obvious that, if $G = 0$ and $\nabla\varphi^D = \nabla\psi^D = \mathbf{0}$, then the entropy-dissipation inequality of the theorem holds for a vanishing right hand-side. In the general case, it remains to estimate the different remainder terms in the right-hand-side of \eqref{eq:remainders_to_estimate}.

   For the first and second remainder terms in \eqref{eq:remainders_to_estimate} we need the following intermediate result. Let $\sigma = K|L\in\mathcal{E}^\text{int}$, then
\begin{align*}
\frac{D_{\sigma} \boldsymbol{\varphi}^D}{d_\sigma}&\leq \frac{1}{d_\sigma m_K m_L}\int_{K}\int_L|\varphi^D(\mathbf{x})-\varphi^D(\mathbf{y})|\,d\mathbf{x}\,d\mathbf{y}\\
&\leq \frac{\mathrm{diam}(K) + \mathrm{diam}(L)}{d_\sigma}\|\nabla\varphi^D\|_{L^\infty}\leq \frac{2}{\xi}\|\nabla\varphi^D\|_{L^\infty}\,.
\end{align*}
The last inequality holds also, if $\sigma\in\mathcal{E}^D$ or by replacing $\varphi^D$ with $\psi^D$. Let us now go back to the first remainder term of \eqref{eq:remainders_to_estimate}. We set
$$
S_1=\frac{\delta}{2\nu}\sum_{\alpha \in \{ n,p \} }   \sum_{\sigma \in \mathcal{E}^\text{int}\cup\mathcal{E}^D} \tau_\sigma \overline{n}^{m}_{\alpha,\sigma}\,(D_{\sigma} \boldsymbol{\varphi}^D)^2.
$$
Since $\overline{n}^{m}_{\alpha,\sigma}$ is the convex combination of two non-negative unknowns (see Lemma~\ref{lemma:bound-mean-value}) it is bounded from above by the sum of these unknowns. It yields
\begin{align*}
  S_1
  &\leq \frac{2\delta}{\nu \xi^2} \|\nabla\varphi^D\|_{L^\infty}^2\sum_{\alpha \in \{ n,p \} }\sum_{\sigma \in \mathcal{E}^\text{int}\cup\mathcal{E}^D} m_\sigma\,d_\sigma \overline{n}^{m}_{\alpha,\sigma}\\
  ~
  &\leq  \frac{2\delta}{\nu \xi^2} \|\nabla\varphi^D\|_{L^\infty}^2\sum_{\alpha \in \{ n,p \} }\left(2\sum_{K\in\mathcal{T}}n_{\alpha,K}^m\sum_{\sigma \in \mathcal{E}_K} m_\sigma\,d_\sigma + \sum_{\sigma \in \mathcal{E}^D} m_\sigma\,d_\sigma y_\alpha^D\right)\\
  ~
  &\leq \frac{2\delta}{\nu \xi^2} \|\nabla\varphi^D\|_{L^\infty}^2\sum_{\alpha \in \{ n,p \} }\left(\frac{2}{\xi}\sum_{K\in\mathcal{T}}m_Kn_{\alpha,K}^m + |\mathbf{\Omega}|y_\alpha^D\right)\\
  ~
  &\leq \frac{2\delta}{\nu \xi^2} \|\nabla\varphi^D\|_{L^\infty}^2\left(\frac{2}{\xi}\varepsilon\mathbb{E}_\mathcal{T}^m + \sum_{\alpha \in \{ n,p \} } \left(\frac{2}{\xi}  |\mathbf{\Omega}| c_{y_\alpha^D,\varepsilon} + |\mathbf{\Omega}|y_\alpha^D\right)\right)\,,
\end{align*}
where $y_\alpha^D$ has been defined in \eqref{def:yalphaD} and $c_{y_\alpha^D,\varepsilon}>0$ is the constant of the inequality (i) in Lemma~\ref{lemma:bound-primitive-argument}. Similarly, by using (ii) in Lemma~\ref{lemma:bound-primitive-argument} we obtain that the second remainder term satisfies
\begin{align*}
 S_2 = \frac{z_a^2}{2} \sum_{\sigma \in \mathcal{E}^\text{int}_\text{intr}}  \tau_\sigma \overline{n}^{m}_{a,\sigma}\, (D_{ \sigma} \boldsymbol{\psi}^D)^2&\leq\frac{4z_a^2}{\xi^3}   \|\nabla\psi^D\|_{L^\infty}^2    (\varepsilon\mathbb{E}_\mathcal{T}^m + c_\varepsilon   |\mathbf{\Omega}|)\,.
\end{align*}
For the last remainder term coming from the  photogeneration we set
   $$
   S_3 = \frac{\delta\gamma}{\nu}   \sum_{K \in \mathcal{T}}  m_K G_{K}   \left( \varphi_{p,K}^m - \varphi_{n,K}^m\right) = \frac{\delta\gamma}{\nu}   \sum_{K \in \mathcal{T}}  m_K G_{K}   \left(\mathcal{F}^{-1}_n(n_{n, K}^m)+\mathcal{F}^{-1}_p(n_{p, K}^m)\right),
   $$
   and we use the state equation \eqref{eq:discrete-state-eq-m} and Lemma~\ref{lemma:relation-inverse-primitive} to estimate
   \begin{align*}
     S_3 &\leq \frac{\delta\gamma}{\nu}   \left(\max_{K\in\mathcal{T}}G_K\right) \sum_{K \in \mathcal{T}}  m_K  \Bigl(\max(\mathcal{F}^{-1}_n(n_{n, K}^m),0)+\max(\mathcal{F}^{-1}_p(n_{p, K}^m),0)\Bigr)\\
     ~
     &\leq \frac{\delta\gamma}{\nu}   ||G||_{L^\infty}\!\!\! \sum_{\alpha \in \{ n,p\} } \sum_{K \in \mathcal{T}} m_K \left( c_{y_\alpha^D, \varepsilon} +    \varepsilon H_\alpha(n_{\alpha, K}^m, n_{\alpha, K}^D ) \right)\\
     ~
     &\leq  \frac{\delta\gamma}{\nu}   ||G||_{L^\infty}\left( |\mathbf{\Omega}|c_{y_\alpha^D, \varepsilon} +    \varepsilon \mathbb{E}_\mathcal{T}^m \right)\,.
\end{align*}
For the last step it is important to remember that each term in the definition of the entropy is non-negative.
Therefore, if we combine everything back into \eqref{eq:remainders_to_estimate} we find
\[
\frac{ \mathbb{E}_\mathcal{T}^m - \mathbb{E}_\mathcal{T}^{m-1}}{\tau^m} +\mathbb{D}_\mathcal{T}^m
~
\leq
~
c_{1, \xi } \Bigl( \varepsilon \mathbb{E}_\mathcal{T}^m + c_{2,\varepsilon, \mathbf{\Omega} } \Bigr),
\]
for some constants $c_{1, \xi } , c_{2,\varepsilon, \mathbf{\Omega} } > 0$ depending on all the aforementioned quantities. Since $c_{1, \xi }$ does not depend on $\varepsilon$, this is equivalent to the desired inequality \eqref{eq:discrete-E-D-inequality} up to a redefinition of $\varepsilon$.
\end{proof}
From the discrete entropy-dissipation  inequality \eqref{eq:discrete-E-D-inequality}, we can deduce some bounds on the entropy $\mathbb{E}_\mathcal{T}^m$ and on the cumulated dissipation $ \sum_{k=1}^m\tau^k\mathbb{D}_\mathcal{T}^k$ for any $m>0$ thanks to a  discrete Grönwall's Lemma.
Corollary~\ref{cor:bounds-entropy-dissip} states the discrete counterpart of Corollary~\ref{cor:bounds-E-D}.
\begin{corollary}\label{cor:bounds-entropy-dissip}
    Provided that $\varepsilon < (\Delta t)^{-1}$, one has for any $m\geq 1$ that
    \begin{equation}\label{eq:bound_discrete_entropy_dissipation}
    \mathbb{E}_\mathcal{T}^m + \sum_{k=1}^m\tau^k\mathbb{D}_\mathcal{T}^k\leq (\mathbb{E}_\mathcal{T}^0 + c_{\varepsilon, \mathbf{\Omega},\xi} t^m)(1-\varepsilon\Delta t)^{-m}.
    \end{equation}
\end{corollary}

\begin{proof}
For $j\in\N$, we define
$$
w^j=\mathbb{E}_\mathcal{T}^j\prod_{k=1}^j (1-\varepsilon \tau^k).
$$
Using \eqref{eq:discrete-E-D-inequality}, we obtain that
$$
w^j-w^{j-1} +\tau^j\mathbb{D}_\mathcal{T}^j \prod_{k=1}^{j-1} (1-\varepsilon \tau^k) \leq c_{\varepsilon, \mathbf{\Omega},\xi} \tau^j\prod_{k=1}^{j-1} (1-\varepsilon \tau^k)
$$
and summing over $m$, we get
$$
w^m-w^0+\sum_{j=1}^m \tau^j \mathbb{D}_\mathcal{T}^j \prod_{k=1}^{j-1} (1-\varepsilon \tau^k)\leq c_{\varepsilon, \mathbf{\Omega}, \xi} \sum_{j=1}^m
 \tau^j\prod_{k=1}^{j-1} (1-\varepsilon \tau^k).
$$
We can now multiply the last inequality by $\prod_{k=1}^m(1-\varepsilon \tau^k)^{-1}$ in order to come back to the discrete entropy.
It yields
$$
\mathbb{E}_\mathcal{T}^m+\sum_{j=1}^m \tau^j \mathbb{D}_\mathcal{T}^j \prod_{k=j}^{m} (1-\varepsilon \tau^k)^{-1}
\leq \prod_{k=1}^m(1-\varepsilon \tau^k)^{-1}\mathbb{E}_\mathcal{T}^0 + c_{\varepsilon,\Omega,\xi} \sum_{j=1}^m
 \tau^j\prod_{k=j}^{m} (1-\varepsilon \tau^k)^{-1}.
$$
But, since $\varepsilon \Delta t <1$, we have
$$
1\leq \prod_{k=j}^{m} (1-\varepsilon \tau^k)^{-1}\leq \prod_{k=1}^{m} (1-\varepsilon \tau^k)^{-1}\leq (1-\varepsilon \Delta t)^{-m},
$$
which yields \eqref{eq:bound_discrete_entropy_dissipation} as $\sum_{j=1}^m\tau^j=t^m.$
\end{proof}
\section{Existence of a discrete solution}
\label{sec:existence}
In this section, we will now establish the existence of a solution to the finite volume scheme \eqref{eq:discrete-mass-balance}--\eqref{eq:inside-Bernoulli}, which consists of a nonlinear system of equations at each time step. Knowing  the solution at step $m-1$, we want to establish the existence of a solution at time step $m$. We may consider that the unknowns of the nonlinear system of equations are the quasi Fermi potentials and the electrostatic potential, as the densities of electrons, holes and anion vacancies are defined as functions of these potentials through \eqref{eq:discrete-state-eq}.
The proof consists of three main parts: we start in \Cref{sec:poisson-existence} showing the existence and uniqueness of a discrete electric potential for given quasi Fermi potentials associated to the Poisson equation and continue in Section \ref{ssec:estimates} with proving some {\em a priori estimates} on the quasi Fermi and electrostatic potentials, obtained as consequences of the bounds on the entropy and the dissipation. Then, in Section \ref{ssec:exthm} the existence of quasi Fermi potentials is shown which finalizes the proof.
For this, forgetting the superscript $m$, we denote by $\mathbf{X}$ the vector containing the unknown quasi Fermi potentials which is defined by
\begin{equation}\label{def.X}
    \mathbf{X}=\Bigl( (\varphi_{n,K}-\varphi_{n,K}^D)_{K\in\T}, (\varphi_{p,K}-\varphi_{p,K}^D)_{K\in\T},(\varphi_{a,K}-\psi_{K}^D)_{K\in\T_{\text{intr}}}\Bigr).
\end{equation}
\subsection{Existence of electric potential} \label{sec:poisson-existence}
The aim of the first lemma is to show the existence of a unique $\boldsymbol{\psi} = \left( \psi_K \right)_{K\in\T}$ dependent on $\mathbf{X}$.
\begin{lemma} \label{lem:discrete-poisson}
    Let $\mathbf{X}$ denote the vector containing the unknown quasi Fermi potentials as defined in \eqref{def.X}. Then, there exists a unique solution $\boldsymbol{\psi}(\mathbf{X})$ to the discrete nonlinear Poisson equation \eqref{eq:model-poisson-discrete}. Further, the mapping $\mathbf{X} \mapsto \boldsymbol{\psi}(\mathbf{X})$ is continuous.
\end{lemma}
\begin{proof}
    Let us define the discrete functional
    \begin{align*}
        \mathcal{J}(\boldsymbol{\Psi}) =\, &\frac{\lambda^2}{2}\sum_{\sigma \in \mathcal{E}}\tau_\sigma |D_{\sigma}\boldsymbol{\Psi}|^2 + \delta\!\!\sum_{\alpha\in\{n,p\}}\sum_{K\in\mathcal{T}}m_K\mathcal{G}_\alpha(z_\alpha(\varphi_{\alpha,K} - \Psi_{K})) + \!\!\!\sum_{K\in\mathcal{T}_\text{intr}}m_K\mathcal{G}_a(z_a(\varphi_{a,K} - \Psi_{K}))\\
        ~
        & - \delta\sum_{K\in\mathcal{T}}m_KC_K\Psi_K,
    \end{align*}
    where $\mathcal{G}_\alpha$ denotes the primitive of $\mathcal{F}_\alpha$ which vanishes at $-\infty$. We can compute $\nabla \mathcal{J} $, where the $K$-th component is given by
    \begin{equation*}
        \frac{\partial \mathcal{J}}{\partial \Psi_K} =
        \left\{
            \begin{aligned}
            &- \lambda^2\sum_{\sigma \in \mathcal{E}_K}\tau_\sigma D_{K,\sigma}\boldsymbol{\Psi}- \delta m_K\Bigl(C_K +  \!\! \sum_{\alpha\in\{n,p\}} z_\alpha\mathcal{F}_\alpha(z_\alpha(\varphi_{\alpha,K} - \Psi_{K})) \Bigr),&&K\in\mathcal{T}\setminus\mathcal{T}_\text{intr},\\[1.0ex]
            ~
            &- \lambda^2\sum_{\sigma \in \mathcal{E}_K}\tau_\sigma D_{K,\sigma}\boldsymbol{\Psi}- \delta m_K\Bigl(C_K +  \!\! \sum_{\alpha\in\{n,p\}} z_\alpha\mathcal{F}_\alpha(z_\alpha(\varphi_{\alpha,K} - \Psi_{K})) \Bigr) \\
            &- m_Kz_a\mathcal{F}_a(z_a(\varphi_{a,K} - \psi_{K})) ,&&K\in\mathcal{T}_\text{intr}.
        \end{aligned}\right.
    \end{equation*}
    We conclude that a solution $\boldsymbol{\psi}$ to the discrete Poisson equation \eqref{eq:model-poisson-discrete} satisfies $\nabla \mathcal{J}(\boldsymbol{\psi}) = \mathbf{0}$. The existence of a global minimum of $\mathcal{J}$ is guaranteed through its continuity and coercivity. The coercivity follows from the coercivity of $\boldsymbol{\Psi} \mapsto \frac{\lambda^2}{2}\sum_{\sigma \in \mathcal{E}}\tau_\sigma |D_{\sigma}\boldsymbol{\Psi}|^2- \delta\sum_{K\in\mathcal{T}}m_KC_K\Psi_K$ (by a discrete Poincar\'e inequality \cite{Eymard2000}) and the boundedness from below of the two other contributions. The strict convexity of $\mathcal{J}$, due to being a sum of a strictly convex and convex functions, gives the uniqueness of this global minimum $\boldsymbol{\psi}$.
    Lastly, the continuity of $\mathbf{X} \mapsto \boldsymbol{\psi}(\mathbf{X})$ follows from the implicit function theorem applied to $\nabla\mathcal{J}$ since the Hessian of $\mathcal{J}$ with respect to $\boldsymbol{\Psi}$ is strictly row diagonally dominant.
\end{proof}
As a consequence of \Cref{lem:discrete-poisson} we can interpret in the following the electric potential as a continuous map $\boldsymbol{\psi} = \boldsymbol{\psi} (\mathbf{X})$.
\subsection{A priori estimates}\label{ssec:estimates}
The discrete entropy defined by \eqref{eq:discrete-entropy} can be denoted by ${\mathbb E}_\T(\mathbf{X})$ and its associated dissipation defined by \eqref{eq:discrete-dissip}
can be denoted by ${\mathbb D}_\T(\mathbf{X})$.  In this dissipation, we may distinguish the contributions of electrons, holes, anion vacancies and of the recombination-generation terms. Therefore, we introduce the following notations
\begin{align}
{\mathbb D}_{\T,a}(\mathbf{X})&=\frac{ z_a^2 }{2}\sum_{\sigma \in \mathcal{E}_{\rm intr}^{\rm int}}\tau_\sigma \overline{n}_{a,\sigma}  ( D_{\sigma} \boldsymbol{\varphi}_{a})^2  ,\label{eq:dis-diss-a}\\
{\mathbb D}_{\T,\alpha}(\mathbf{X})&= \frac{ \delta }{2 \nu }\sum_{\sigma \in \mathcal{E} }\tau_\sigma \overline{n}_{\alpha,\sigma}  ( D_{\sigma} \boldsymbol{\varphi}_{\alpha})^2 ,\quad \text{for all} \; \alpha \in \{n,p\}. \label{eq:dis-diss-alpha}
\end{align}
In this section, the letter $R$ refers to a positive number, not to the recombination term.
\begin{lemma}\label{lem.bounds.psi}
Assume that there exists $M_E>0$, such that ${\mathbb E}_\T(\mathbf{X})\leq M_E$. Then, there exists some $R>0$ depending on $M_E$, $\lambda$ and on the mesh $\T$, such that

\begin{equation}\label{bounds.psi}
-R\leq \psi_K-\psi_K^D\leq R, \quad \forall K\in\T.
\end{equation}
\end{lemma}

\begin{proof}
As the entropy contributions of anion vacancies and of the relative entropies for electrons and holes are non-negative, the bound on ${\mathbb E}_\T(\mathbf{X})$ directly implies a bound on the electric energy,
$$
\frac{\lambda^2}{2}\sum_{ \sigma \in \mathcal{E} } \tau_\sigma (D_\sigma( \boldsymbol{\psi} - \boldsymbol{\psi}^D ))^2 \leq M_E.
$$
For each edge $\sigma\in  \E^D$, such that $\sigma\in \E_K$, we deduce a bound on $|\psi_K-\psi_K^D|$ depending on $M_E$, $\lambda$ and on the mesh. The same bound applies to $|(\psi_K-\psi_K^D) - (\psi_L-\psi_L^D)|$ for any interior edge $\sigma = K|L$. From these bounds, and by using the connectedness of the mesh and the finite number of control volumes one can inductively get a uniform finite bound for all $(\psi_K-\psi_K^D)_{K\in\mathcal{T}}$.
\end{proof}
\begin{lemma}\label{lem.bounds.anions}
Assume that there exists $M_{D}>0$ such that ${\mathbb D}_{\T,a}(\mathbf{X})\leq M_{D}$ and that there also exists ${\bar n}\in (0,1)$, such that
\begin{align}
\frac{1}{| \mathbf{\Omega}|}\sum_{K\in\T_{\textrm{intr}}}m_K n_{a,K}={\bar n}. \label{eq:lem.mass-conserv-a}
\end{align}
Then, there exists some $R>0$ depending on $M_{D}$, ${\bar n}$ and $\T$,  such that
\begin{equation}\label{bounds.phia}
-R\leq \varphi_{a,K}\leq R, \quad \forall K\in\T_{\textrm{intr}}.
\end{equation}
\end{lemma}
Let us first note that due to hypothesis \eqref{hyp:statistics-a} on ${\mathcal F}_a$,  the result stated in Lemma \ref{lem.bounds.anions} is equivalent to the fact that there exists an $\varepsilon \in (0,1)$ satisfying
$$
\varepsilon \leq n_{a,K}\leq 1-\varepsilon,  \quad \forall K\in\T_{\textrm{intr}}.
$$
This result is a direct consequence of \cite[Lemma 3.2]{Cances2020}. Its proof follows the main lines of the proof of Lemma 3.7 in \cite{Cances2020} and is left to the reader.
Lastly, we prove bounds on the quasi Fermi potentials of electric charge carriers.
\begin{lemma}\label{lem.bounds.phialpha}
Let $\alpha \in \{n, p\}$.
Assume that there exists $M_E>0$, such that ${\mathbb E}_\T(\mathbf{X})\leq M_E$ and $M_{D}>0$ such that ${\mathbb D}_{\T,\alpha}(X)\leq M_{D}$. Then, there exists some $R>0$ depending on $M_E$, $M_D$, $\mathbf{\Omega}$, $\T$, $\psi^D$, $\varphi^D$, such that
\begin{equation}\label{bounds.phialpha}
-R\leq \varphi_{\alpha,K}, \leq R, \quad \forall K\in\T.
\end{equation}
\end{lemma}

\begin{proof}
In order to prove Lemma \ref{lem.bounds.phialpha}, we will still stay close to the proof of Lemma 3.7 in  \cite{Cances2020}. It needs an adaptation of Lemma 3.2 in \cite{Cances2020} due to the different hypotheses on the statistics function ${\mathcal F}_{\alpha}$ and the different kind of boundary conditions.

Let us first rewrite ${\mathbb D}_{\T,\alpha}(\mathbf{X})$ by using the reformulation of the fluxes \eqref{eq:reformulated-fluxes} based on the definition \eqref{eq:interface-densities_np} of  ${\overline n}_{\alpha,\sigma}$
$$
{\mathbb D}_{\T,\alpha}(\mathbf{X})=- \frac{ \delta}{2 \nu } \sum_{\sigma \in \mathcal{E}} J_{\alpha,K,\sigma}D_{K,\sigma} {\boldsymbol \varphi}_\alpha,
~
$$
where the flux discretization is defined through \eqref{eq:discrete-flux-elec-holes} and \eqref{eq:inside-Bernoulli}.
 Introducing the function 
 ${\mathcal K}_\alpha : \R\times\R \to \R$ defined by
$$
{\mathcal K}_\alpha(x,a)=\log ({\mathcal F}_\alpha(x-a))-x, \quad \forall (x,a) \in \R\times\R,
$$
we note that
$$
Q_{\alpha,K,\sigma}={\mathcal K}_\alpha(z_\alpha \varphi_{\alpha,K},z_\alpha\psi_K)-{\mathcal K}_\alpha(z_\alpha \varphi_{\alpha,K,\sigma},z_\alpha\psi_{K,\sigma}),
$$
where $\varphi_{\alpha,K,\sigma}$ and  $\psi_{K,\sigma}$ stand for $\varphi_{\alpha, L}, \psi_{L}$, if $\sigma=K|L\in \E^{\textrm{int}}$ and for $\varphi_{\sigma}^D, \psi_{\sigma}^D$, if $\sigma \in\E^D$.
Thus, ${\mathbb D}_{\T,\alpha}(\mathbf{X})$ can be rewritten as
 $$
 {\mathbb D}_{\T,\alpha}(\mathbf{X})=  \frac{ \delta}{2 \nu } \sum_{\sigma \in \mathcal{E}}  \tau_\sigma {\mathcal D}_\alpha\bigl(z_\alpha \varphi_{\alpha,K},
  z_\alpha \varphi_{\alpha,K,\sigma},z_{\alpha}\psi_K,z_\alpha\psi_{K,\sigma}\bigr),
 $$
 with  ${\mathcal D}_\alpha : \R^4 \to \R$ defined by
$$
{\mathcal D}_\alpha(x,y,a,b)=(x-y)\Bigl[ B\Bigl({\mathcal K}_\alpha(x,a)-{\mathcal K}_\alpha(y,b)\Bigr){\mathcal F}_\alpha(x-a)
- B\Bigl({\mathcal K}_\alpha(y,b)-{\mathcal K}_\alpha(x,a)\Bigr){\mathcal F}_\alpha(y-b)
 \Bigr].
$$
Following the strategy of proof of Lemma 3.7 in \cite{Cances2020}, we introduce $\Upsilon_{\overline{\Phi},\overline{\Psi}}:\R\to \R$ defined by
$$
\Upsilon_{\overline{\Phi},\overline{\Psi}}(x)=\inf \left\{{\mathcal D}_\alpha(x,y,a,b); \ -\overline{\Phi}\leq y\leq \overline{\Phi}, -\overline{\Psi}\leq a,b\leq \overline{\Psi}\right\}
$$
and we establish (see \Cref{app:technical-proof}, Lemma \ref{lem:limupsilon}) that
\begin{equation}\label{lim.upsilon}
\lim_{x\to -\infty} \Upsilon_{\overline{\Phi},\overline{\Psi}}(x)=+\infty \quad \text{and} \quad \lim_{x\to +\infty} \Upsilon_{\overline{\Phi},\overline{\Psi}}(x)=+\infty.
\end{equation}
Then, we use that the discrete values of the electrostatic potential are bounded thanks to Lemma \ref{lem.bounds.psi} and that the Dirichlet boundary conditions ensure that there exists
at least one $\varphi_{\alpha,K,\sigma}= \varphi_{\sigma}^D$ which is bounded. Lastly, the bound on $ {\mathbb D}_{\T,\alpha}(\mathbf{X})$ implies that the value $\varphi_{\alpha,K}$ is also bounded thanks to \eqref{lim.upsilon}, where this property propagates from cell to cell, such that it holds on the whole domain.
\end{proof}
\subsection{Existence of quasi Fermi potentials}\label{ssec:exthm}
Finally, we can formulate and prove the existence of discrete solutions in \Cref{thm.exresult}.
\begin{theorem}\label{thm.exresult}
For all $m\geq 1$, the finite volume scheme \eqref{eq:discrete-mass-balance}--\eqref{eq:inside-Bernoulli} for the perovskite model has at least one solution
$(\boldsymbol{\varphi}_n^m, \boldsymbol{\varphi}_p^m,\boldsymbol{\varphi}_a^m, \boldsymbol{\psi}^m)\in\R^{\theta}$
with $\theta=3\textrm{Card}(\T)+\textrm{Card}(\T_{\textrm{intr}})$.
Moreover, this solution satisfies the following $L^{\infty}$ bounds. There exists $R>0$ depending on the data and on the mesh such that
$$
-R\leq \boldsymbol{\varphi}_n^m, \boldsymbol{\varphi}_p^m,\boldsymbol{\varphi}_a^m, \boldsymbol{\psi}^m\leq R, \quad \text{for all}\; m\geq 1,
$$
holds component-wise.
\end{theorem}
The discrete mass balances in \eqref{eq:discrete-mass-balance} at step $m$ constitute a nonlinear system of equations. More precisely, we can introduce a continuous vector field $\mathbf{P}_m : \R^{\theta_\mathbf{X}} \to \R^{\theta_\mathbf{X}}$ with $\theta_\mathbf{X} = 2\textrm{Card}(\T)+\textrm{Card}(\T_{\textrm{intr}})$, such that  $\mathbf{P}_m (\mathbf{X}^m)=\mathbf{0}$ is equivalent to \eqref{eq:discrete-mass-balance}, where $\mathbf{X}^m$ is defined by \eqref{def.X}, noting that we have omitted the superscript $m$ there. In order to formulate the electron and hole components of  $\mathbf{P}_m(\mathbf{X}^m)$, we put every term of the equations \eqref{eq:discrete-mass-balance-elec} and \eqref{eq:discrete-mass-balance-holes} on the left-hand side and rescale by a factor $\delta\tau^m / \nu$. The anion related components  are given by \eqref{eq:discrete-mass-balance-anions} rescaled by $\tau^m$. In order to prove Theorem~\ref{thm.exresult}, we apply a corollary of Brouwer's fixed point theorem \cite[Section 9.1]{Evans:10} to a regularized version of $\mathbf{P}_m$, and then take limits in the regularization parameter. The fixed point lemma reads as follows.
\begin{lemma}\label{lem.Evans}
Let $N\in\N $ and $\mathbf{P}: \R^N\to \R^N$ be a continuous vector field. Assume that there exists $R>0$, such that $\mathbf{P}(\mathbf{X})\cdot \mathbf{X} \geq 0$, if $\Vert \mathbf{X}\Vert =R$. Then, there exists $\mathbf{X}^{\ast}\in\R^N$ such that $\mathbf{P}(\mathbf{X}^{\ast})=\mathbf{0}$ and $\Vert \mathbf{X}\Vert \leq R$.
\end{lemma}
\begin{proof}[Proof of Theorem \ref{thm.exresult}]
First, we prove the existence of quasi Fermi potentials $\mathbf{X}$, where for the sake of readability, we omit the superscript $m$.  We recall that \Cref{lem:discrete-poisson} guarantees the existence of a continuous and uniquely determined map $\mathbf{X} \mapsto \boldsymbol{\psi}(\mathbf{X})$ solving the nonlinear Poisson equation \eqref{eq:model-poisson-discrete} for any given quasi Fermi potentials $\mathbf{X}$. Thus, $\mathbf{P}_m$ is well-defined and continuous.
The scalar product $\mathbf{P}_m(\mathbf{X})\cdot \mathbf{X}$ is given by
\begin{align*}
    \mathbf{P}_m(\mathbf{X})\cdot \mathbf{X} =&
    ~
    \sum_{\alpha \in \{ n, p \} } \sum_{K \in \mathcal{T}} \left( \delta z_\alpha m_K (n_{\alpha,K} - n_{\alpha,K}^{m-1})  (\varphi_{\alpha, K} - \varphi_K^D)
    ~
    +  \frac{\delta\tau^m}{\nu}  \sum_{\sigma \in \mathcal{E}_K} J_{\alpha,K, \sigma} (\varphi_{\alpha, K} - \varphi_K^D) \right)\\
    ~
    ~
    &+\frac{\delta\tau^m}{\nu} \sum_{K \in \mathcal{T}} m_K R(n_{n,K}, n_{p,K}) (\varphi_{p, K} - \varphi_{n,K}) - \frac{\delta \gamma\tau^m}{\nu} \sum_{K \in \mathcal{T}} m_K G_K (\varphi_{p, K} - \varphi_{n,K}) \\
    ~
    ~
    &+ \sum_{K \in \mathcal{T}} z_a m_K (n_{a,K} - n_{a,K}^{m-1}) (\varphi_{a, K} - \psi_K^D) + \tau^m\sum_{K \in \mathcal{T}} \sum_{\sigma \in \mathcal{E}_K} J_{a,K, \sigma}(\varphi_{a, K} - \psi_K^D).
\end{align*}
We have established the following inequality within the proof of \Cref{thm:discrete-E-D} for $\varepsilon > 0$
$$
\mathbf{P}_m(\mathbf{X})\cdot \mathbf{X}\geq (1-\varepsilon \tau^m){\mathbb E}_\T(\mathbf{X})-{\mathbb E}_\T(\mathbf{X}^{m-1}) +\tau^m {\mathbb D}_\T(\mathbf{X}) - \tau^m c_{\varepsilon, \mathbf{\Omega}, \xi},
$$
where $\mathbf{X}^{m-1}$ denotes the known solution at the previous time step $m-1$. For suitable $\varepsilon$, there exists $M > 0$, such that
$$
\mathbf{P}_m(\mathbf{X})\cdot \mathbf{X}\geq\frac{1}{2} {\mathbb E}_\T(\mathbf{X})+\tau^m {\mathbb D}_\T(\mathbf{X}) -M.
$$
Our goal is to use \Cref{lem.Evans} to show the existence of a solution at time step $t^m$. Instead of showing now the non-negativity of the scalar product $\mathbf{P}_m(\mathbf{X})\cdot \mathbf{X}$, we introduce a parameter-dependent regularization of $\mathbf{P}_m$ which satisfies the assumptions of \Cref{lem.Evans}.

For a given $\mu>0$, we define $\mathbf{P}_m^{\mu}(\mathbf{X})=\mathbf{P}_m(\mathbf{X})+\mu \mathbf{X}$, which satisfies
$$
\mathbf{P}_m^{\mu}(\mathbf{X})\cdot \mathbf{X}=\mathbf{P}_m(\mathbf{X})\cdot \mathbf{X} +\mu \Vert \mathbf{X}\Vert^2\geq \mu \Vert \mathbf{X}\Vert^2-M  \: \geq 0, \quad \text{for} \: \Vert \mathbf{X}\Vert \geq  \sqrt{M/\mu} .
$$
Then, Lemma \ref{lem.Evans} shows the existence of $\mathbf{X}^{m,\mu}\in B(\mathbf{0},\sqrt{M/\mu})$, such that $\mathbf{P}_m^{\mu}(\mathbf{X}^{m,\mu})=\mathbf{0}$.
Next, we need to show $\mathbf{X}^{m,\mu}$ is actually uniformly bounded in $\mu$. Let us check the hypotheses of Lemmas \ref{lem.bounds.psi}, \ref{lem.bounds.anions} and \ref{lem.bounds.phialpha}. We take the scalar product of $\mathbf{P}_m^{\mu}(\mathbf{X}^{m,\mu})$ with the vector $  \mathbf{V}= ({\mathbf 0}_\T, {\mathbf 0}_\T,{\mathbf 1}_{\T_{\rm intr}}) $.
Since the sum over all fluxes in the intrinsic region vanishes, we obtain
$$
\sum_{K\in\T_{\rm intr}} z_a  m_K n_{a,K}^{m,\mu}-\sum_{K\in\T_{\rm intr}}  z_a  m_K n_{a,K}^{m-1} +\mu \mathbf{X}^{m,\mu}\cdot \mathbf{V}=0,
$$
and therefore, after rescaling with the measure of $\mathbf{\Omega}$, we have
\begin{equation}\label{eq.mass}
 \left|\frac{1}{| \mathbf{\Omega}|}\sum_{K\in\T_{\rm intr}} z_a  m_K n_{a,K}^{m,\mu}-\frac{1}{| \mathbf{\Omega}|}\sum_{K\in\T_{\rm intr}} z_a  m_K n_{a,K}^{m-1}\right| \leq \frac{\mu}{| \mathbf{\Omega}|}  \Vert \mathbf{X}^{m,\mu}\Vert \Vert \mathbf{V}\Vert \leq \frac{\sqrt{M\mu}}{| \mathbf{\Omega}|}  \Vert  \mathbf{V}\Vert.
\end{equation}
But since the solution at the previous time step exists and hence is bounded, there exists $\varepsilon^{(m-1)} \in (0,1)$, such that $ \frac{1}{| \mathbf{\Omega}|} \sum_{K\in\T_{\rm intr}} m_K n_{a,K}^{m-1}\in (\varepsilon^{(m-1)}, 1-\varepsilon^{(m-1)})$. Thus, we deduce from \eqref{eq.mass} that, for $\mu$ sufficiently small, $\mathbf{X}^{m,\mu}$ satisfies for $\varepsilon^{(m)} = \varepsilon^{(m-1)}/2$
$$
\frac{1}{| \mathbf{\Omega}|} \sum_{K\in\T_{\rm intr}} m_K n_{a,K}^{m,\mu}\in (\varepsilon^{(m)}, 1-\varepsilon^{(m)}).
$$
Moreover, as $\mathbf{P}_m^{\mu}(\mathbf{X}^{m,\mu})\cdot \mathbf{X}^{m,\mu}=0\geq \frac{1}{2} {\mathbb E}_\T(\mathbf{X}^{m,\mu})+\tau^m {\mathbb D}_\T(\mathbf{X}^{m,\mu}) -M$, we
see that ${\mathbb E}_\T(\mathbf{X}^{m,\mu})$ and ${\mathbb D}_\T(\mathbf{X}^{m,\mu})$ are uniformly bounded in $\mu$ by $M$.
Hence, we can apply Lemmas \ref{lem.bounds.psi}, \ref{lem.bounds.anions}, \ref{lem.bounds.phialpha} to deduce that $\Vert \mathbf{X}^{m,\mu}\Vert$ is bounded uniformly in $\mu$ (for  $\mu$ sufficiently small). Finally, we can extract a subsequence, which converges to a limit denoted by $\mathbf{X}^m$ as $\mu$ tends to $0$. This limit satisfies $\mathbf{P}_m^{0}(\mathbf{X}^m)=\mathbf{P}_m(\mathbf{X}^m)=\mathbf{0}$. Thus, we have found quasi Fermi potentials which solve the discrete system \eqref{eq:discrete-mass-balance}.  It remains to show the existence of a uniquely determined $\boldsymbol{\psi}(\mathbf{X}^m)$ which solves \eqref{eq:model-poisson-discrete}. However, this follows from \Cref{lem:discrete-poisson}, which ends the proof of \Cref{thm.exresult}.

\end{proof}

\section{Numerical experiments}
\label{sec:numerics}
The numerical examples were performed with \texttt{ChargeTransport.jl}, a Julia package for the simulation of charge transport in semiconductors \cite{ChargeTransport}. In a first step the aim is to verify properties of the finite volume scheme \eqref{eq:discrete-mass-balance}-\eqref{eq:inside-Bernoulli} such as a special case of the entropy-dissipation inequality in \Cref{thm:discrete-E-D} as well as the spatial convergence rate. In a second step, the charge transport model \eqref{eq:model-dimless}-\eqref{eq:dimlessphotogen} is simulated for a physical meaningful set of parameters. In all simulation setups we are interested in the large time behavior of the model. For this reason, we introduce an entropy with respect to the steady state
\begin{equation}\label{eq:continuous-entropy-steady-state}
    \mathbb{E}_\infty(t) = \frac{\lambda^2}{2} \int_{\mathbf{\Omega}} |\nabla (\psi - \psi^\infty)|^2\,d\mathbf{x}
    + \int_{\mathbf{\Omega}_{\text{intr}}} H_a(n_a, n^\infty_a)\,d\mathbf{x}
    + \delta \sum_{\alpha \in \{n, p\}}
    \int_{\mathbf{\Omega}} H_\alpha(n_\alpha, n_\alpha^\infty)  \,d\mathbf{x},
\end{equation}
where $H_\alpha$ is defined in \eqref{eq:H-function}. The non-negative functional $E_{\infty}$ can be seen as a measure of the distance between a solution at time $t$ and the steady state of the model which vanishes, if and only if the solution at time $t$ and the steady state coincide almost everywhere. Furthermore, from an analytical point of view $E_{\infty}$ may help to prove the convergence of the discrete solution to the discrete steady state \cite{ChainaisHillairet2019}.
\subsection{Verifying the properties of the scheme}
Within this section we assume a one-dimensional domain $\Omega=(0,6) $ and set $\Omega_{\text{HTL}} = (0, 2), \Omega_{\text{intr}} = (2, 4), \Omega_{\text{ETL}} = (4, 6)$. We choose $513$ nodes per subdomain, resulting in a total number of $1537$ nodes with a grid spacing $h \approx \num{3.9e-3}$. The time domain is given by $[0, 80]$ which we discretize with a time step of $ \Delta t=\num{1.0e-1}$. We set the rescaled Debye length to $\lambda = 1$, the relative mobility of anion vacancies to $\nu = 1$, the relative concentration to $\delta = 1$, and the rescaled photogeneration rate to $\gamma = 1$.
\paragraph{Thermal Equilibrium boundary conditions}
Let us first study the implications of the assumptions in \Cref{rem:thermodynamic-eq}. To this end, we assume a constant doping $C = 0.1$ and no generation and recombination, i.e.\ $G = R = 0$. The Dirichlet functions \eqref{eq:Dirichlet-cond} are chosen as constant functions $\varphi^D = 0.5$ and $\psi^D = \text{arcsinh}(C/2) + 0.5$. The sinusoidal initial conditions for electrons, holes and the electric potential as well as the constant initial condition for anion vacancies along with the steady state solutions are depicted in \Cref{fig:test-case-1A-sol} on the left panel. On the right panel we show the steady state densities.
\begin{figure}[ht]
    \begin{subfigure}{0.49\textwidth}
        \hspace*{-0.5cm} \includegraphics[scale=0.42]{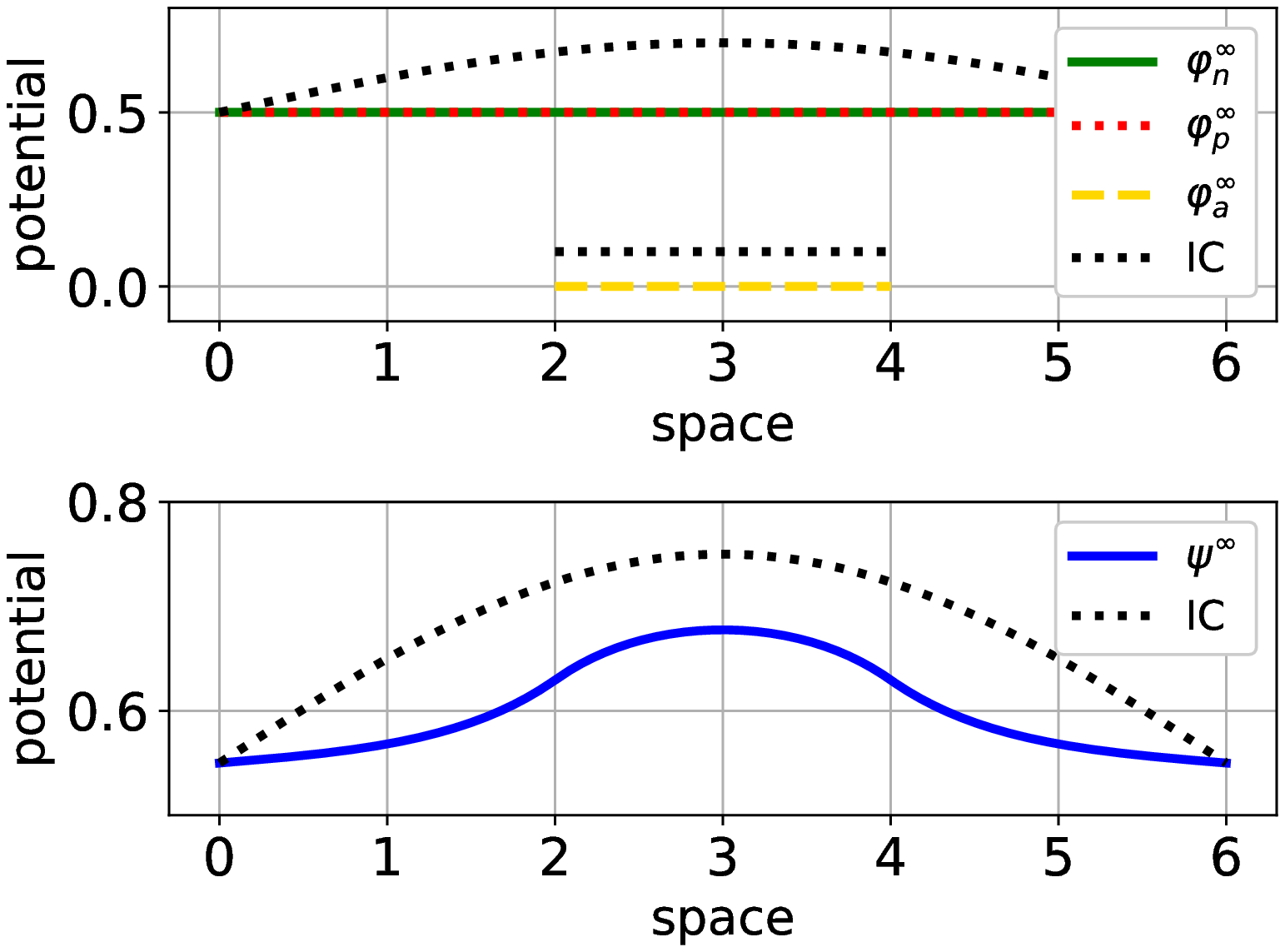}
    \end{subfigure}
    \begin{subfigure}{0.49\textwidth}
        \includegraphics[scale=0.42]{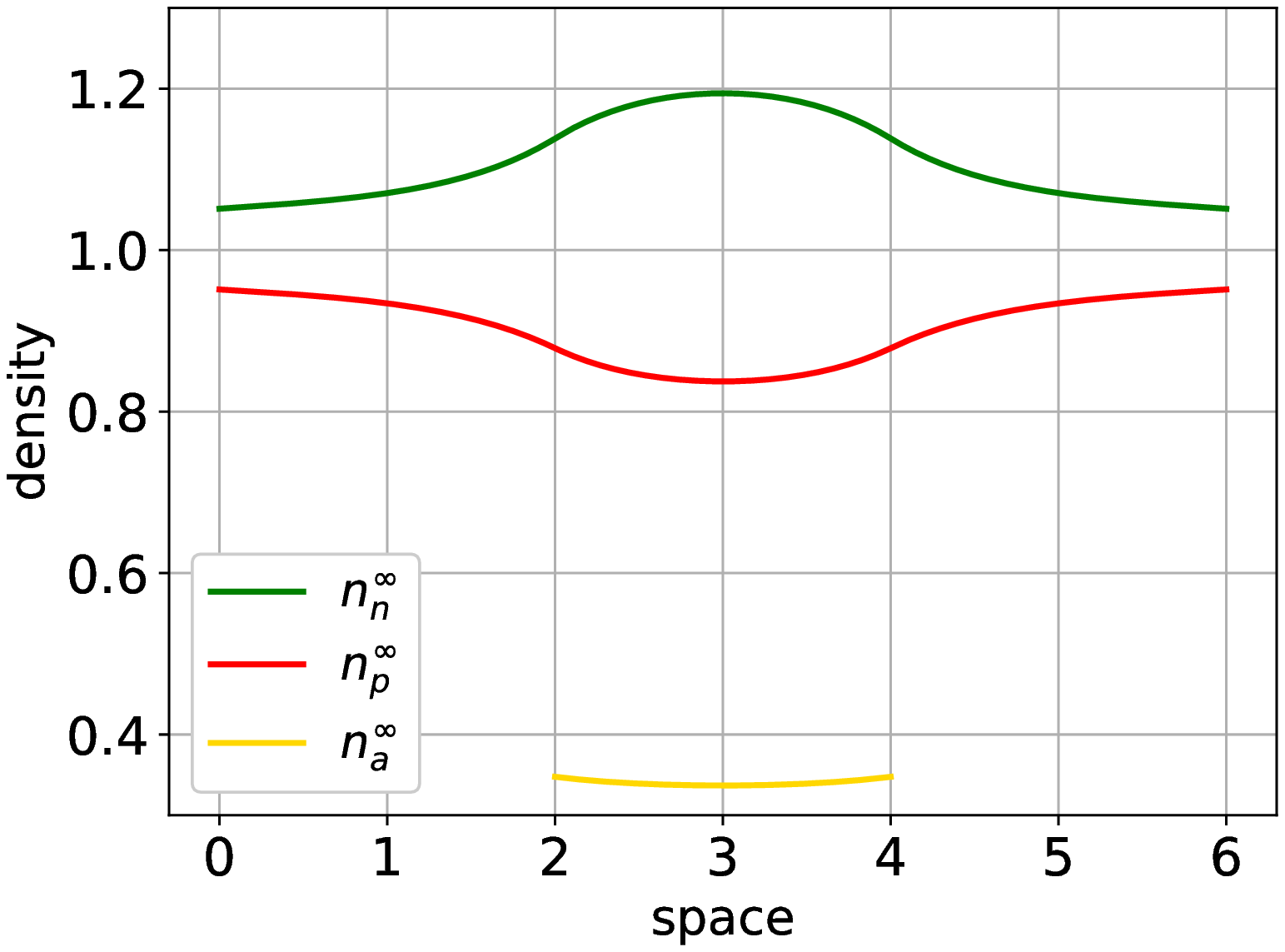}
    \end{subfigure}
    \caption{Steady state solutions $(\varphi_n^{\infty}, \varphi_p^{\infty}, \varphi_a^{\infty}, \psi^\infty)$ with the respective initial conditions as dotted lines (left) and the associated steady state densities of charge carriers (right) calculated via \eqref{eq:state-eq}.}
    \label{fig:test-case-1A-sol}
\end{figure}
Since for these specific choices, we have $ \mathbf{0} = \nabla \varphi^D = \nabla \psi^D$ and $G=0$ the discrete entropy-dissipation inequality in \Cref{thm:discrete-E-D} indicates that the relative entropy with respect to the Dirichlet boundary values \eqref{eq:discrete-entropy} does not increase in time. This result can be numerically verified, see \Cref{fig:test-case-1A-entropy}. Due to a non-constant electric potential $\psi^\infty$ we observe that the relative entropy \eqref{eq:discrete-entropy} (in blue on the left panel) levels off after an initial decrease. Furthermore, the relative entropy with respect to the steady state \eqref{eq:continuous-entropy-steady-state} (in green on the left panel of \Cref{fig:test-case-1A-entropy}) as well as the quadratic $L^2$ errors between the steady state and a solution at time $t$ (right panel) decay exponentially  with a similar slope, reaching machine precision at a similar time.
\begin{figure}[ht]
    \begin{subfigure}[b]{0.49\textwidth}
    \hspace*{-0.3cm}	 \includegraphics[scale=0.41]{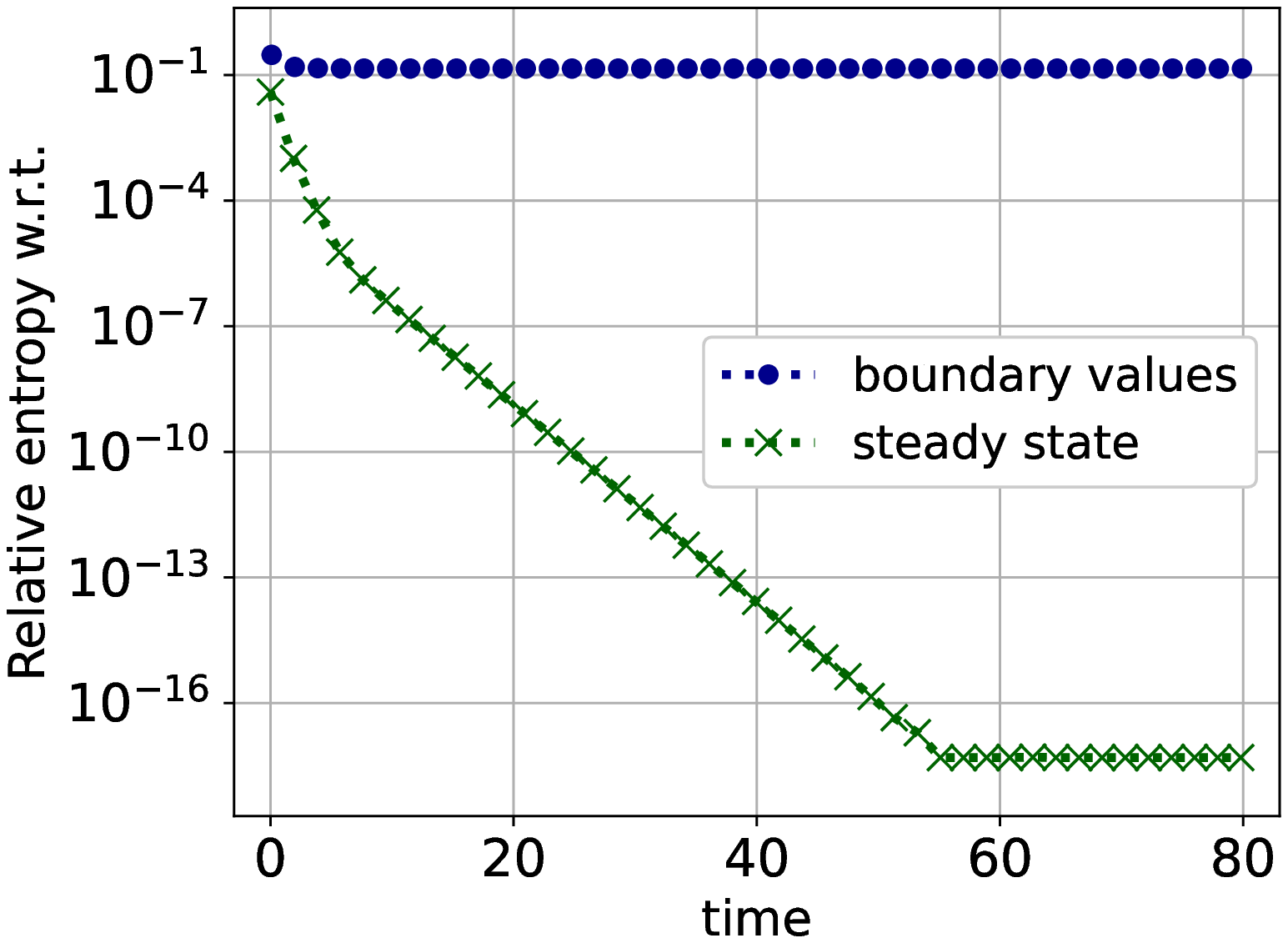}
    \end{subfigure}
    \begin{subfigure}[b]{0.49\textwidth}
        \includegraphics[scale=0.42]{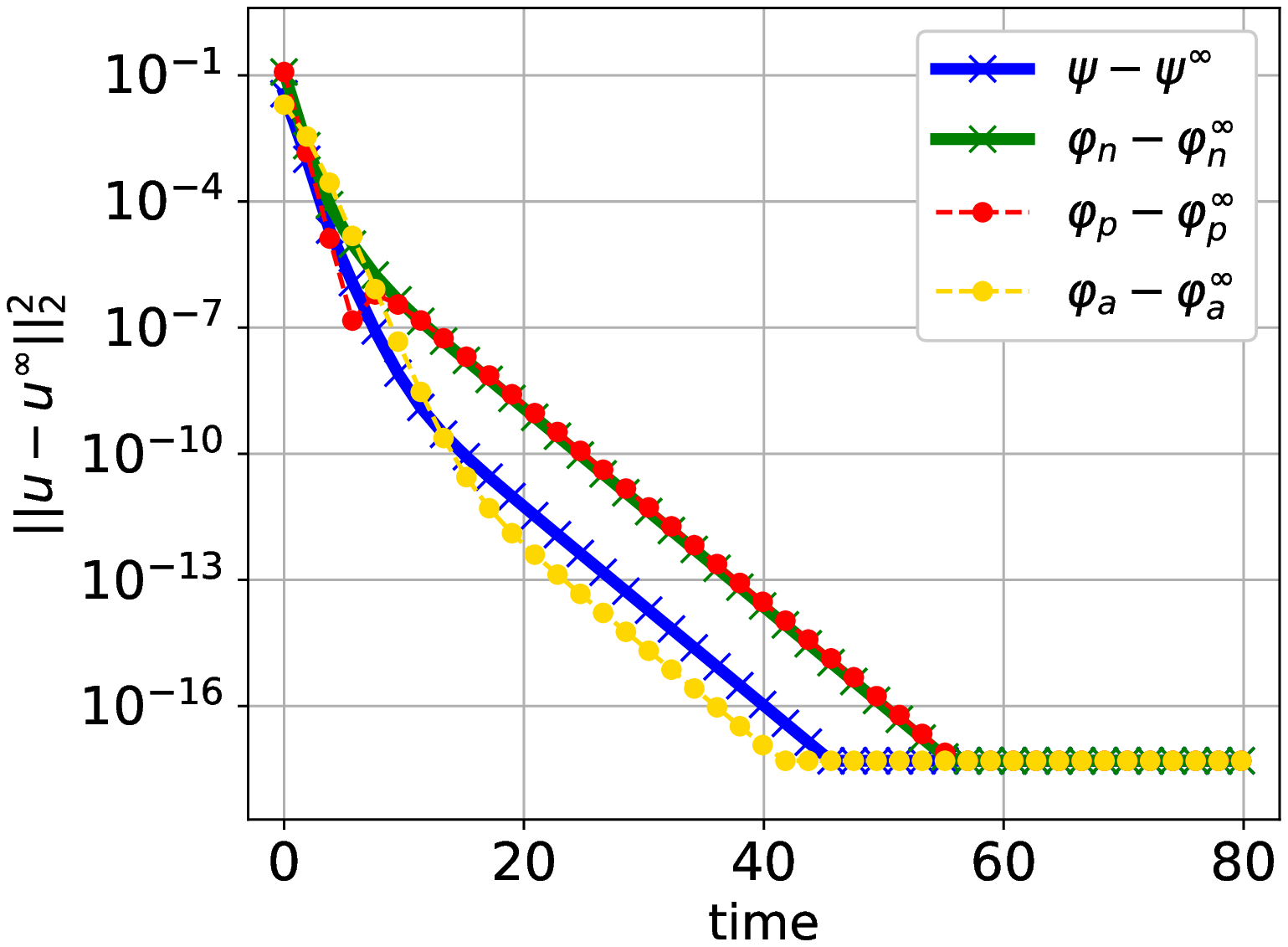}
    \end{subfigure}
    \caption{Left: Time evolution of the relative entropy with respect to the Dirichlet boundary functions \eqref{eq:continuous-entropy} as well as the relative entropy with respect to the steady state \eqref{eq:continuous-entropy-steady-state}. Right: Time evolution of the quadratic $L^2$ errors between the computed and the steady state solutions. }
    \label{fig:test-case-1A-entropy}
\end{figure}
\paragraph{Non-constant boundary values}
Next, we adjust the doping and the boundary values. Let us assume that the doping $C$ is a piecewise constant function given by $0.5$ in $\Omega_{\text{ETL}}$ and by $-0.5$ in $\Omega_{\text{HTL}} \cup \Omega_{\text{intr}}$. The boundary values are set to $\varphi|_{x = 0}^D = 1, \varphi|_{x = 6}^D = 0$, $\quad \psi|_{x = 0}^D = \text{arcsinh}(- 0.5/2) + 1, \psi|_{x = 6}^D = \text{arcsinh}(0.5/2)$. We choose quadratic initial conditions for $\varphi_n, \varphi_p, \psi$ and a constant initial condition for $\varphi_a$. The initial conditions are additionally to the steady state solutions depicted in \Cref{fig:test-case-1B-sol} as dotted lines.
\begin{figure}[ht]
    \begin{subfigure}{0.49\textwidth}
        \hspace*{-0.5cm} \includegraphics[scale=0.42]{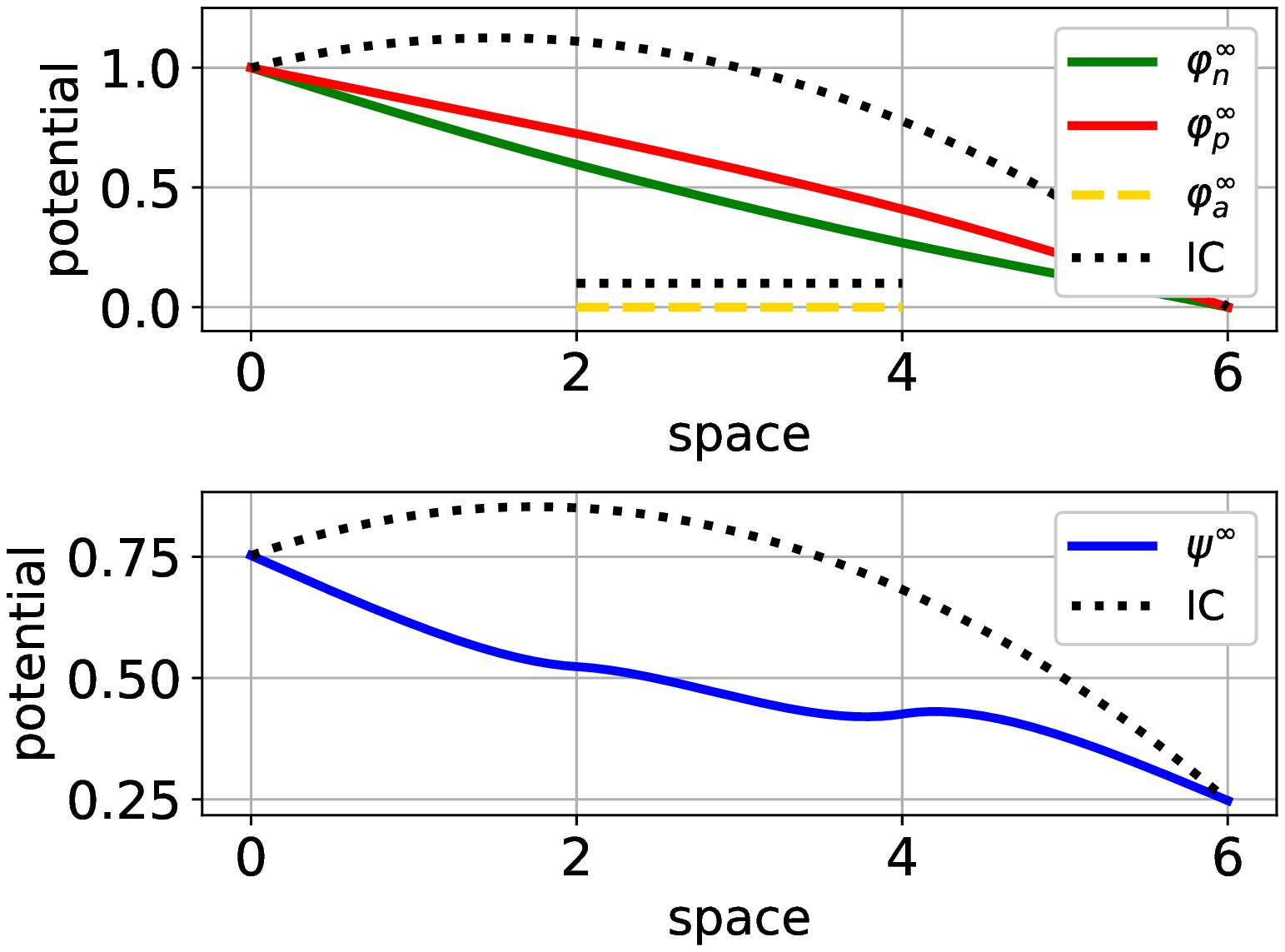}
    \end{subfigure}
    \begin{subfigure}{0.49\textwidth}
        \includegraphics[scale=0.42]{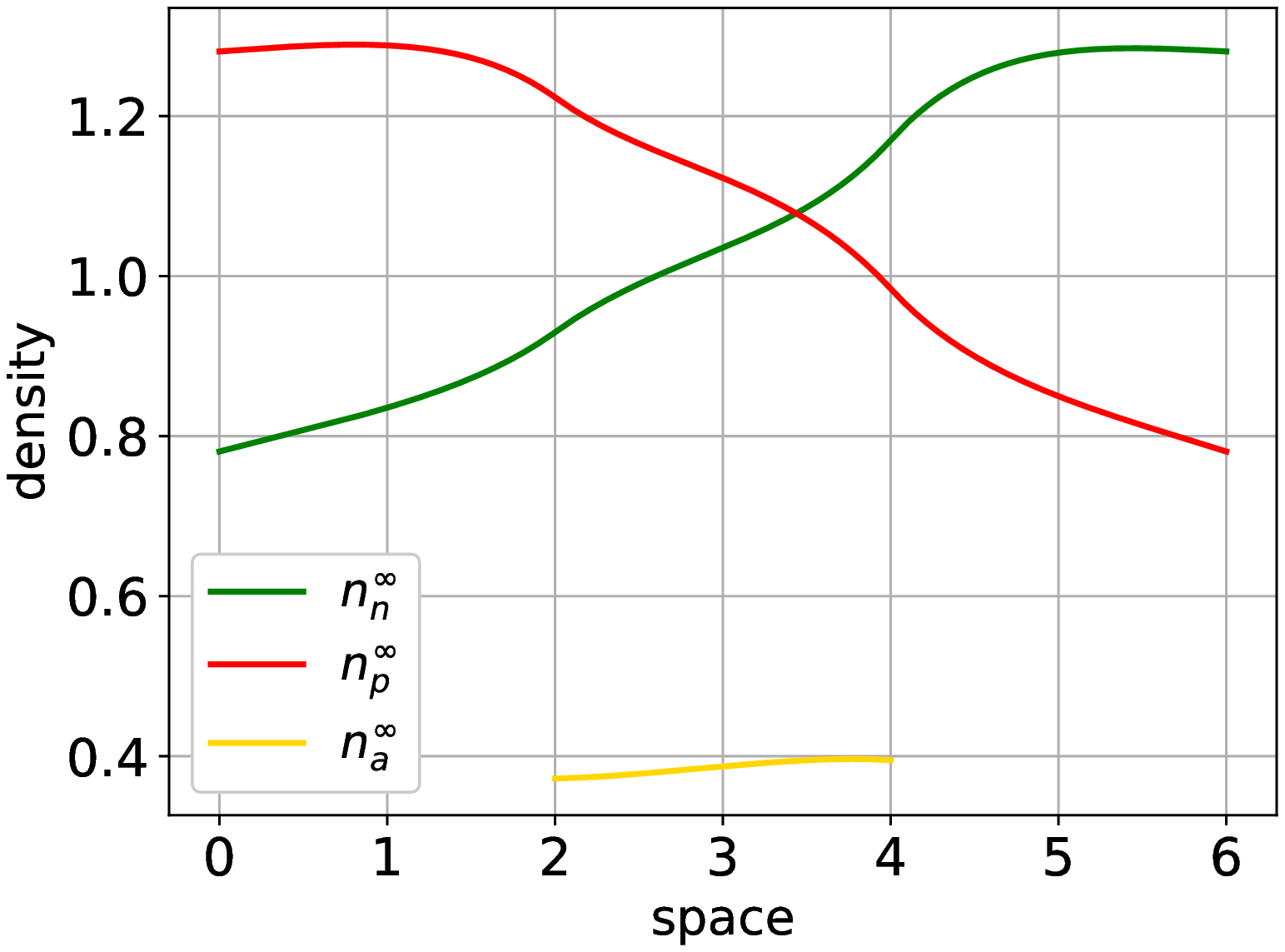}
    \end{subfigure}
    \caption{Steady state potentials with the corresponding initial conditions (left) and the associated steady state densities of charge carriers (right) calculated via the state equation \eqref{eq:state-eq}.}
    \label{fig:test-case-1B-sol}
\end{figure}
Again, the relative entropy with respect to the steady state and the quadratic $L^2$ errors decay exponentially and reach machine precision with a similar slope, see \Cref{fig:test-case-1B-entropy} in the left and middle panel.

Finally, we complete this section with an investigation of the spatial convergence behavior. Suppose $n_{*}\in\mathbb{N}$ is given, then $n = 2 \cdot 2^{n_{*} -1} + 1$ nodes are chosen in each of the three subdomain, i.e.\ in total we have $n_{\text{tot}} = 3 \cdot 2 \cdot 2^{n_{*} -1} + 1 $ nodes.
We calculate a reference solution on a grid with $n_{*} = 9$ corresponding to $1537$ nodes with a grid spacing $h_* \approx \num{3.9e-3}$. The $L^2$ errors between the solution $u_{n_*}$, for $n_{*} = 2,\ldots,8$, and the reference solution projected onto the coarser mesh evaluated at the final time are shown in \Cref{fig:test-case-1B-entropy}, right panel. Since for the final time $t_F=80$ the system is already within machine precision of the steady state, the error shown is purely due to the spatial discretization. We observe second order experimental convergence.
\begin{figure}[ht]
    \begin{subfigure}[b]{0.33\textwidth}
    \hspace*{-0.3cm}	 \includegraphics[scale=0.31]{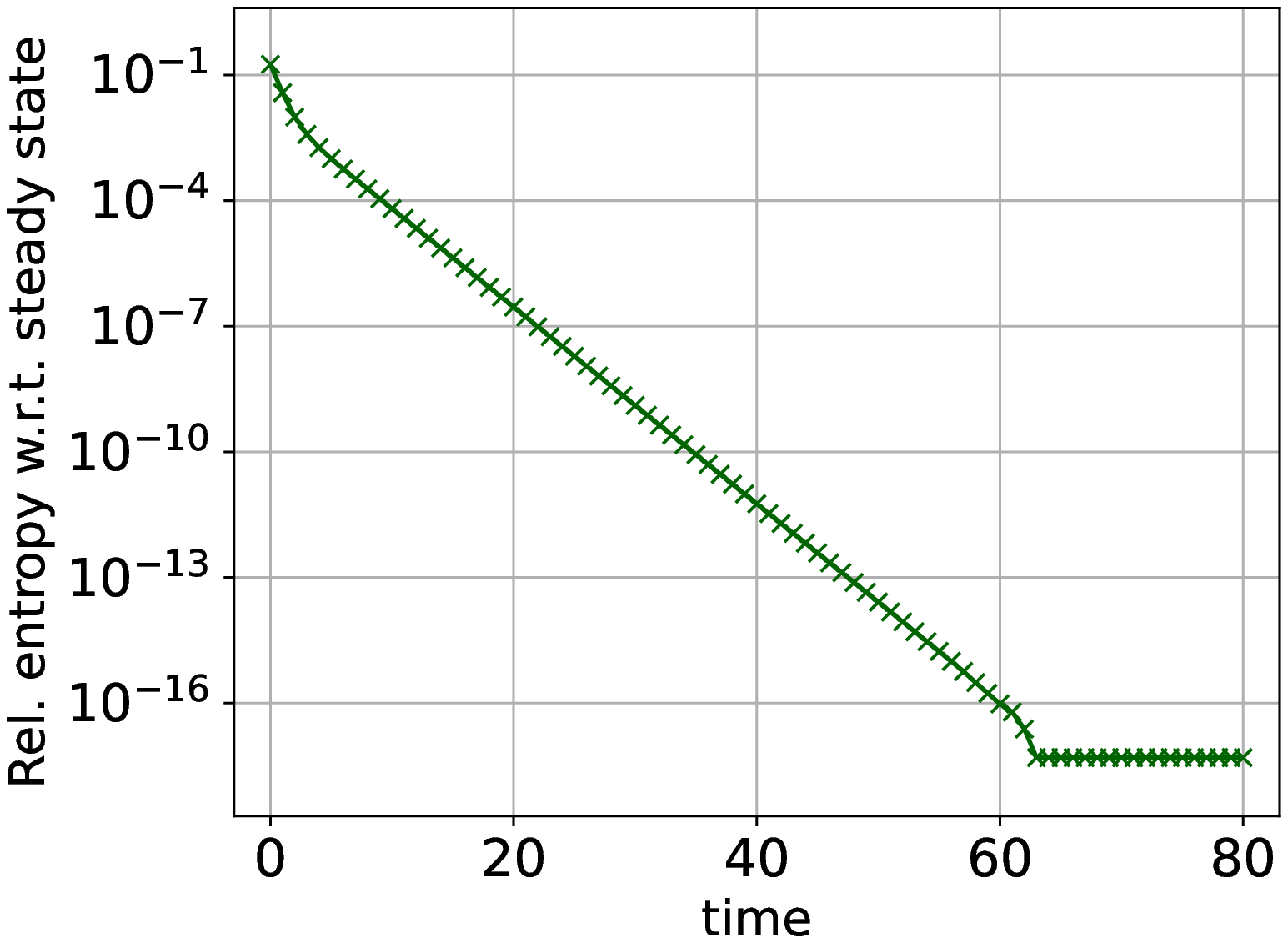}
    \end{subfigure}
    \begin{subfigure}[b]{0.33\textwidth}
        \includegraphics[scale=0.30]{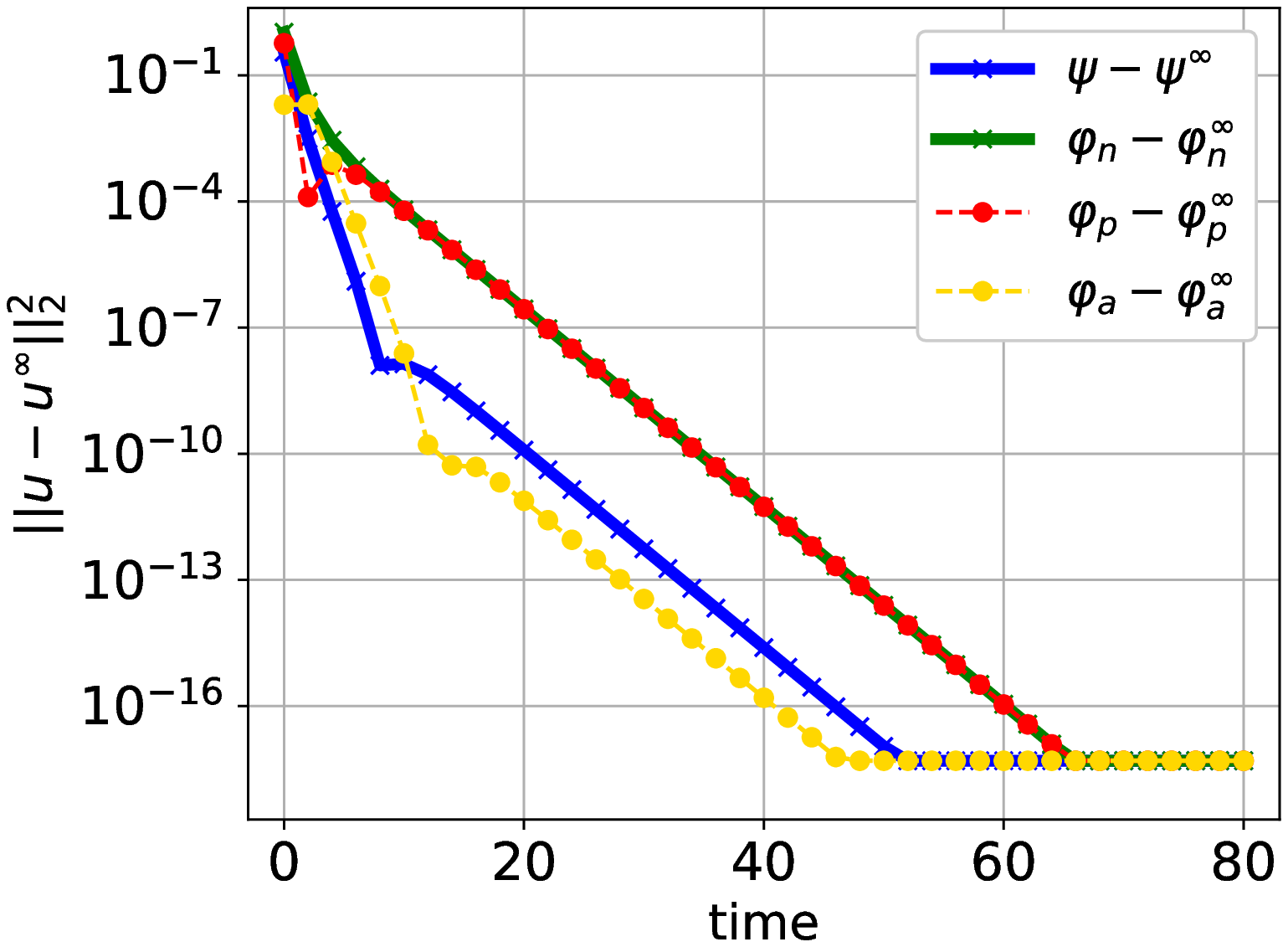}
    \end{subfigure}
    \begin{subfigure}[b]{0.33\textwidth}
        \includegraphics[scale=0.30]{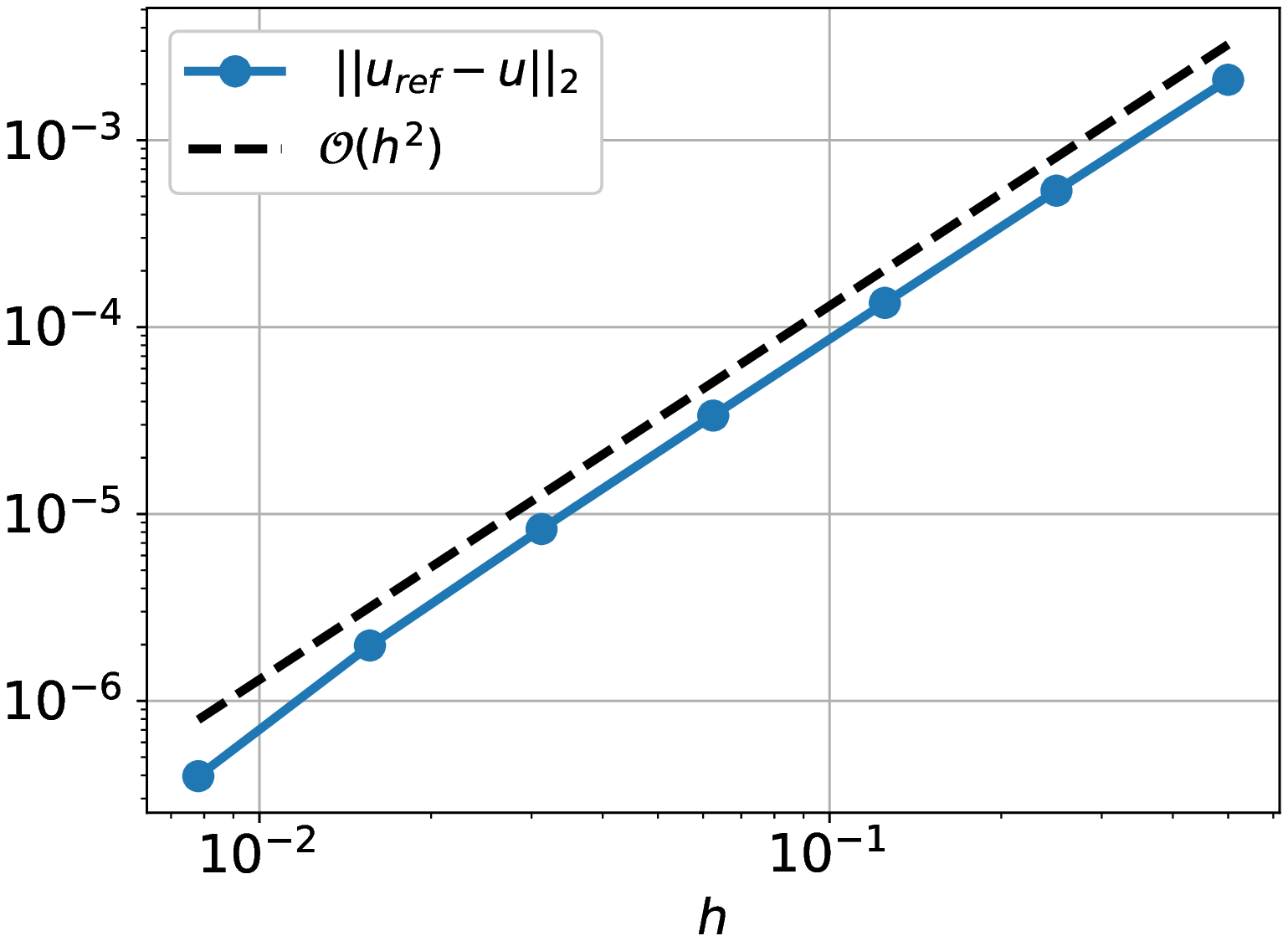}
    \end{subfigure}
    \caption{Time evolution of the relative entropy with respect to steady state \eqref{eq:continuous-entropy-steady-state} for non-constant boundary values (left) and of the quadratic $L^2$ errors between steady state and solutions at time $t$ (middle). On the right, the $L^2$ error with respect to the grid spacing $h$ is shown.}
    \label{fig:test-case-1B-entropy}
\end{figure}
\subsection{PSC simulation setup}\label{sec:pscsetup}
In the final simulation setup, we choose the rescaling factors and non-dimensionalized parameters in such a way that the resulting solutions correspond to a realistic PSC device. All parameters are chosen in agreement with \Cref{sec:nondimens-model} and with non-zero band-edge energies. Apart from the additional non-scaled parameters $N_a = \num{1.0e21} \, \text{cm}^{-3}$ and $E_a = -4.45 \; \text{eV}$, we used the parameter set provided in \cite{Courtier2019a}. The mesh is given by $385$ nodes with a uniform grid spacing in each layer, namely $h_{\text{ETL}} \approx \num{7.8e-8}$cm, $h_{\text{HTL}} \approx \num{1.6e-7}$cm in the transport layers and $h_{\text{intr}} \approx \num{3.1e-7}$cm in the perovskite layer. The uniform time mesh is built with a step size of $ \Delta t = 0.5\, \text{s}$ and the final time is given by $t_F=220 \,\text{s}$.
\begin{figure}[ht]
    \begin{subfigure}{0.49\textwidth}
        \includegraphics[scale=0.7]{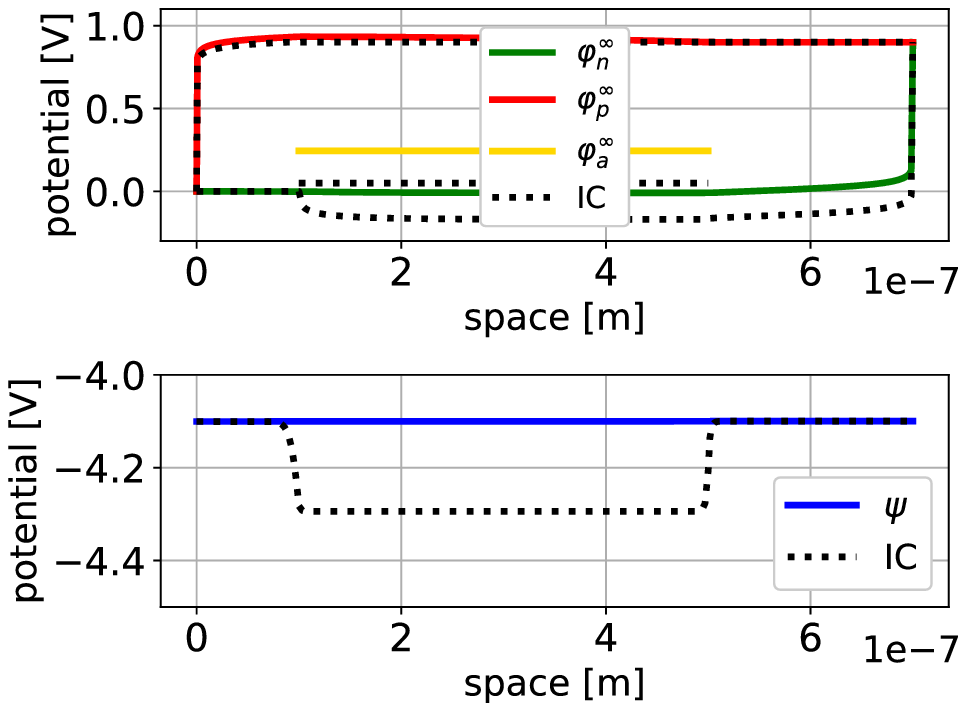}
    \end{subfigure}
    \begin{subfigure}{0.49\textwidth}
        \includegraphics[scale=0.63]{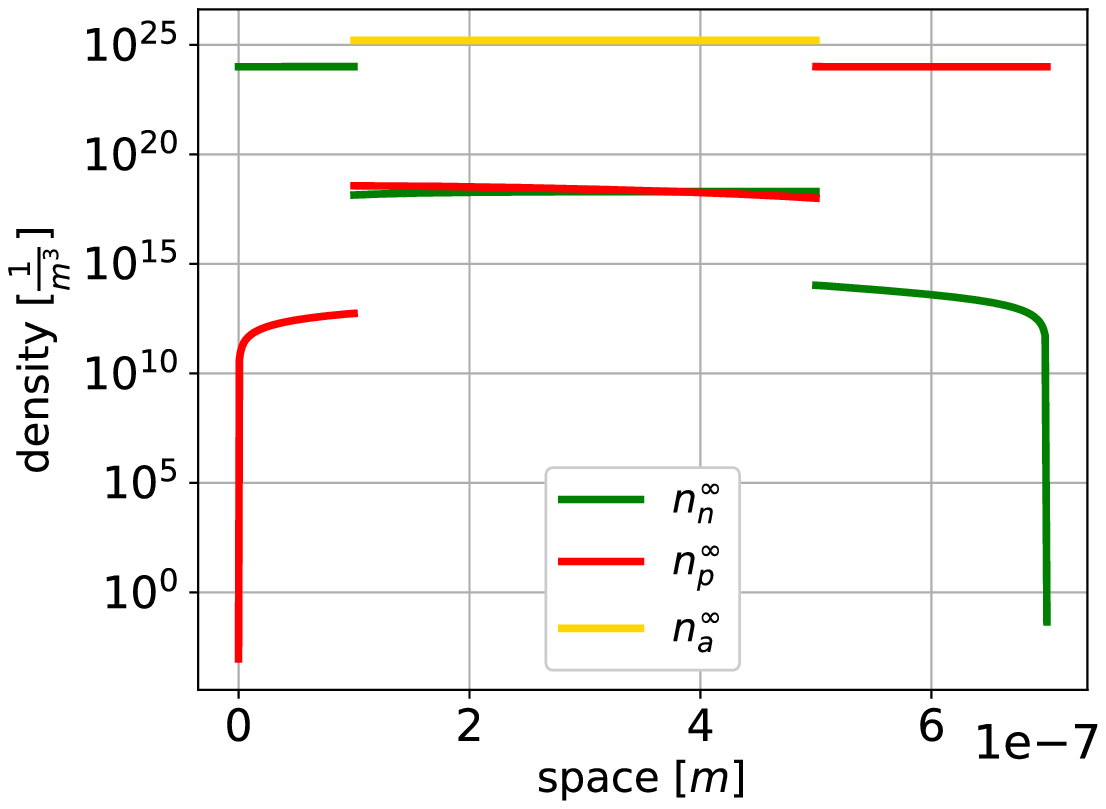}
    \end{subfigure}
    \caption{
        Steady state potentials with the respective initial conditions (left) and the associated densities of charge carriers (right) for a PSC three-layer device at an applied voltage $0.9 \, \text{V}$.}
    \label{fig:test-case-2-sol}
\end{figure}
Usually a PSC device is held for several seconds at a constant voltage, ensuring that ionic charges equilibrate. This procedure is often called \textit{preconditioning protocol} \cite{Courtier2019}. Afterwards, scan protocols with a time-dependent applied voltage --  incorporated via time-dependent Dirichlet boundary conditions -- are performed to study the device physics.
Thus, the steady potentials and their respective densities depicted in \Cref{fig:test-case-2-sol} can be regarded as the solutions after a successful preconditioning scan. Within the presented configuration the applied voltage is chosen such that the steady state electric potential $\psi^\infty$ is constant which can be observed well in \Cref{fig:test-case-2-sol}.
The depicted initial conditions correspond to a solution of the charge transport model with a non-constant vacancy concentration. As before, we consider the large time behavior of the quadratic $L^2$ errors and the relative entropy with respect to the steady state. Taking the thermodynamic free energy \eqref{eq:physical-free-energy} into account, we can reformulate the dimensional relative entropy with respect to the steady state \eqref{eq:continuous-entropy-steady-state}
\begin{equation}\label{eq:free-energy-steady-state}
    \mathbb{E}_\infty(t) = \frac12 \int_{\mathbf{\Omega}}  \varepsilon_s  |\nabla (\psi - \psi^\infty)|^2\,d\mathbf{x}
    + \int_{\mathbf{\Omega}_{\text{intr}}} H_a(n_a, n^\infty_a)\,d\mathbf{x}
    + \sum_{\alpha \in \{n, p\}}
    \int_{\mathbf{\Omega}} H_\alpha(n_\alpha, n_\alpha^\infty)  \,d\mathbf{x},
\end{equation}
where $H_\alpha(x,y) = \Phi_\alpha(x) - \Phi_\alpha(y) - \Phi_\alpha'(y)(x-y)$, as defined in \eqref{eq:H-function}, but with
\begin{align} \label{eq:free-entropy-Phi}
    \Phi_n(x) =  k_B T x &\left(  \log\left(\frac{x}{N_n} \right)  - 1 \right) - z_n E_n  x, \quad \Phi_p(x) =  k_B T x \left(  \log\left(\frac{x}{N_p} \right)  - 1 \right) - z_p E_p x, \quad  \\
     ~
    \Phi_a(x) &=   k_B T  \left(  x \log \left( \frac{x}{N_a}\right) + \left( N_a -x \right) \log \left( 1- \frac{x}{N_a} \right) \right)- z_aE_a x,
\end{align}
i.e. we extend the contributions of the relative entropy with respect to the steady state such that they are consistent with the thermodynamic free energy \eqref{eq:physical-free-energy}. As before, \Cref{fig:test-case-2-entropy-with-gen} indicates an exponential decay towards zero of the relative entropy \eqref{eq:free-energy-steady-state} as well as of the quadratic $L^2$ errors with respect to time.
In contrast to the observations made in previous section, the relative entropy with respect to the steady state \eqref{eq:free-energy-steady-state} vanishes faster than the quadratic $L^2$ errors. This may be explained by the additional terms in \eqref{eq:free-entropy-Phi} due to non-zero band-edge energies which influence the convergence behavior. Still, we see a similar convergence rate of the two introduced measures for the deviation of a solution at time $t$ from the steady state.
\begin{figure}[ht]
    \begin{subfigure}[b]{0.49\textwidth}
        \includegraphics[scale=0.42]{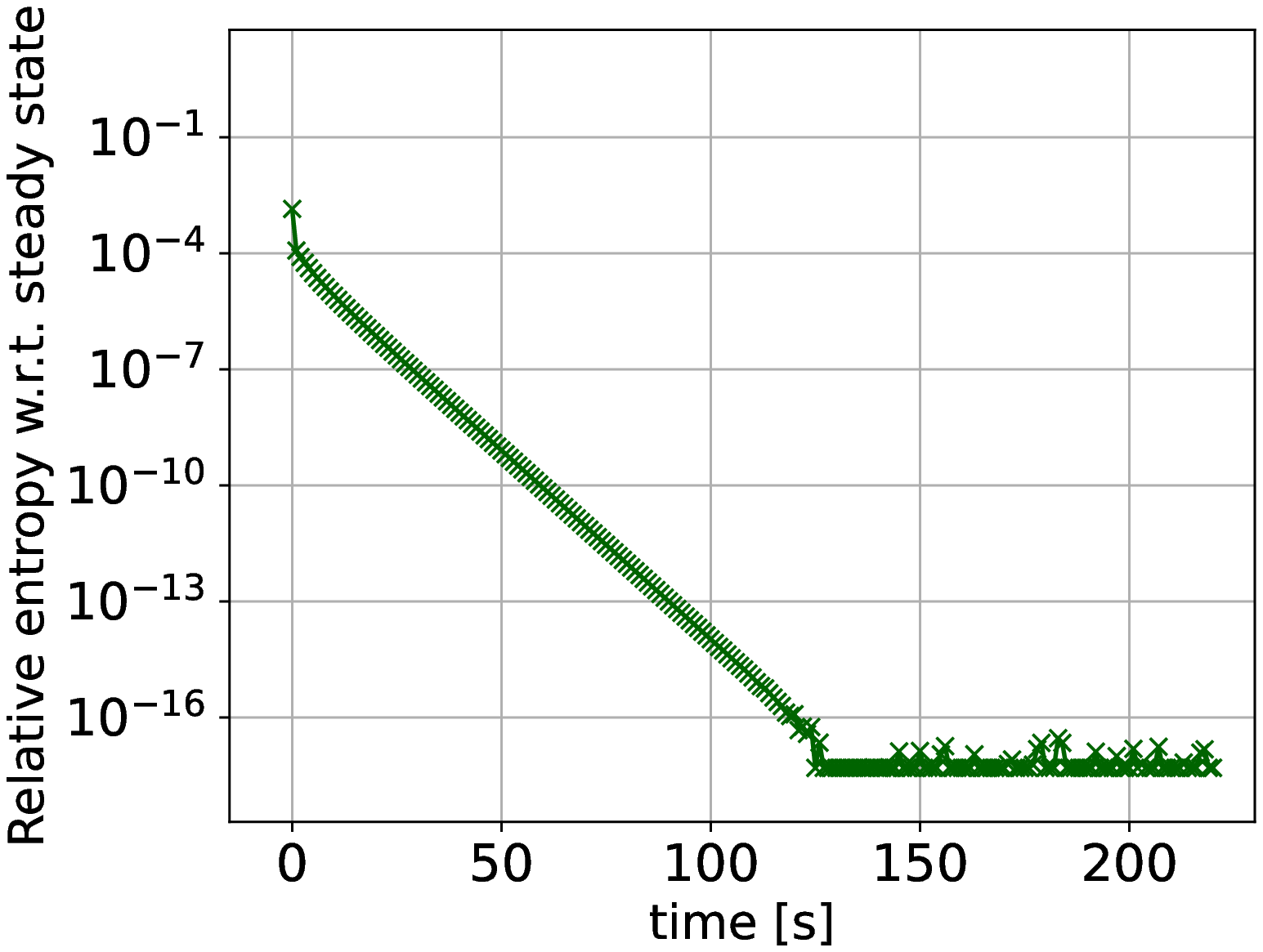}
    \end{subfigure}
    \begin{subfigure}[b]{0.49\textwidth}
        \includegraphics[scale=0.42]{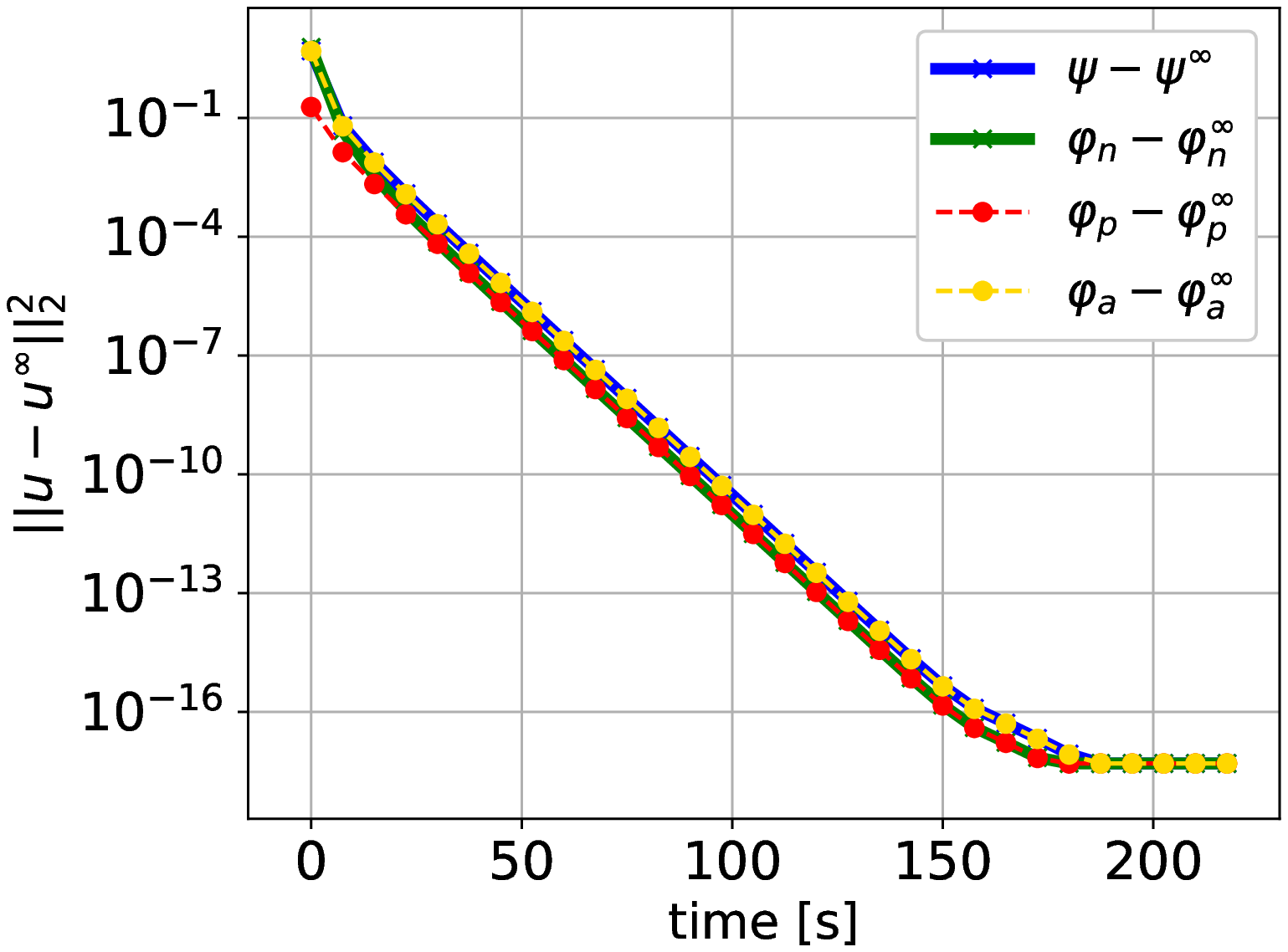}
    \end{subfigure}
    \caption{Time evolution of the relative entropy with respect to the steady state \eqref{eq:free-energy-steady-state} (left) and of the quadratic $L^2$ errors between the computed and the steady state solutions (right).}
    \label{fig:test-case-2-entropy-with-gen}
\end{figure}
\section{Conclusion and outlook}
\label{sec:conclusion}
For a charge transport model for perovskite solar cells, we discussed and proved a continuous entropy-dissipation inequality. We allowed general statistics function for the electric charge carriers. Moreover, we proved an analogous entropy-dissipation inequality for a finite volume scheme based on the excess chemical potential flux. The entropy-dissipation inequality helped us to prove the existence of a discrete solution at every time step. Furthermore, for a model in thermodynamic equilibrium we proved the decay of the continuous and discrete relative entropy with respect to the boundary conditions and numerically verified this result. A spatial convergence of order $2$ was also numerically shown.
In the last experiment, we studied the numerical convergence towards the steady state for a setup which can be physically interpreted as preconditioning a PSC device before applying a measurement protocol. Especially the relative entropy with respect to the steady state for non-zero band-edge energies decays exponentially towards zero. Studying the model behavior with non-zero and even irregular band-edge energies is from a mathematical and physical point of view of interest in the future. Also time-dependent Dirichlet functions coinciding with physically realistic measurement techniques for perovskite solar cells can be a topic of future research. Finally, deriving reduced models by ignoring small dimensionless parameters could also be investigated.

\vspace{1em}
\textit{Acknowledgments} This work was partially supported by the Leibniz competition as well as the CEMPI (ANR-11-LABX-0007) and the scientific department of the French Embassy in Germany.
\appendix
\section{Estimates on statistics and entropy functions} \label{app:estimates-statistics}
In this section we deal with the estimates concerning statistics and entropy functions. First, we provide the proof of Lemma~\ref{lemma:bound-primitive-argument} and Lemma~\ref{lemma:relation-inverse-primitive} which are stated under abstract assumptions \eqref{hyp:statistics-n-p}, \eqref{hyp:statistics-a}  and/or \eqref{hyp:limit-H-finverse} on the statistics. Then, we show that the examples of Fermi-Dirac and Boltzmann statistics \eqref{eq:FermiDirac}, \eqref{eq:Boltzmann} and \eqref{eq:Fermi--1} satisfy these  assumptions.
\subsection{Proof of Lemma~\ref{lemma:bound-primitive-argument}}
Let us show point (i) of the Lemma. Let $\mathcal{F}_\alpha$  with $\alpha = n, p$ be a statistics function satisfying \eqref{hyp:statistics-n-p} and $H_\alpha$ be the associated relative entropy function. Let $\varepsilon > 0$ and $y_0 \geq 0$. For $x\geq0$ and $y\in[0,y_0]$ one has
    \begin{align*}
            x \leq \sup_{x \in \mathbb{R}} \left(x - \varepsilon H_\alpha(x,y)\right) + \varepsilon H_\alpha(x,y).
    \end{align*}
    The first term on the right-hand side is the Legendre transform of $x \mapsto \varepsilon H_\alpha(x,y)$ evaluated at $1$. It is exactly given by $\bar{x} - \varepsilon H_\alpha(\bar{x}, y)$ with $\bar{x} = \mathcal{F}_\alpha \left( \frac{1}{\varepsilon} + \mathcal{F}_\alpha^{-1}(y)\right)$. In turn, one has
    \begin{align*}
        x \leq c_{y_0, \varepsilon} + \varepsilon H_\alpha(x,y), \quad  \text{for all}\; x\geq0,\  y\in[0,y_0],
    \end{align*}
    for $c_{y_0, \varepsilon} =  \mathcal{F}_\alpha \left( \frac{1}{\varepsilon} + \mathcal{F}_\alpha^{-1}(y_0)\right)$, since $\mathcal{F}_\alpha$ is increasing.

    For point (ii), where $\mathcal{F}_a$ is a statistics function satisfying \eqref{hyp:statistics-a} and $\Phi_a$ is the associated entropy function, an analogous calculation will prove the estimate with $c_{\varepsilon} = \mathcal{F}_a(\frac{1}{\varepsilon})$
    \begin{align*}
        x \leq c_\varepsilon + \varepsilon \Phi_a(x), \quad \text{for all}\; x\geq 0.
    \end{align*}
\subsection{Proof of Lemma~\ref{lemma:relation-inverse-primitive}}
Now, let us consider $\mathcal{F}_\alpha$ with $\alpha = n, p$ as a statistics function satisfying \eqref{hyp:statistics-n-p} and \eqref{hyp:limit-H-finverse}. Let $\varepsilon > 0$ and $y_0 \geq 0$. On the one hand, because of \eqref{hyp:limit-H-finverse} there exists $x_0 = x_0(y_0, \varepsilon)\geq 0$ such that
\begin{align*}
        0\leq \mathcal{F}_\alpha^{-1}(x)
        \leq \varepsilon H_\alpha(x,y_0), \quad \text{for all}\; x\geq x_0.
\end{align*}
Further, the calculation $\partial_y H_\alpha(x, y) = - \Phi_\alpha''(y) (x-y)$ reveals that $y \mapsto H_\alpha(x,y)$ is non-increasing for all $y \in [0, y_0], y \leq x$, due to the convexity of $\Phi_\alpha$. Hence,
\begin{align*}
        0\leq \mathcal{F}_\alpha^{-1}(x)
        \leq \varepsilon H_\alpha(x,y)\leq \varepsilon H_\alpha(x,y_0), \quad \text{for all}\; y \in [0, y_0], \; x\geq \max \{x_0, y_0\} =: \bar{x}.
\end{align*}
On the other hand, for $0\leq x\leq \bar{x}$
\begin{align*}
        \mathcal{F}_\alpha^{-1}(x)
        &\leq  \mathcal{F}_\alpha^{-1}(\bar{x}).
\end{align*}
Hence, in total the claim is proven with  $c_{y_0, \varepsilon} =  \mathcal{F}_\alpha^{-1} (\bar{x}) $
\begin{align*}
        \max(\mathcal{F}_\alpha^{-1}(x),0)
        \leq c_{y_0, \varepsilon} + \varepsilon H_\alpha(x,y), \quad \text{for all}\; x \geq 0, \quad y \in [0, y_0].
\end{align*}
\subsection{Boltzmann and Fermi-Dirac entropy functions}
We now relate the  statistics \eqref{eq:FermiDirac}, \eqref{eq:Boltzmann} to the hypotheses \eqref{hyp:statistics-n-p}  and \eqref{hyp:limit-H-finverse}, and the statistics \eqref{eq:Fermi--1} to the hypothesis \eqref{hyp:statistics-a}. We only give a proof in the case of the Fermi-Dirac statistics of order $1/2$ since the results are essentially trivial for the other statistics.
\begin{lemma}[Boltzmann]
Assume that the statistics for electrons and holes is the Boltzmann statistics \eqref{eq:Boltzmann}, namely
\[
\mathcal{F}_\alpha(\eta) = F_B(\eta) = e^\eta\,,\quad\eta\in\mathbb{R}\,,\ \alpha \in\{n,p\}.
\]
Then, $\mathcal{F}_\alpha$ satisfies \eqref{hyp:statistics-n-p} and \eqref{hyp:limit-H-finverse}.
\end{lemma}
\begin{lemma}[Fermi-Dirac of order $-1$]
Assume that the statistics for anion vacancies is the Fermi-Dirac integral of order $-1$ in \eqref{eq:Fermi--1}, namely
\[
\mathcal{F}_a(\eta) = F_{-1}(\eta) =  \frac{1}{\exp(-\eta)+1} \,,\quad\eta\in\mathbb{R}\,.
\]
Then, $\mathcal{F}_a$ satisfies \eqref{hyp:statistics-a}.
\end{lemma}

\begin{lemma}[Fermi-Dirac of order $1/2$]
Assume that the statistics for electrons and holes is the Fermi-Dirac integral of order $1/2$ in \eqref{eq:FermiDirac}, namely
\[
\mathcal{F}_\alpha(\eta) = F_{1/2}(\eta) =  \frac{2}{ \sqrt{\pi} } \int_{0}^{\infty} \frac{\xi^{1/2}}{\exp(\xi - \eta) + 1}\: \mathrm{d}\xi\,,\quad\eta\in\mathbb{R}\,,\ \alpha \in\{n,p\}.
\]
Then, $\mathcal{F}_\alpha$ satisfies \eqref{hyp:statistics-n-p} and \eqref{hyp:limit-H-finverse}.
\end{lemma}
\begin{proof}
First, observe that $\mathcal{F}_\alpha$ is smooth and strictly increasing with limits $0$ and $+\infty$, when $\eta\to-\infty$ and $\eta\to+\infty$ respectively. Then,
\[
F_{1/2}(\eta)\exp(-\eta) = \frac{2}{ \sqrt{\pi} } \int_{0}^{\infty} \frac{\xi^{1/2}}{\exp(\xi) + \exp(\eta)}\: \mathrm{d}\xi \leq \frac{2}{ \sqrt{\pi} }\int_{0}^{\infty}\xi^{1/2}\exp(-\xi)\: \mathrm{d}\xi = 1.
\]
Moreover,
\[
F_{1/2}'(\eta) = \frac{2}{ \sqrt{\pi} } \int_{0}^{\infty} \frac{\xi^{1/2}\exp(\xi - \eta)}{(\exp(\xi - \eta) + 1)^2}\: \mathrm{d}\xi \leq  F_{1/2}(\eta).
\]
This proves \eqref{hyp:statistics-n-p}. Now let us focus on the behavior at infinity of $F_{1/2}$. We claim the existence of constants $c_1,c_2>0$, such that
\begin{align}
        c_1\eta^{3/2}\leq F_{1/2}(\eta) \leq c_2\eta^{3/2}, \quad \text{for}\; \eta \geq 1.
        \label{eq:F_j_inequality}
\end{align}
With \eqref{eq:F_j_inequality} we can conclude $F_{1/2}^{-1}(s) = \mathcal{O}(s^{2/3})$ for $s\to \infty$. Therefore, the associated entropy function behaves like $\mathcal{O}(s^{5/3})$ and  \eqref{hyp:limit-H-finverse} readily follows. To see that \eqref{eq:F_j_inequality} is indeed satisfied, let us consider \eqref{eq:FermiDirac} on the two intervals $[0, \eta]$ and $[\eta, \infty)$, where the respective integrals are denoted by $I_1, I_2$ and substitute $z = \xi - \eta$. This yields $ F_{1/2}(\eta) = \frac{2}{\sqrt{\pi}}(I_1+I_2)$ with
$$
I_1=\int_{0}^{\eta} \frac{\xi^{1/2}}{\exp \left(\xi -\eta \right)+1}\,d\xi \quad \mbox{ and } \quad I_2=\int_0^\infty \frac{(z+\eta)^{1/2}}{\exp z+1}\,dz.
$$
We bound $I_1$ and $I_2$ separately. On the one hand, since $0\leq \exp \left(\xi - \eta \right) \leq 1$ for $\xi \leq \eta$, we obtain
\begin{align*}
        \frac{1}{3}\eta^{3/2}
        \leq I_1
        \leq \frac{2}{3}\eta^{3/2}.
\end{align*}
On the other hand, we split $I_2$ into an integral over $[0,\eta]$ and one over $[\eta,+\infty)$ and bound each term
$$
 I_2 =\int_0^{\eta} \frac{(z+\eta)^{1/2}}{\exp z+1}\,dz
                +
                \int_{\eta}^\infty \frac{(z+\eta)^{1/2}}{\exp z+1}\,dz
                \leq 2^{1/2}\eta^{1/2} \int_0^{\eta} \frac{1}{\exp z+1}\,dz
                +
                2^{1/2} \int_{\eta}^\infty \frac{z^{1/2}}{\exp z+1}\,dz.
$$
But,
$$
 \int_0^{\eta} \frac{1}{\exp z+1}\,dz\leq  \int_0^{\eta} \exp(-z) dz\leq 1 \quad \mbox{ and } \quad  \int_{\eta}^\infty \frac{z^{1/2}}{\exp z+1}\,dz \leq \int_{0}^\infty \frac{z^{1/2}}{\exp z}\,dz = \Gamma(3/2),
$$
        where $\Gamma$ is the Euler's Gamma function, satisfying $\Gamma(3/2)=\frac{\sqrt{\pi}}{2}$. Hence, assuming $\eta \geq 1$, we receive for $I_2$
        \begin{align*}
                0 \leq I_2 \leq \sqrt{2} \Bigl( \eta^{1/2} + \frac{\sqrt{\pi}}{2} \Bigr) \leq \sqrt{2} \eta^{3/2} \Bigl(1 + \frac{\sqrt{\pi}}{2}\Bigr).
        \end{align*}
        And the claim in \eqref{eq:F_j_inequality} is shown with the constants $c_1 = \frac{2}{ 3\sqrt{\pi}} $ and $c_2= \frac{2}{\sqrt{\pi}}\left(\frac{2}{3} + \sqrt{2} \Bigl(1+\frac{\sqrt{\pi}}{2}\Bigr)\right)$.
    \end{proof}


\section{A technical result}  \label{app:technical-proof}
In this section, we establish a technical result, stated in Lemma \ref{lem:limupsilon}, which is crucial for the proof of bounds satisfied by the quasi Fermi potentials of electrons and holes, see Lemma \ref{lem.bounds.phialpha}.
\begin{lemma}\label{lem:limupsilon}
Assume that the statistics function $\mathcal F_\alpha$ satisfies the hypothesis \eqref{hyp:statistics-n-p}. Let us define the functions ${\mathcal K}_\alpha:(x,a)\in \R^2\mapsto {\mathcal K}_\alpha(x,a)\in \R$ and
${\mathcal D}_\alpha : (x,y,a,b)\in\R^4 \mapsto {\mathcal D}_\alpha(x,y,a,b)\in \R$ by
$$
\begin{aligned}
&{\mathcal K}_\alpha(x,a)=\log ({\mathcal F}_\alpha(x-a))-x,\\
&{\mathcal D}_\alpha(x,y,a,b)=(x-y)\Bigl[B\Bigl({\mathcal K}_\alpha(x,a)-{\mathcal K}_\alpha(y,b)\Bigr){\mathcal F}_\alpha(x-a)-B\Bigl({\mathcal K}_\alpha(y,b)-{\mathcal K}_\alpha(x,a)\Bigr){\mathcal F}_\alpha(y-b)
\Bigr].
\end{aligned}
$$
Then, for all $\overline{\Phi},\overline{\Psi}\in\R$, the function $ \Upsilon_{\overline{\Phi},\overline{\Psi}}:\R\to \R$ defined by
$$
\Upsilon_{\overline{\Phi},\overline{\Psi}}(x)=\inf \left\{{\mathcal D}_\alpha(x,y,a,b); \ -\overline{\Phi}\leq y\leq \overline{\Phi}, -\overline{\Psi}\leq a,b\leq \overline{\Psi}\right\}
$$
verifies
$$
\begin{aligned}
\lim_{x\to -\infty} \Upsilon_{\overline{\Phi},\overline{\Psi}}(x)=+\infty \quad \text{and} \quad \lim_{x\to +\infty} \Upsilon_{\overline{\Phi},\overline{\Psi}}(x)=+\infty.
\end{aligned}
$$
\end{lemma}
\begin{proof}
Let us first remark that the function  ${\mathcal K}_\alpha$ is non-increasing with respect to its both variables $x$ and $a$ and that the Bernoulli function $B$ is also non-increasing on $\R$. We assume that $\overline{\Phi}$ and $\overline{\Psi}$ are given.
The regularity of the functions ensure that there exist positive constants $\lambda, \underline{\mu}, \overline{\mu} $ 
 such that, for $y\in [-\overline{\Phi},\overline{\Phi}]$ and $a,b\in [-\overline{\Psi}, \overline{\Psi}]$, we have
\begin{align*}
-\lambda \leq {\mathcal K}_\alpha(y,b)\leq \lambda
~
\quad \text{and} \quad
~
\underline{\mu}\leq {\mathcal F}_\alpha(y-b)\leq \overline{\mu},
~
\quad \text{for all} \; -\overline{\Phi}\leq y\leq \overline{\Phi},\  -\overline{\Psi}\leq b\leq \overline{\Psi}.
\end{align*}
This implies the following inequalities, for $x\in\R$, $y\in [-\overline{\Phi},\overline{\Phi}]$ and $a,b\in [-\overline{\Psi}, \overline{\Psi}]$,
\begin{eqnarray*}
\hphantom{-} B\Bigl({\mathcal K}_\alpha(x,-\overline{\Psi})+\lambda\Bigr)\leq& \!\!\! \hphantom{-} B\Bigl({\mathcal K}_\alpha(x,a)-{\mathcal K}_\alpha(y,b)\Bigr) \!\!\!&\leq \hphantom{-} B\Bigl({\mathcal K}_\alpha(x,\overline{\Psi})-\lambda\Bigr),\\
~
-B\Bigl(-\lambda-{\mathcal K}_\alpha(x, -\overline{\Psi})\Bigr)\leq&  \!\!\! -B\Bigl({\mathcal K}_\alpha(y,b)-{\mathcal K}_\alpha(x,a)\Bigr) \!\!\!&\leq -B\Bigl(\lambda-{\mathcal K}_\alpha(x, \overline{\Psi})\Bigr),
\end{eqnarray*}
yielding 
\begin{multline}\label{ineq:D}
B\Bigl({\mathcal K}_\alpha(x,-\overline{\Psi})+\lambda\Bigr){\mathcal F}_\alpha(x- \overline{\Psi})-B\Bigl(-\lambda-{\mathcal K}_\alpha(x, -\overline{\Psi})\Bigr){\overline \mu}
~
\leq
~
\frac{{\mathcal D}_\alpha(x,y,a,b)}{x-y}
~
\leq \\
~
B\Bigl({\mathcal K}_\alpha(x,\overline{\Psi})-\lambda\Bigr){\mathcal F}_\alpha(x+ \overline{\Psi})
-B\Bigl(\lambda-{\mathcal K}_\alpha(x, \overline{\Psi})\Bigr){ \underline{\mu}}.
\end{multline} 
Let us first consider that $x\leq -\overline{\Phi}$. Then, we deduce from \eqref{ineq:D}, that
$$
\frac{{\mathcal D}_\alpha(x,y,a,b)}{x-y}
\leq
B\Bigl({\mathcal K}_\alpha(-\overline{\Phi},\overline{\Psi})-\lambda\Bigr){\mathcal F}_\alpha(x+ \overline{\Psi})
-B\Bigl(\lambda-{\mathcal K}_\alpha(-\overline{\Phi}, \overline{\Psi})\Bigr){ \underline{\mu}}
$$
But, due to \eqref{hyp:statistics-n-p}, $\displaystyle\lim_{x\to -\infty} {\mathcal F}_\alpha(x+ \overline{\Psi})=0$, which implies that, for $-x$ large enough, the right-hand-side of the last inequality is negative. Therefore, for such an $x$ with $x \leq y$ we have
$$
{\mathcal D}_\alpha(x,y,a,b) \geq (x-y)\Bigl[ B\Bigl({\mathcal K}_\alpha(-\overline{\Phi},\overline{\Psi})-\lambda\Bigr){\mathcal F}_\alpha(x+ \overline{\Psi})
-B\Bigl(\lambda-{\mathcal K}_\alpha(-\overline{\Phi}, \overline{\Psi})\Bigr){ \underline{\mu}}\Bigr]
$$
and, taking the infimum in $y\in [-\overline{\Phi},\overline{\Phi}]$, we obtain
$$
\Upsilon_{\overline{\Phi},\overline{\Psi}}(x)\geq (x+\overline{\Phi}) \Bigl[B\Bigl({\mathcal K}_\alpha(-\overline{\Phi},\overline{\Psi})-\lambda\Bigr){\mathcal F}_\alpha(x+ \overline{\Psi})
-B\Bigl(\lambda-{\mathcal K}_\alpha(-\overline{\Phi}, \overline{\Psi})\Bigr){ \underline{\mu}}\Bigr].
$$
As the first product in the right-hand-side tends to 0, while the second one tends to $+\infty$, we deduce that
$$
\lim_{x\to -\infty} \Upsilon_{\overline{\Phi},\overline{\Psi}}(x)=+\infty.
$$
We may now consider that $x\geq \overline{\Phi}$. From \eqref{ineq:D}, we deduce that
$$
\frac{{\mathcal D}_\alpha(x,y,a,b)}{x-y}
\geq
\Bigl(B\bigl({\mathcal K}_\alpha(\overline{\Phi},-\overline{\Psi})+\lambda\bigr){\mathcal F}_\alpha(x- \overline{\Psi})
-B\bigl(-\lambda-{\mathcal K}_\alpha(\overline{\Phi}, -\overline{\Psi})\bigr){ \overline \mu}\Bigr).
$$
For $x$ sufficiently large, the right-hand-side of the last inequality is positive and
$$
\Upsilon_{\overline{\Phi},\overline{\Psi}}(x)\geq (x-\overline{\Phi}) \Bigl(B\bigl({\mathcal K}_\alpha(\overline{\Phi},-\overline{\Psi})+\lambda\bigr){\mathcal F}_\alpha(x- \overline{\Psi})
-B\bigl(-\lambda-{\mathcal K}_\alpha(\overline{\Phi}, -\overline{\Psi})\bigr){ \overline \mu}\Bigr).
$$
Therefore, we get
$$
\lim_{x\to +\infty} \Upsilon_{\overline{\Phi},\overline{\Psi}}(x)=+\infty.
$$
\end{proof}


\bibliographystyle{abbrv}
\bibliography{literature}

\begin{thebibliography}{10}

\bibitem{Abdel2021b}
D.~Abdel, P.~Farrell, and J.~Fuhrmann.
\newblock Assessing the quality of the excess chemical potential flux scheme
  for degenerate semiconductor device simulation.
\newblock {\em Optical and Quantum Electronics}, 53(163), 2021.

\bibitem{ChargeTransport}
D.~Abdel, P.~Farrell, and J.~Fuhrmann.
\newblock {ChargeTransport.jl}: Simulating charge transport in semiconductors.
\newblock https://github.com/PatricioFarrell/ChargeTransport.jl, 2022.

\bibitem{Abdel2021}
D.~Abdel, P.~Vágner, J.~Fuhrmann, and P.~Farrell.
\newblock Modelling charge transport in perovskite solar cells: Potential-based
  and limiting ion depletion.
\newblock {\em Electrochimica Acta}, 390:138696, 2021.

\bibitem{Albinus2002}
G.~Albinus, H.~Gajewski, and R.~Hünlich.
\newblock Thermodynamic design of energy models of semiconductor devices.
\newblock {\em Nonlinearity}, 15(2):367--383, 2002.

\bibitem{bessemoulin2012finite}
M.~Bessemoulin-Chatard.
\newblock A finite volume scheme for convection--diffusion equations with
  nonlinear diffusion derived from the scharfetter--gummel scheme.
\newblock {\em Numerische Mathematik}, 121(4):637--670, 2012.

\bibitem{BessemoulinChatard2017}
M.~Bessemoulin-Chatard and C.~Chainais-Hillairet.
\newblock Exponential decay of a finite volume scheme to the thermal
  equilibrium for drift–diffusion systems.
\newblock {\em Journal of Numerical Mathematics}, 25(3):147--168, 2017.

\bibitem{bessemoulin14}
M.~Bessemoulin-Chatard, C.~Chainais-Hillairet, and M.-H. Vignal.
\newblock Study of a finite volume scheme for the drift-diffusion system.
  {Asymptotic} behavior in the quasi-neutral limit.
\newblock {\em SIAM J. Numer. Anal.}, 52(4):1666--1691, 2014.

\bibitem{biler2000long}
P.~Biler and J.~Dolbeault.
\newblock Long time behavior of solutions to nernst-planck and debye-h{\"u}ckel
  drift-diffusion systems.
\newblock {\em Annales Henri Poincar{\'e}}, 1:461--472, 2000.

\bibitem{brezzi2005discretization}
F.~Brezzi, L.~Marini, S.~Micheletti, P.~Pietra, R.~Sacco, and S.~Wang.
\newblock Discretization of semiconductor device problems (i).
\newblock {\em Handbook of numerical analysis}, 13:317--441, 2005.

\bibitem{Calado2016}
P.~Calado, A.~Telford, D.~Bryant, X.~Li, J.~Nelson, B.~O'Regan, and P.~R.~F.
  Barnes.
\newblock Evidence for ion migration in hybrid perovskite solar cells with
  minimal hysteresis.
\newblock {\em Nature Communications}, 7, 2016.

\bibitem{Cances2020}
C.~Canc{\`e}s, C.~Chainais-Hillairet, J.~Fuhrmann, and B.~Gaudeul.
\newblock {A numerical-analysis-focused comparison of several finite volume
  schemes for a unipolar degenerate drift-diffusion model}.
\newblock {\em IMA Journal of Numerical Analysis}, 41(1):271--314, 07 2020.

\bibitem{ChainaisHillairet2019}
C.~Chainais-Hillairet and M.~Herda.
\newblock {Large-time behaviour of a family of finite volume schemes for
  boundary-driven convection–diffusion equations}.
\newblock {\em IMA Journal of Numerical Analysis}, 40(4):2473--2504, 2019.

\bibitem{Courtier2019}
N.~E. Courtier.
\newblock {\em Modelling ion migration and charge carrier transport in planar
  perovskite solar cells}.
\newblock PhD thesis, University of Southampton, 2019.

\bibitem{Courtier2019a}
N.~E. Courtier, J.~M. Cave, A.~B. Walker, G.~Richardson, and J.~M. Foster.
\newblock Ionmonger: a free and fast planar perovskite solar cell simulator
  with coupled ion vacancy and charge carrier dynamics.
\newblock {\em Journal of Computational Electronics}, 18:1435--1449, 2019.

\bibitem{Courtier2018}
N.~E. Courtier, G.~Richardson, and J.~M. Foster.
\newblock A fast and robust numerical scheme for solving models of charge
  carrier transport and ion vacancy motion in perovskite solar cells.
\newblock {\em Applied Mathematical Modelling}, 2018.

\bibitem{Evans:10}
L.~C. Evans.
\newblock {\em {Partial Differential Equations: Second Edition}}, volume~19 of
  {\em Graduate Studies in Mathematics}.
\newblock American Mathematical Society, Providence, R.I., 2010.

\bibitem{Eymard2000}
R.~Eymard, T.~Gallou\"{e}t, and R.~Herbin.
\newblock Finite volume methods.
\newblock In {\em Handbook of numerical analysis, {V}ol. {VII}}, pages
  713--1020. North-Holland, Amsterdam, 2000.

\bibitem{Farrell2017}
P.~Farrell, D.~H. Doan, M.~Kantner, J.~Fuhrmann, T.~Koprucki, and N.~Rotundo.
\newblock Drift-diffusion models.
\newblock In {\em Optoelectronic Device Modeling and Simulation: Fundamentals,
  Materials, Nanostructures, LEDs, and Amplifiers}, pages 733--771. CRC Press
  Taylor \& Francis Group, 2017.

\bibitem{farrell2018comparison}
P.~Farrell, M.~Patriarca, J.~Fuhrmann, and T.~Koprucki.
\newblock Comparison of thermodynamically consistent charge carrier flux
  discretizations for fermi--dirac and gauss--fermi statistics.
\newblock {\em Optical and Quantum Electronics}, 50(2):1--10, 2018.

\bibitem{gajewski85}
H.~Gajewski.
\newblock On existence, uniqueness and asymptotic behavior of solutions of the
  basic equations for carrier transport in semiconductors.
\newblock {\em Z. Angew. Math. Mech.}, 65:101--108, 1985.

\bibitem{gajewski86}
H.~Gajewski and K.~Gr{\"o}ger.
\newblock On the basic equations for carrier transport in semiconductors.
\newblock {\em J. Math. Anal. Appl.}, 113:12--35, 1986.

\bibitem{gajewski1989semiconductor}
H.~Gajewski and K.~Gr{\"o}ger.
\newblock Semiconductor equations for variable mobilities based on boltzmann
  statistics or fermi-dirac statistics.
\newblock {\em Mathematische Nachrichten}, 140(1):7--36, 1989.

\bibitem{Gaudeul2021}
B.~Gaudeul and J.~Fuhrmann.
\newblock {Entropy and convergence analysis for two finite volume schemes for a
  Nernst-Planck-Poisson system with ion volume constraints}.
\newblock {\em Numerische Mathematik}, 151:99--149, 2022.

\bibitem{glitzky2011uniform}
A.~Glitzky.
\newblock Uniform exponential decay of the free energy for voronoi finite
  volume discretized reaction-diffusion systems.
\newblock {\em Mathematische Nachrichten}, 284(17-18):2159--2174, 2011.

\bibitem{jungel1995numerical}
A.~J{\"u}ngel.
\newblock Numerical approximation of a drift-diffusion model for semiconductors
  with nonlinear diffusion.
\newblock {\em ZAMM-Journal of Applied Mathematics and Mechanics/Zeitschrift
  f{\"u}r Angewandte Mathematik und Mechanik}, 75(10):783--799, 1995.

\bibitem{jungel2016entropy}
A.~J{\"u}ngel.
\newblock {\em Entropy methods for diffusive partial differential equations},
  volume 804.
\newblock Springer, 2016.

\bibitem{jungel01}
A.~J{\"u}ngel and Y.-J. Peng.
\newblock A hierarchy of hydrodynamic models for plasmas. {Quasi}-neutral
  limits in the drift-diffusion equations.
\newblock {\em Asymptotic Anal.}, 28(1):49--73, 2001.

\bibitem{Kantner2020}
M.~Kantner.
\newblock Generalized {S}charfetter--{G}ummel schemes for electro-thermal
  transport in degenerate semiconductors using the {K}elvin formula for the
  {S}eebeck coefficient.
\newblock {\em Journal of Computational Physics}, 402:109091, 2020.

\bibitem{Kantner2020b}
M.~Kantner and T.~Koprucki.
\newblock {Non-isothermal Scharfetter--Gummel Scheme for Electro-Thermal
  Transport Simulation in Degenerate Semiconductors}.
\newblock In {\em Finite Volumes for Complex Applications IX - Methods,
  Theoretical Aspects, Examples}, pages 173--182. Springer International
  Publishing, 2020.

\bibitem{kim2020high}
J.~Y. Kim, J.-W. Lee, H.~S. Jung, H.~Shin, and N.-G. Park.
\newblock High-efficiency perovskite solar cells.
\newblock {\em Chemical Reviews}, 120(15):7867--7918, 2020.

\bibitem{liu2021positivity}
C.~Liu, C.~Wang, S.~Wise, X.~Yue, and S.~Zhou.
\newblock A positivity-preserving, energy stable and convergent numerical
  scheme for the poisson-nernst-planck system.
\newblock {\em Mathematics of Computation}, 90(331):2071--2106, 2021.

\bibitem{markowich1985stationary}
P.~A. Markowich.
\newblock {\em The stationary semiconductor device equations}.
\newblock Springer Science \& Business Media, 1985.

\bibitem{markowich2012semiconductor}
P.~A. Markowich, C.~A. Ringhofer, and C.~Schmeiser.
\newblock {\em Semiconductor equations}.
\newblock Springer Science \& Business Media, 2012.

\bibitem{moatti2022structure}
J.~Moatti.
\newblock A structure preserving hybrid finite volume scheme for semi-conductor
  models with magnetic field on general meshes.
\newblock {\em arXiv preprint arXiv:2207.02567}, 2022.

\bibitem{mock1983analysis}
M.~S. Mock.
\newblock {\em Analysis of mathematical models of semiconductor devices},
  volume~3.
\newblock Boole Press, 1983.

\bibitem{SG69}
D.~Scharfetter and H.~Gummel.
\newblock Large-signal analysis of a silicon read diode oscillator.
\newblock {\em IEEE Transactions on electron devices}, 16(1):64--77, 1969.

\bibitem{Sze2006}
S.~M. Sze and K.~K. Ng.
\newblock {\em Physics of Semiconductor Devices}.
\newblock Wiley, 2006.

\bibitem{Tessler2020}
N.~Tessler and Y.~Vaynzof.
\newblock Insights from device modeling of perovskite solar cells.
\newblock {\em ACS Energy Letters}, 5(4):1260--1270, 04 2020.

\bibitem{van1950theory}
W.~Van~Roosbroeck.
\newblock Theory of the flow of electrons and holes in germanium and other
  semiconductors.
\newblock {\em The Bell System Technical Journal}, 29(4):560--607, 1950.

\bibitem{Yu1988}
Z.~Yu and R.~Dutton.
\newblock {SEDAN III} -- {A} one-dimensional device simulator.
\newblock \url{www-tcad.stanford.edu/tcad/programs/sedan3.html}, 1988.

\end{thebibliography}

\end{document}